\documentclass[11pt,leqno]{amsart}
\usepackage{amssymb}
\usepackage{amsfonts}
\usepackage{amsfonts}
\usepackage{amsfonts}
\usepackage{amsfonts}
\usepackage{mathrsfs}
\usepackage{amssymb}
\usepackage{txfonts}
\usepackage{amsfonts}
\usepackage{amssymb,amsmath}

\newcommand{\equ}[1]{(\ref{#1})}
\long\def\symbolfootnote[#1]#2{\begingroup%
\def\thefootnote{\fnsymbol{footnote}}\footnote[#1]{#2}\endgroup}

\newcommand\R{{\mathbb R}}

\newtheorem{lemma}{Lemma}[section]
\newtheorem{proposition}{Proposition}[section]
\newtheorem{theorem}{Theorem}[section]

\let\Section=\section
\def\section{\setcounter{equation}{0}\Section}
\sf

\begin{document}
\title[Nondegeneracy of Nonradial Nodal Solutions to Yamabe
Problem]{Nondegeneracy of Nonradial  Nodal Solutions to Yamabe
 Problem}

\author{Monica Musso}
\address{\noindent M. Musso - Departamento de Matem\'atica,
Pontificia Universidad Catolica de Chile, Avda. Vicu\~na Mackenna
4860, Macul, Chile}\email{mmusso@mat.puc.cl}

\author{Juncheng Wei}
\address{\noindent J. Wei -Department of Mathematics, University of British Columbia, Vancouver, BC V6T1Z2, Canada }\email{jcwei@math.ubc.ca}

\thanks{ The research of the first  author
has been partly supported by Fondecyt Grant 1120151. The research of the third author is partially supported by NSERC of Canada.}

\maketitle

\vskip 0.2cm \arraycolsep1.5pt
\newtheorem{Lemma}{Lemma}[section]
\newtheorem{Theorem}{Theorem}[section]
\newtheorem{Definition}{Definition}[section]
\newtheorem{Proposition}{Proposition}[section]
\newtheorem{Remark}{Remark}[section]
\newtheorem{Corollary}{Corollary}[section]

\maketitle

\noindent {\bf Abstract}: We provide the first example of  a sequence of {\em nondegenerate}, in the sense of  Duyckaerts-Kenig-Merle \cite{DKM},  nodal nonradial solutions    to the  critical Yamabe problem
$$ -\Delta Q= |Q|^{\frac{2}{n-2}} Q, \ \ Q \in {\mathcal D}^{1,2} (\R^n). $$

\vspace{3mm}

\maketitle

\section{Introduction}\label{intro}

In this paper we consider the critical Yamabe problem
\begin{equation}
\label{1m}
-\Delta u = \frac{ n(n-2)}{4} |u|^{\frac{4}{n-2}} u, \quad u \in {\mathcal D}^{1,2} (\R^n)
\end{equation}
where $n\geq 3$ and $ {\mathcal D}^{1,2} (\R^n)$ is the completion of $C_0^\infty (\R^n)$ under the norm $ \sqrt{ \int_{\R^n} |\nabla u|^2} $.

If $u>0$ Problem (\ref{1m}) is the conformally invariant Yamabe problem. For sign-changing $u$ Problem (\ref{1m}) corresponds to  the steady state of the energy-critical focusing nonlinear wave equation
\begin{equation}
\label{2m}
\partial_t^2 u-\Delta u- |u|^{\frac{4}{n-2}} u=0, \ (t,x)\in \R \times \R^n.
\end{equation}
These are classical problems that have attracted the attention of several researchers in order to understand the structure and properties of the solutions to Problem  (\ref{1m}).

\medskip
\noindent
Denote the set of non-zero finite energy solutions to Problem  (\ref{1m}) by
\begin{equation}
\label{3m}
\Sigma:= \left\{ Q \in {\mathcal D}^{1,2} (\R^n) \backslash \{0\}: \ -\Delta Q= \frac{n(n-2)}{4} |Q|^{\frac{4}{n-2}} Q\right \}.
\end{equation}
This set has been completely characterized in the class of positive solutions to Problem (\ref{1m})  by the classical work of Caffarelli-Gidas-Spruck \cite{CGS} (see also \cite{aubin, obata, talenti}): all positive solutions to (\ref{1m}) are radially symmetric around some point $a \in \R^n$ and are of the form
\begin{equation}
W_{\lambda, a}(x)= \Big(\frac{\lambda}{\lambda^2 +|x-a|^2}\Big)^{\frac{n-2}{2}}, \ \lambda>0.
\end{equation}
Much less is known in the sign-changing case. A direct application of Pohozaev's identity gives that all sign-changing solutions  to Problem \equ{1m} are non-radial. The existence of elements of $\Sigma$  that are nonradial sign-changing, and with arbitrary large energy was first proved by Ding \cite{D} using  Ljusternik-Schnirelman  category theory.  Indeed,
 via stereographic projection to $S^n$  Problem \equ{1m} becomes
$$
\Delta_{S^n}  v +  \frac{ n(n-2)}{4}  ( |v|^{\frac 4{n-2}} v - v) = 0 \quad \hbox{in } S^n,
$$
(see for instance  \cite{Sch-Yau}, \cite{Heb}) and
  Ding  showed the existence of infinitely many critical points to the associated energy functional within  functions of the form
$$  v(x)= v(|x_1|,|x_2|), \quad x= (x_1, x_2) \in  S^n \subset \R^{n+1}= \R^k\times \R^{n+1-k}, \quad k\ge 2, $$
where compactness of critical Sobolev's embedding holds, for any $n\ge 3$. No other qualitative properties are known for the corresponding solutions. Recently more explicit constructions of sign changing solutions to Problem \equ{1m} have been obtained by del Pino-Musso-Pacard-Pistoia \cite{dmpp1, dmpp2}.  However so far only existence is available, and there are no rigidity results on these solutions.

\medskip
\noindent
The main purpose of this paper is to prove that these solutions are rigid, up to the transformations of the equation. In other words, these solutions are {\em nondegenerate}, in the sense of the definition introduced by Duyckaerts-Kenig-Merle in \cite{DKM}.
Following \cite{DKM}, we first find out all possible invariances of the equation (\ref{1m}). Equation (\ref{1m}) is invariant under the following four transformations:

\medskip
\noindent
(1) (translation): If $Q\in \Sigma$ then $  Q(x+a) \in \Sigma, \forall a \in \R^n$;

\medskip
\noindent
(2) (dilation): If $Q \in \Sigma$ then $ \lambda^{\frac{n-2}{2}} Q(\lambda x) \in \Sigma, \forall \lambda >0$;

\medskip
\noindent
(3) (orthogonal transformation): If $Q\in \Sigma$ then $ Q(Px) \in \Sigma$ where $P \in {\mathcal O}_n$ and ${\mathcal O}_n$ is the classical orthogonal group;

\medskip
\noindent
(4) (Kelvin transformation): If $Q \in \Sigma$ then $|x|^{2-N} Q (\frac{x}{|x|^2}) \in \Sigma$.

If we denote by ${\mathcal M}$  the group of isometries of ${\mathcal D}^{1,2} (\R^n)$ generated by the previous four transformations, a result of Duyckaerts-Kenig-Merle [Lemma 3.8,\cite{DKM}] states that ${\mathcal M}$ generates an $N-$parameter family of transformations in a neighborhood of the  identity, where the dimension $N$ is given by
\begin{equation}
\label{4m}
N= 2n+1 +\frac{n(n-1)}{2}.
\end{equation}
In other words, if $Q \in \Sigma$ we denote
$$ L_{Q}:=-\Delta - \frac{n(n+2)}{4} |Q|^{\frac{4}{n-2}} $$
 the linearized operator around $Q$. Define the null space of $L_Q$
\begin{equation}
{\mathcal Z}_Q=\left \{ f \in {\mathcal D}^{1,2} (\R^n): \ L_{Q} f=0 \right\}
\end{equation}
The elements in ${\mathcal Z}_Q$ generated by the family of transformations ${\mathcal M}$ define the following vector space
\begin{equation}
\label{Zq}
\tilde{{\mathcal Z}}_{Q}= \mbox{span} \ \left\{ \begin{array}{l}
(2-n) x_j Q +|x|^2 \partial_{x_j} Q- 2 x_j x \cdot \nabla Q, \quad  \partial_{x_j} Q, \quad 1\leq j\leq n, \\
\\
(x_j \partial_{x_k} - x_k \partial_{x_j})Q, \quad 1\leq j <k\leq n, \quad \frac{n-2}{2} Q + x \cdot Q \end{array}
\right\}.
\end{equation}
Observe that the dimension of $\tilde{{\mathcal Z}}_Q$ is at most $N$, but in principle it could be strictly less than $N$. For example in the case of the positive solutions $Q= W$, it turns out that  the dimension of $\tilde{{\mathcal Z}}_Q$ is $n+1$ as a consequence of being $Q$  radially symmetric. Indeed, it is known that
\begin{equation}
\tilde{{\mathcal Z}}_{W}=\left \{ \frac{n-2}{2} W+ x \cdot \nabla W, \quad \partial_{x_j} W, \quad 1\leq j\leq n \right\}.
\end{equation}

Duyckaerts-Kenig-Merle \cite{DKM} introduced the following definition of nondegeneracy for a solution of Problem \equ{1m}: $Q\in \Sigma$ is said to be {\em nondegenerate} if
\begin{equation}
\label{5m}
{\mathcal Z}_Q= \tilde{{\mathcal Z}}_{Q}.
\end{equation}

So far the only nondegeneracy example of $Q\in \Sigma$ is the positive solution $W$. The proof of this fact relies heavily on the radial symmetry of $W$ and  it is straightforward: In fact since $Q=W$ is radially symmetric (around some point) one can decompose the linearized operator into Fourier modes, getting \equ{5m} as consequence of a simple ode analysis. See also \cite{rey}. In the case of nodal (nonradial) solutions this strategy no longer works out.

The knowledge of nondegeneracy is a crucial ingredient to show the soliton resolution for a solution to the
energy-critical wave equation (\ref{2m})   with the {\it compactness property} obtained by Kenig and Merle in \cite{KM1, KM2}. If the dimension $n$ is $3,4$ or $5$, and  under the above nondegeneracy assumption,  they prove that any non zero such solution is a  sum of
stationary solutions and solitary waves that are Lorentz transforms of the former.

\medskip
The main result of this paper can be stated as follows:

\medskip
\noindent
{\em {\bf Main Result:} Assume that $n \geq 4$. Then there exists a sequence of nodal solutions to (\ref{1m}), with arbitrary large energy, such that they are nondegenerate in the sense of (\ref{5m}).
}
\medskip

Now let us be more precise.

Let
\begin{equation}
\label{nonlinearity}
f(t) = \gamma \,  |t|^{p-1} \,  t , \quad {\mbox {for}} \quad t \in \R, \quad {\mbox {and}} \quad p={n+2 \over n-2}.
\end{equation}
The constant $\gamma >0$ is chosen for normalization purposes to be
$$
\gamma = {n (n-2) \over 4}.
$$
In  \cite{dmpp1}, del Pino, Musso, Pacard and Pistoia  showed that Problem
\begin{equation}
\label{eq}
\Delta u + f( u ) = 0 \quad {\mbox {in}} \quad \R^n,
\end{equation}
admits a sequence of entire non radial sign changing solutions with finite energy.

\noindent
To give a first description of these solutions, let us introduce some notations.
Fix an integer $k$. For any integer $l=1, \ldots , k$, we define angles $\theta_l$ and vectors $\textsf{n}_l$, $\textsf{t}_l$ by
\begin{equation}
\label{defanglespoints}
\theta_l = {2\pi \over k} \, (l-1), \quad \textsf{n}_l = (\cos \theta_l , \sin \theta_l, \textsf{0} ), \quad \textsf{t}_l = (-\sin \theta_l , \cos \theta_l , \textsf{0}).
\end{equation}
Here $\textsf{0}$ stands for the zero vector in $\R^{n-2}$. Notice that  $\theta_1=0$, $\textsf{n}_1 = (1, 0, \textsf{0})$, and  $\textsf{t}_1 = (0, 1,\textsf{0})$.

\medskip
\noindent
In  \cite{dmpp1}  it was proved that there exists $k_0$ such that for all integer $k>k_0$
there exists a solution $u_k$ to \equ{eq} that can be described as follows
\begin{equation}
\label{sol}
u_k(x) =U_* (x) +\tilde  \phi (x).
\end{equation}
where
\begin{equation}
\label{defU*}
U_* (x) =  U(x) - \sum_{j=1}^k U_j (x),
\end{equation}
while $\tilde \phi$ is smaller than $U_*$.
The functions $U$ and  $U_j$ are positive solutions to \equ{eq}, respectively defined as
\begin{equation}
\label{basic}
U(x) = \left( {2 \over 1+ |x|^2} \right)^{n-2 \over 2}, \quad U_j (x) = \mu_k^{-{n-2 \over 2}} U(\mu_k^{-1} (x-\xi_j )).
\end{equation}
For any integer $k$ large, the parameters $\mu_k  >0$ and the $k$ points $\xi_l$, $l=1, \ldots , k$
are given by
\begin{equation} \label{parameters}
\left[  \sum_{l > 1}^k {1 \over (1-\cos \theta_l )^{n-2 \over 2}} \right] \, \mu_k^{n-2 \over 2} =  \left(1+ O ({1\over k}) \right) , \quad {\mbox {for}} \quad k \to \infty
\end{equation}
in particular $ \mu_k \sim k^{-2}$ as $k \to \infty$, and
\begin{equation} \label{parameters1}
 \xi_l = \sqrt{1-\mu^2} \, ( \textsf{n}_l , 0).
\end{equation}
The functions $U$, $U_j$ and $U_*$ are invariant under rotation of angle ${2\pi \over k}$ in the $x_1 , x_2$ plane, namely
\begin{equation}
\label{sim00}
U( e^{2\pi \over k} \bar x , x' ) = U(\bar x , x' ), \quad \bar x= (x_1, x_3) , \quad x'= (x_3, \ldots , x_n ).
\end{equation}
They are even in the $x_j$-coordinates, for any $j=2, \ldots , n$
\begin{equation}
 U (x_1,\ldots,x_j, \ldots, x_{n}  ) =  U (x_1,\ldots,-x_j, \ldots, x_{n}  ),\quad  j=2,\ldots,n
\label{sim22}\end{equation}
and they respect invariance under Kelvin's transform:
\begin{equation}
 U (x )\  = \  |x|^{2-n} U (|x|^{-2} x)\ .
\label{sim33}\end{equation}

In \equ{sol}  the function $\tilde \phi$ is a small function when compared with $U_*$. We will further describe the function $u$, and in particular the function $\tilde \phi$ in Section \ref{des}.
Let us just mention that
$\tilde \phi$ satisfies all the symmetry properties \equ{sim00}, \equ{sim22} and \equ{sim33}.

\medskip
 Recall that Problem \equ{eq} is invariant under the four transformations mentioned before: translation,  dilation, rotation and Kelvin transformation.
These invariances will be reflected in the element of the kernel of the linear operator
\begin{equation}
\label{linearop}
L(\varphi ) :=\Delta \varphi + f' (u_k) \varphi =  \Delta \varphi + p \gamma |u_k |^{p-2} u \, \varphi
\end{equation}
which is the linearized equation associated to \equ{eq} arount $u_k$.

\medskip

From now on, for simplicity we will drop the label $k$ in $u_k$, so that $u$ will denote the solution to Problem \equ{eq} described in \equ{sol}.

\medskip

Let us introduce the following set of $3n$ functions
\begin{equation}
\label{capitalzeta0}
z_0 (x) = {n-2 \over 2} u(x) + \nabla u (x) \cdot x ,
\end{equation}
\begin{equation}
\label{capitalzetaj}
z_\alpha (x)  = {\partial \over \partial x_\alpha  } u(x) , \quad {\mbox {for}} \quad \alpha=1, \ldots , n,
\end{equation}
and
\begin{equation}
\label{capitalzeta2}
z_{n+1} (x) = -x_2  {\partial \over \partial x_1 } u(x) + x_1  {\partial \over \partial x_2  } u(x)
\end{equation}
where $u$ is the solution to \equ{eq} described in \equ{sol}.
Observe that $z_{n+1}$ is given by
$$
 z_{n+1} (x) =   {\partial \over \partial \theta} [u(R_\theta x) ]_{|\theta=0}
$$
where $R_\theta$ is the rotation in the $x_1, x_2$ plane of angle $\theta$.
Furthermore,
\begin{equation}
\label{chico1}
z_{n+2} (x) = -2 x_1 z_0 (x) + |x|^2 z_1 (x) , \quad z_{n+3} (x) =  -2 x_2 z_0 (x) + |x|^2 z_2 (x)
\end{equation}
for $l=3, \ldots , n$
\begin{equation}
\label{chico2}
z_{n+l+1} (x) = -x_l z_1 (x) + x_1 z_l (x), \quad u_{2n+l-1} (x) =  -x_l z_2 (x) + x_2 z_l (x).
\end{equation}
The functions defined in \equ{chico1} are related to the invariance of Problem \equ{eq} under Kelvin transformation, while the functions defined in \equ{chico2} are related to the invariance under rotation in
the $(x_1 , x_l)$ plane and in the $(x_2 , x_l)$ plane respectively.

\bigskip

The invariance of Problem \equ{eq} under scaling, translation, rotation and Kelvin transformation gives that
the set $\tilde{{\mathcal Z}}_Q$ (introduced in (\ref{Zq})) associated to
the linear operator $L$ introduced in \equ{linearop} has dimension at least $3n$, since
\begin{equation}
\label{r1}
L(z_\alpha ) = 0 , \quad \alpha=0, \ldots , 3n-1.
\end{equation}
We shall show that these functions are the {\em only} bounded elements of the kernel of the operator $L$.
In other words, the sign changing solutions \equ{sol} to Problem \equ{eq} constructed in \cite{dmpp1} are non degenerate in the sense of Duyckaerts-Kenig-Merle \cite{DKM}.

\medskip
\noindent
To state our result, we introduce the following function: For any $i$ integer, we define
$$
P_i (x) = \sum_{l=1}^\infty {\cos ( l \, x ) \over l^i }
\quad
{\mbox {and}} \quad
Q_i (x) = \sum_{l=1}^\infty {\sin ( l \, x ) \over l^i }.
$$
Up to a normalization constant, when $n$ is even, $P_n$ and $Q_n$ are related to  the Fourier series of the  Bernoulli polynomial $B_n (x)$, and when $n$ is odd $P_n$  and $Q_n$ are related to the Fourier series of  the Euler polynomial $E_n (x)$. We refer to \cite{as} for further details.

We now define
\begin{equation}
\label{defg}
g(x) = \sum_{j=1}^\infty {1-\cos (jx) \over j^n} , \quad 0 \leq x \leq \pi.
\end{equation}
which can be rewritten as
$$
g(x) = P_n (0) - P_n (x).
$$
Observe that
$$
 \quad g'(x) = Q_{n-1} (x), \quad g ''(x) = P_{n-2} (x).
$$

\begin{theorem}\label{teo} Assume that $n\geq 4$ and that
\begin{equation}
\label{condig}
g'' (x) < {n-2 \over n-1} {(g'(x) )^2 \over g(x) } \quad \forall x \in (0,\pi ).
\end{equation}
Then all  bounded solutions to the equation
$$
L(\varphi ) = 0
$$
are a linear combination of the functions $z_\alpha (x) $, for $\alpha=0, \ldots, 3 n-1$.
\end{theorem}

When $n=4$, let us check the condition (\ref{condig}): let $ x=2\pi t, t \in (0, \frac{1}{2})$. Using the explicit formula for the Bernoulli polynomial $B_4$ we find that
\begin{equation}
g (t)= t^2 (1-t)^2
\end{equation}
and hence (\ref{condig}) is reduced to showing
\begin{equation}
12 t^2 -12 t +2 < \frac{8}{3} (1+t)^2, \quad  t \in (0, \frac{1}{2})
\end{equation}
which is trivial to verify.

In general we believe that condition (\ref{condig}) should be true for any dimension $n\geq 4$. In fact, we have checked (\ref{condig}) numerically, up to dimension $n \leq 48$. Nevertheless, let us mention that even if (\ref{condig}) fails, our result is still valid for a   {\em  subsequence $u_{k_j}$, $k_j \to +\infty$}, of solutions
\equ{sol} to Problem \equ{eq}. Indeed, also in this case, our proof can still go through by choosing {\em a subsequence $k_j \to +\infty$} in order  to avoid the resonance. Furthermore, let us mention that the  $n=3$ case can also be treated, being the difference with the case $n\geq 4$ in several estimates and  in the definition of the parameters $\mu_k$ in \equ{parameters}, where a $log$-correction term is needed. Indeed, when $n=3$, one has $\mu_k \sim k^{-2} (\ln k)^{-2}$, as $k \to \infty$ (see \cite{dmpp1}).  We leave the $n=3$ case for a forthcoming work.

\medskip
\noindent
We end this section with some remarks.

First: very few are the results on sign-changing solutions to the Yamabe problem. In the critical
exponent case and $n = 3$ the topology of lower energy level sets was analyzed in Bahri-Chanillo \cite{bahri} and Bahri-Xu \cite{BX}. For the construction of sign-changing bubbling solutions we refer to  Hebey-Vaugon \cite{HV}, Robert-Vetois \cite{RV1, RV2}, Vaira \cite{vaira} and the references therein.  We believe that the non-degeneracy property established in Theorem \ref{teo} may be used to obtain new type of constructions for sign changing bubbling solutions.

Second: as far as we know  the kernels due to the Kelvin transform (i.e. $ -2x_j z_0 + |x|^2 z_j$)  were first used by Korevaar-Mazzeo-Pacard-Schoen \cite{KMPS} and  Mazzeo-Pacard (\cite{MP}) in the  construction of isolated singularities for Yamabe problem by using a gluing procedure.
An interesting question is to determine if and how  the non-degenerate sign-changing solutions can used in gluing methods.

\bigskip

\noindent
{\bf Acknowledgements:} The authors express their deep thanks to Professors M. del Pino and  F. Robert for stimulating discussions. We thank Professor C. Kenig for communicating
his unpublished result \cite{DKM}.

\section{Description of the solutions}\label{des}

 In this section we describe the solutions $u_k$ in \equ{sol}, recalling some properties that have already been established in  \cite{dmpp1}, and
adding some further properties that will be useful for later purpose.

In terms of the function $ \tilde \phi$ in the decomposition \equ{sol}, equation \equ{eq} gets re-written as
\begin{equation}
\Delta \tilde  \phi +  p\gamma|U_*|^{p-1}\tilde \phi +   E + \gamma  N(\tilde \phi) = 0
\label{eq1}\end{equation}
where $E$ is defined
by
\begin{equation}
 \label{error}
E= \Delta U_* + f(U_* )
\end{equation}
and
$$
N(\phi) =  |U_* + \phi |^{p-1} (U_* + \phi)  - |U_*  |^{p-1} U_* - p|U_*|^{p-1} \phi.
$$
The function $E$ defined in \equ{error} is the so called Error function. One has a precise control of the size of the function $E$ when measured for instance in the following norm.
Let us fix a number $q$, with $ \frac n2 < q < n$, and consider the weighted $L^q$ norm
\begin{equation}
\|h\|_{**} =  \|  \, (1+ |y|)^{{n+2} - \frac {2n} q} h \|_{L^q(\R^n)}.
\label{nomstar2} \end{equation}

In  \cite{dmpp1} it is proved that there exist an integer $k_0$ and a positive constant $C$ such that for all $k\geq k_0$ the following estimates
hold true
\begin{equation}
\label{rivoli4}
\| E \|_{**} \leq C k^{1-{n\over q}} \quad {\mbox {if}} \quad n\geq 4
\end{equation}

\medskip

To be more precise, we have  estimates for the $\| \cdot \|_{**}$-norm of the error term $E$
first in the {\it exterior region} $\bigcap_{j=1}^k \{ |y-\xi_j| > {\eta \over k}  \} $, and also in the {\it interior regions} $ \{ |y-\xi_j| < {\eta \over k}  \} $, for any $j=1, \ldots , k$. Here $\eta >0$ is a positive and small constant,  independent of $k$.

\medskip

\medskip
\noindent {\it In the exterior region}. We have
$$
\|  \, (1+ |y|)^{{n+2} - \frac {2n} q} E (y)  \|_{L^q(\bigcap_{j=1}^k \{ |y-\xi_j| > {\eta \over k}  \})} \leq C k^{1-{n\over q}}
$$

\medskip
\noindent {\it In the interior regions}.
Now, let $|y-\xi_j| < {\eta \over k} $ for some $j \in \{ 1, \ldots , k\}$ fixed.
It is convenient to measure the error after a change of scale.
Define
$$
\tilde E_j(y) :=  \mu^{\frac {n+2}2} E ( \xi_j + \mu y ) , \quad |y| < \frac \eta {\mu  k}
$$
We have
$$
\|  \, (1+ |y|)^{{n+2} - \frac {2n} q}\tilde E_j (y) \|_{L^q( |y-\xi_j| <{\eta \over \mu k })  } \leq C k^{-{n\over q}}.
$$
We refer the readers to \cite{dmpp1}.

\medskip
\noindent

The function $ \tilde \phi $ in \equ{sol} can be further decomposed. Let us introduce  some cut-off functions $\zeta_j$
to be defined as follows.
Let $\zeta(s)$ be a smooth function such that $\zeta(s) = 1$ for $s<1$ and $\zeta(s)=0$ for $s>2$. We also let $\zeta^-(s) = \zeta(2s)$.
Then we set
$$
 \zeta_j(y) = \left\{ \begin{matrix}   \zeta(\, k\eta^{-1} |y|^{-2} |( y  -\xi|y|)\,  |\, )  & \hbox{ if } |y|> 1\, , \\ & \\     \zeta{ (\,k \eta^{-1}\,|y-\xi|\, )}   & \hbox{ if } |y|\le  1\, , \end{matrix}
  \right. \
$$
in such a way that that
$$
\zeta_j( y) = \zeta_j( y/|y|^2).
$$
The function $\tilde \phi$  has the form
\begin{equation} \label{decphi}
\tilde  \phi\  =\  \sum_{j=1}^k  \tilde \phi_j  + \psi.
\end{equation}
In the decomposition \equ{decphi} the functions $\tilde \phi_j$, for $j>1$, are defined in terms of $\tilde \phi_1$
\begin{equation}
\tilde  \phi_j (\bar y, y')= \tilde \phi_1( e^{\frac{2\pi j} k i} \bar y, y'), \quad j=1,\ldots, k-1.
\label{sim11}\end{equation}
Each function $\tilde \phi_j$, $j=1, \ldots , k$,  is constructed to be a solution in the whole $\R^n$ to the problem
\begin{equation}
\Delta\tilde \phi_j +  p\gamma|U_*|^{p-1}\zeta_j  \tilde \phi_j   + \zeta_{j} [\, p\gamma |U_*|^{p-1}\psi + E + \gamma N(\tilde \phi_j + \Sigma_{i\ne j} \tilde\phi_i  + \psi) ] = 0 ,
\label{eq4}\end{equation}
while  $\psi $ solves in $\R^n$
$$
\Delta \psi  +   p\gamma U^{p-1}\psi
+ \, [\,p\gamma\,(|U_*|^{p-1}- U^{p-1})( 1- \Sigma_{j=1}^k \zeta_{j}) + p\gamma U^{p-1}\Sigma_{j=1}^k \zeta_{j}   \,]\,\psi\
$$
\begin{equation}
+ \ p\gamma |U_*|^{p-1} \sum_j (1-\zeta_j) \tilde \phi_j
+( 1- \Sigma_{j=1}^k \zeta_{j} )\, (\,E + \gamma N( \Sigma_{j=1}^k  \tilde \phi_j + \psi)\,)  \,=\, 0.
\label{eq5}\end{equation}
Define now $\phi_1 (y) = \mu^{n-2 \over 2} \tilde \phi_1 (\mu y + \xi_1 ) $. Then $\phi_1$ solves the equation
$$
\Delta \phi_1 + f' ( U ) \phi_1  + \chi_1 (\xi_1 + \mu y ) \mu^{n+2 \over 2} E(\xi_1 + \mu y )
$$
\begin{equation}
\label{tor1}
 + \gamma \mu^{n+2 \over 2} {\mathcal N} (\phi_1 ) (\xi_1 + \mu y ) = 0 \quad {\mbox {in}} \quad \R^n
\end{equation}
where
$$
 {\mathcal N} (\phi_1 )  = p (|U_*|^{p-1} \zeta_1 - U_1^{p-1} ) \tilde \phi_1 + \zeta_1 [ p
|U_*|^{p-1} \Psi (\phi_1)
$$
\begin{equation}
\label{defNstorta}
 + N ( \tilde \phi_1 + \sum_{j\not= 1} \tilde \phi_j + \Psi (\phi_1) ) ]
\end{equation}
In  \cite{dmpp1} it is shown that
 the following estimate on the function  $\psi$ hold true:
\begin{equation}
\label{estpsi}
 \| \psi \|_{n-2} \leq C k^{1-{n\over q}}
\end{equation}
where
\begin{equation}
 \|\phi\|_{n-2} :=
\|\,(1+ |y|^{n-2}) \phi \, \|_{L^\infty(\R^n)} \ .
\label{nomstar1}\end{equation}
On the other hand, if we rescale and traslate the function $\tilde \phi_1$
\begin{equation}
\label{phiriscalata}
\phi_1 (y) =\mu^{n-2 \over 2}  \tilde \phi_1 (\xi_1 + \mu y )
\end{equation}
we have the validity of the following estimate for $\phi_1$
\begin{equation}
\label{estphi1}
\|  \phi_1 \|_{n-2} \leq C k^{-{n\over q}} .
\end{equation}
Furthermore, we have
\begin{equation}
\label{gianma}
\| {\mathcal N} (\phi_1 ) \|_{**} \leq C k^{- {2n \over q}},
\end{equation}
see \equ{defNstorta}.
Let us now define the following functions
\begin{equation}
\label{ang2}
 \begin{array}{rr}
\pi_\alpha (y)  = {\partial \over \partial y_\alpha } \tilde  \phi (y)  ,  & \quad   {\mbox {for}} \quad \alpha = 1, \ldots , n; \\
 & \\
 \pi_0 (y)   ={n-2 \over 2}\tilde  \phi (y) + \nabla \tilde \phi (y) \cdot y.  &
\end{array}
\end{equation}
In the above formula $\tilde \phi$ is the function defined in \equ{sol} and described in \equ{decphi}.
Observe that the function $\pi_0 $ is even in each of its variables, namely
$$
\pi_0 ( y_1, \ldots , y_j , \ldots , y_n) = \pi_0 ( y_1, \ldots ,- y_j , \ldots , y_n)\quad \forall j=1, \ldots , n,
$$
while $\pi_\alpha$, for $\alpha = 1, \ldots , n$ is odd in the $y_\alpha$ variable, while it is even in all the other variables. Furthermore, all functions $\pi_\alpha$ are invariant under rotation of ${2\pi \over k}$ in the first two coordinates, namely they satisfy \equ{sim00}.
The functions $\pi_\alpha$ can be further described, as follows.

\begin{proposition}
\label{prop1}
The functions $\pi_\alpha$ can be decomposed into
\begin{equation}
\label{martes4}
\pi_\alpha (y) = \sum_{j=1}^k \tilde \pi_{\alpha , j} (y) + \hat \pi_\alpha (y)
\end{equation}
where
$$
\tilde \pi_{\alpha ,  j} (y) = \tilde \pi_{\alpha , 1} (e^{ {2\pi \over k} \, j \, i} \bar y , y').
$$
Furthermore, there exists a positive constant $C$ so that
$$
\| \hat \pi_0 \|_{n-2} \leq C k^{1-{n\over q}}, \quad \| \hat \pi_j \|_{n-1} \leq C k^{1-{n\over q}}, \quad j=1, \ldots , k,
$$
and, if we denote $\pi_{\alpha , 1} (y) = \mu^{n-2 \over 2} \tilde \pi_{\alpha , 1} (\xi_1 + \mu y )$, then
$$
\|  \pi_{0 , 1}  \|_{n-2} \leq C k^{-{n\over q}}, \quad \|  \pi_{\alpha , 1}  \|_{n-1} \leq C k^{-{n\over q}}, \quad \alpha=1, \ldots , n.
$$

\end{proposition}

The proof of this result can be obtained using similar arguments as the ones used in \cite{dmpp1}. We leave the details to the reader.

\section{Scheme of the proof}

Let $\varphi$ be a bounded function satisfying $L(\varphi ) = 0$, where $L$ is the linear operator defined in \equ{linearop}. We write our function $\varphi$ as
\begin{equation}
\label{defvarphi}
\varphi (x) = \sum_{\alpha=0}^{3n-1} a_\alpha z_\alpha (x) + \tilde \varphi (x)
\end{equation}
where the functions $z_\alpha (x)$ are defined in \equ{capitalzeta0}, \equ{capitalzetaj}, \equ{capitalzeta2} \equ{chico1}, \equ{chico2} respectively,  while the constants $a_\alpha$ are chosen so that
\begin{equation}
\label{orto}
\int u^{p-1} z_\alpha \, \tilde \varphi = 0, \quad  \alpha=0, \ldots , 3n-1 .
\end{equation}
Observe that $L(\tilde \varphi )=0$.
Our aim is to show that, if $\tilde \varphi$ is bounded, then $\tilde \varphi \equiv 0$.

\medskip
\noindent
For this purpose, recall that
$$
u (x) =
U(x) - \sum_{j=1}^k U_j (x) +\tilde  \phi (x), \quad {\mbox {with}} \quad  U(x) = \left( {2 \over 1+ |x|^2} \right)^{n-2 \over 2}
$$
and
$$
 U_j (x) = \mu_k^{-{n-2 \over 2}} U(\mu_k^{-1} (x-\xi_j )).
$$
We introduce the following functions
\begin{equation}\label{defzetapiccole}
Z_0 (x) = {n-2 \over 2} U(x) + \nabla U(x) \cdot x , \end{equation}
and
\begin{equation}\label{defzetapiccole1}
 Z_\alpha (x) = {\partial \over \partial x_\alpha} U(x), \quad {\mbox {for}} \quad \alpha = 1, \ldots , n \end{equation}
Moreover, for  any $l=1, \ldots , k$,  we define
\begin{equation}\label{defzetapiccole11}
Z_{0 l} (x) = {n-2 \over 2} U_l (x) + \nabla U_l (x) \cdot (x-\xi_l ) , \end{equation}
Observe that, as a consequence of \equ{capitalzeta0} and \equ{capitalzetaj}, we have that
$$
z_0 (x) = Z_0 (x) - \sum_{l=1}^k \left[ Z_{0 , l}  (x) + \sqrt{1-\mu^2 } \cos \theta_l  {\partial \over \partial x_1} U_l (x) \right.
$$
$$
+ \left. \sqrt{1-\mu^2} \sin \theta_l {\partial \over \partial x_2} U_l (x)  \right] + \pi_0 (x),
$$
where $\pi_0$ is defined in \equ{ang2}.
Define, for $l=1, \ldots , k$,
\begin{equation}\label{canada1}
 Z_{1 l} (x) =\sqrt{1-\mu^2} \,  \left[\cos \theta_l  {\partial \over \partial x_1} U_l (x)+ \sin \theta_l {\partial \over \partial x_2} U_l (x)  \right]\end{equation}
\begin{equation}\label{canada1}
 Z_{2 l} (x) =\sqrt{1-\mu^2} \,  \left[ -\sin  \theta_l  {\partial \over \partial x_1} U_l (x)+ \cos \theta_l {\partial \over \partial x_2} U_l (x)  \right]
\end{equation}
where
$\theta_l = {2\pi \over k} \, (l-1)$. Furthermore, for any $l=1, \ldots , k$,
\begin{equation}
\label{ccanada}  \quad Z_{\alpha , l} (x) = {\partial \over \partial x_\alpha} U(x), \quad {\mbox {for}} \quad \alpha = 3, \ldots , n.
\end{equation}
Thus, we can write
\begin{equation}
\label{ang1}
z_0 (x) = Z_0 (x) - \sum_{l=1}^k \left[ Z_{0 , l}  (x) +Z_{1,l} (x)  \right] + \pi_0 (x),
\end{equation}
\begin{eqnarray}
\label{ang111}
z_1 (x) &=& Z_1 (x) - \sum_{l=1}^k  {\partial \over \partial x_1} U_l (x) + \pi_1 (x) \\
&=& Z_1 (x) - \sum_{l=1}^k  {  [ \cos \theta_l Z_{1l} (x) -\sin \theta_l Z_{2,l} (x) ] \over \sqrt{1-\mu^2}} + \pi_1 (x)  \nonumber
\end{eqnarray}
\begin{eqnarray}
\label{ang222}
z_2 (x) &=& Z_2 (x) - \sum_{l=1}^k  {\partial \over \partial x_2} U_2 (x) + \pi_2 (x) \\
&=& Z_2 (x) - \sum_{l=1}^k  {  [ \sin \theta_l Z_{1l} (x) +\cos \theta_l Z_{2,l} (x) ] \over \sqrt{1-\mu^2}} + \pi_2 (x)  \nonumber
\end{eqnarray}
and, for $\alpha =3, \ldots , n$,
\begin{equation}
\label{ang333}
z_\alpha (x) = Z_\alpha (x) - \sum_{l=1}^k Z_{\alpha , l} +\pi_\alpha (x)
\end{equation}
Furthermore
\begin{equation}
\label{canada2}
z_{n+1} (x) = \sum_{l=1}^k Z_{2 l} (x)  + x_2 \pi_1 (x) - x_1 \pi_2 (x)
\end{equation}
\begin{eqnarray}
\label{canada3}
z_{n+2} (x) = \sum_{l=1}^k \, \sqrt{1-\mu^2} \, & &  \cos \theta_l Z_{0 l} (x)  - \sum_{l=1}^k  \, \sqrt{1-\mu^2} \,  \cos \theta_l Z_{1 l} (x) \nonumber \\
& & -2x_1 \pi_0 (x) + |x|^2 \pi_1 (x)
\end{eqnarray}
\begin{eqnarray}
\label{canada3}
z_{n+3} (x) =  \sum_{l=1}^k \, \sqrt{1-\mu^2} \,  \sin \theta_l Z_{0 l} (x) & - & \sum_{l=1}^k  \, \sqrt{1-\mu^2} \,  \sin \theta_l Z_{1 l} (x) \nonumber \\
& - & 2x_2 \pi_0 (x) + |x|^2 \pi_2 (x)
\end{eqnarray}
and, for $\alpha =3, \ldots , n$,
\begin{equation}
\label{canada4}
z_{n+\alpha +1} (x) =  \sqrt{1-\mu^2}\,  \sum_{l=1}^k  \cos \theta_l Z_{\alpha l} (x) + x_1 \pi_\alpha (x)
\end{equation}
\begin{equation}
\label{canada5}
z_{2n+\alpha -1} (x) =  \sqrt{1-\mu^2} \,  \sum_{l=1}^k \sin \theta_l Z_{\alpha l} (x) + x_2 \pi_\alpha (x) .
\end{equation}
Let
\begin{equation}
\label{Zalpha0}
Z_{\alpha 0} (x) = Z_\alpha (x) + \pi_\alpha (x)  , \quad \alpha = 0, \ldots , n,
\end{equation}
and introduce the $(k+1)$-dimensional vector functions
$$
\Pi_\alpha (x) =\left[ \begin{array}{r} Z_{\alpha , 0} (x)  \\  Z_{\alpha 1} (x) \\
 Z_{\alpha 2} (x) \\.. \\  Z_{\alpha k} (x)
\end{array}
\right] \quad {\mbox {for}} \quad \alpha = 0, 1 , \ldots , n.
$$
For  a given real vector $\bar c = \left[ \begin{array}{r} c_0 \\  c_1 \\
 c_2 \\.. \\  c_k
\end{array}
\right] \in \R^{k+1}$, we write
$$
\bar c \cdot \Pi_\alpha  (x) =  \sum_{l=0}^{k} c_l Z_{\alpha l} (x).
$$
With this in mind, we write our function $\tilde \varphi $ as
\begin{equation}
\label{scritturauno}
\tilde \varphi  ( x) = \sum_{\alpha =0}^n \textsf{c}_\alpha \cdot \Pi_\alpha (x) +  \varphi^\perp (x)
\end{equation}
where $ \textsf{c}_\alpha = \left[ \begin{array}{r} c_{\alpha 0} \\
c_{\alpha 1} \\  \ldots \\
c_{\alpha k} \end{array} \right]$, $\alpha = 0 , 1 , \ldots , n$, are $(n+1)$ vectors in $\R^{k+1}$ defined  so that
$$
\int U_l^{p-1} (x) Z_{\alpha l} (x) \varphi^{\perp} (x) \, dx= 0,  \quad {\mbox {for all}} \quad l=0, 1, \ldots , k, \quad \alpha = 0, \ldots , n.
$$
Observe that
\begin{equation}
\label{conclusion}
 \textsf{c}_\alpha = 0 \quad {\mbox {for all}} \quad \alpha \quad {\mbox {and}} \quad \varphi^\perp \equiv 0 \Longrightarrow \tilde \varphi  \equiv 0 .
\end{equation}
Hence, our purpose is to show that all vector $\textsf{c}_\alpha$ are zero vectors and that $\varphi^\perp \equiv 0$. This will be consequence of the following three facts.

\medskip
\bigskip
\noindent
{\bf Fact 1.}
The orthogonality conditions \equ{orto} take the form
\begin{equation}
\label{sun1}
\sum_{\alpha = 0 }^n \textsf{c}_\alpha \cdot \int \Pi_\alpha u^{p-1} z_\beta = \sum_{\alpha=0}^n \sum_{l=0}^k c_{\alpha l} \int Z_{\alpha l } u^{p-1} z_\beta = -\int \varphi^\perp u^{p-1} z_\beta
\end{equation}
for $\beta = 0 , \ldots , 3n- 1$. Equation \equ{sun1} is a system of $(n+2)$ linear equations ($\beta = 0 , \ldots , 3n-1$) in the $(n+1) \times (k+1) $ variables  $c_{\alpha l }$.

\medskip
Let us introduce the following three vectors in $\R^k $
\begin{equation}\label{igna1}
 \bf{1}_k = \left[ \begin{array}{r} 1 \\ 1 \\ \ldots \\ 1 \end{array} \right]  , \quad
\bf{cos} = \left[  \begin{array}{r} 1 \\ \cos \theta_2 \\ \ldots \\ \cos \theta_{k-1} \end{array} \right] , \quad
\bf{sin} = \left[  \begin{array}{r} 0 \\ \sin \theta_2 \\ \ldots \\ \sin \theta_{k-1} \end{array} \right] .
\end{equation}
Let us write
$$
\textsf{c}_\alpha = \left[ \begin{array}{r} c_{\alpha , 0} \\  \bar {\textsf{c}}_\alpha \end{array} \right],
 \quad {\mbox {with}} \quad c_{\alpha , 0} \in \R, \, \bar {\textsf{c}}_\alpha  \in \R^k, \quad \alpha = 0,1 , \ldots , n,
$$
and
$$
\bar{ \textsf{c}} = \left[ \begin{array}{r}  \bar {\textsf{c}}_{0} \\ ..\\ \bar {\textsf{c}}_n \end{array} \right]\in \R^{n (k+1)},
\quad \hat{\textsf{c}} =  \left[ \begin{array}{r}   c_{0,0} \\ ..\\ c_{n,0} \end{array} \right] \in \R^{n+1}
$$
\medskip
We have the validity of the following
\begin{proposition}\label{danilo}
The system \equ{sun1} reduces to the following $3n$ linear conditions of the vectors $\textsf{c}_\alpha$:
\begin{equation} \label{len1}
 \textsf{c}_0 \cdot \left[ \begin{array}{r} 1 \\ -  \bf{1}_k \end{array} \right]
+  \textsf{c}_1 \cdot \left[ \begin{array}{r} 0 \\ -  \bf{1}_k \end{array} \right]
 = t_0   +O( k^{-{n\over q}} ) {\mathcal L}_0 ( \bar {\textsf{c}})  +O( k^{1-{n\over q}} ) {\hat {\mathcal L}}_0 ( \hat {\textsf{c}}) ,
\end{equation}
\begin{equation}
\label{len2}
 \textsf{c}_1 \cdot \left[ \begin{array}{r} 1 \\ -  \bf{cos } \end{array} \right]
+  \textsf{c}_2 \cdot \left[ \begin{array}{r} 0 \\     \bf{sin } \end{array} \right]
 =  t_1    +O( k^{-{n\over q}} ) {\mathcal L}_1 ( \bar {\textsf{c}})  +O( k^{1-{n\over q}} ) {\hat {\mathcal L}}_1 ( \hat {\textsf{c}}) ,
\end{equation}
\begin{equation}
\label{len3}
 \textsf{c}_1 \cdot \left[ \begin{array}{r} 0 \\ -  \bf{sin } \end{array} \right]
+  \textsf{c}_2 \cdot \left[ \begin{array}{r} 1 \\  -  \bf{cos } \end{array} \right]
 = t_2    +O( k^{-{n\over q}} ) {\mathcal L}_2 ( \bar {\textsf{c}})  +O( k^{1-{n\over q}} ) {\hat {\mathcal L}}_2 ( \hat {\textsf{c}}) ,
\end{equation}
for $\alpha = 3, \ldots , n$
\begin{equation}
\label{len4}
 \textsf{c}_\alpha \cdot \left[ \begin{array}{r} 1 \\ -  \bf{1 }_k \end{array} \right]
 = t_\alpha   +O( k^{-{n\over q}} ) {\mathcal L}_\alpha ( \bar {\textsf{c}})  +O( k^{1-{n\over q}} ) {\hat {\mathcal L}}_\alpha ( \hat {\textsf{c}}) ,
\end{equation}
\begin{equation}
\label{len50}  \textsf{c}_2 \cdot \left[ \begin{array}{r} 0 \\  \bf{1 }_k \end{array} \right]
 = t_{n+1}    +O( k^{-{n\over q}} ) {\mathcal L}_{n+1} ( \bar {\textsf{c}})  +O( k^{1-{n\over q}} ) {\hat {\mathcal L}}_{n+1} ( \hat {\textsf{c}}) ,
\end{equation}
\begin{equation}
\label{len5}
 \textsf{c}_0 \cdot \left[ \begin{array}{r} 0 \\  \bf{cos }\end{array} \right] - \textsf{c}_1 \cdot \left[ \begin{array}{r} 0 \\  \bf{cos }\end{array} \right]
 =  t_{n+2}    +O( k^{-{n\over q}} ) {\mathcal L}_{n+2} ( \bar {\textsf{c}})  +O( k^{1-{n\over q}} ) {\hat {\mathcal L}}_{n+2} ( \hat {\textsf{c}}) ,
\end{equation}
\begin{equation}
\label{len6} \textsf{c}_0 \cdot \left[ \begin{array}{r} 0 \\  \bf{sin }\end{array} \right] - \textsf{c}_1 \cdot \left[ \begin{array}{r} 0 \\  \bf{sin }\end{array} \right]
 =  t_{n+3}    +O( k^{-{n\over q}} ) {\mathcal L}_{n+3} ( \bar {\textsf{c}})  +O( k^{1-{n\over q}} ) {\hat {\mathcal L}}_{n+3} ( \hat {\textsf{c}}) ,
\end{equation}
for $\alpha = 3, \ldots , n$,
\begin{equation}
\label{len7}
 \textsf{c}_\alpha \cdot \left[ \begin{array}{r} 0 \\  \bf{cos } \end{array} \right]
 =  t_{n+\alpha +1}    +O( k^{-{n\over q}} ) {\mathcal L}_{n+\alpha +1} ( \bar {\textsf{c}})  +O( k^{1-{n\over q}} ) {\hat {\mathcal L}}_{n+\alpha +1} ( \hat {\textsf{c}}) ,
\end{equation}
\begin{equation}
\label{len8}
 \textsf{c}_\alpha \cdot \left[ \begin{array}{r} 0 \\  \bf{sin } \end{array} \right]
 =  t_{2n+\alpha -1}    +O( k^{-{n\over q}} ) {\mathcal L}_{2n+\alpha -1} ( \bar {\textsf{c}})  +O( k^{1-{n\over q}} ) {\hat {\mathcal L}}_{2n +\alpha -1} ( \hat {\textsf{c}}) ,
\end{equation}
In the above expansions, $ \left[ \begin{array}{r} \textsf{t}_0 \\
\textsf{t}_1 \\
\ldots \\
\textsf{t}_n \end{array} \right] $ is a fixed vector with
$$
\|  \left[ \begin{array}{r} \textsf{t}_0 \\
\textsf{t}_1 \\
\ldots \\
\textsf{t}_n \end{array} \right]  \| \leq C \| \varphi^{\perp} \|_*
$$
and ${\mathcal L}_j : \R^{k(n+1) } \to \R^{3n}$, $\hat {{\mathcal L}}_j : \R^{n } \to \R^{3n}$ are linear functions, whose coefficients are constants
uniformly bounded as $k \to \infty$. The number $q$, with $ \frac n2 < q < n$, is the one already fixed in \equ{nomstar2}.
Furthermore, $O(1)$ denotes a quantity which is uniformly bounded as $k \to \infty$.

\end{proposition}
\medskip
We shall prove \equ{len1}--\equ{len8} in Section \ref{proof1}.

\medskip
\bigskip
\noindent
{\bf Fact 2.}
Since
$
L(\tilde \varphi ) = 0$, we have that
\begin{equation}
\label{due}
  \sum_{\alpha =0}^n  \textsf{c}_\alpha \cdot  L ( \Pi_\alpha (x) )= \sum_{\alpha = 0}^n \sum_{l=0}^k c_{\alpha l } L(Z_{\alpha , l} ) = - L(  \varphi^\perp  )
\end{equation}
Let $\varphi^\perp = \varphi_0^\perp + \sum_{l=1}^k \varphi_l^\perp$ where
$$
- L(  \varphi_0^\perp  ) =  \sum_{\alpha = 0}^n  c_{\alpha 0} L(Z_{\alpha , 0}  )
$$
and for any $l=1, \ldots , k$
$$
- L(  \varphi_l^\perp  ) =  \sum_{\alpha = 0}^n  c_{\alpha l } L(Z_{\alpha , l} ).
$$
Furthermore, let
$$
\tilde \varphi_l^{\perp}  (y) = \mu^{n-2 \over 2} \varphi_l^\perp ( \mu y + \xi_l ),
$$
and define
\begin{equation}
\label{defno1}
\| \varphi^\perp \|_* = \| \varphi^\perp \|_{n-2} + \sum_{l=1}^k \| \tilde \varphi_l^\perp \|_{n-2}
\end{equation}
where the $\| \cdot \|_{n-2}$ is defined in \equ{nomstar1}.
A first consequence of \equ{due} is that there exists a positive constant $C$ such that
\begin{equation}
\label{fattouno}
\| \varphi^\perp \|_* \leq C \mu^{1\over 2}  \sum_{\alpha =0}^n \, \|{\textsf{c}}_{\alpha }\|
\end{equation}
for all $k$ large.
We postpone the proof of \equ{fattouno} to Section \ref{proof2}.

\medskip
\bigskip
\noindent
{\bf Fact 3.}
Let us now multiply \equ{due} against $Z_{\beta l}$, for $\beta = 0, \ldots , n$ and $l=0, 1, \ldots , k$. After integrating in $\R^n$ we get a linear system of $(n+1) \times (k+1) $ equations
in the $(n+1) \times (k+1)$ constants $c_{\alpha j}$ of the form
\begin{equation}
\label{sistem1}
M  \left[ \begin{array}{r}  \textsf{c}_0 \\
 \textsf{c}_1\\.. \\  \textsf{c}_n
\end{array}
\right] = - \left[ \begin{array}{r}  \textsf{r}_0 \\
 \textsf{r}_1\\.. \\  \textsf{r}_n
\end{array}
\right], \quad {\mbox {with}} \quad \textsf{r}_\alpha =
 \left[ \begin{array}{r}  \int_{\R^n}  L(\varphi^\perp)  Z_{\alpha , 0}\\
  \int_{\R^n}  L(\varphi^\perp) Z_{\alpha , 1} \\.. \\   \int_{\R^n}  L(\varphi^\perp) Z_{\alpha , k}
\end{array}
\right]
\end{equation}
Observe first that relation  \equ{ang1} together with the fact that $L(z_\alpha )=0$ for all $\alpha =0, \ldots , n$, allow us to say that the vectors $\textsf{r}_\alpha$ have the form
\begin{equation}\label{decr0}
{\mbox {row}}_1 \left( \textsf{r}_0 \right)= \sum_{l=2}^{k+1} \left[  {\mbox {row}}_l ( \textsf{r}_0 )+ {\mbox {row}}_l ( \textsf{r}_1 ) \right]
\end{equation}
\begin{equation}\label{decr00}
{\mbox {row}}_1 \left( \textsf{r}_1 \right)= {1\over \sqrt{1-\mu^2}}\,  \sum_{l=2}^{k+1} \left[ \cos \theta_l  {\mbox {row}}_l   ( \textsf{r}_1  ) - \sin \theta_l  {\mbox {row}}_l ( \textsf{r}_2 ) \right], \quad
\end{equation}
\begin{equation}\label{decr000}
 {\mbox {row}}_1 \left( \textsf{r}_2 \right)= {1\over \sqrt{1-\mu^2}}\,  \sum_{l=2}^{k+1}\left[ \sin \theta_l  {\mbox {row}}_l (  \textsf{r}_1 )  + \cos \theta_l  {\mbox {row}}_l ( \textsf{r}_2 ) \right]
\end{equation}
\begin{equation}\label{decr0000}
{\mbox {row}}_1 \left( \textsf{r}_\alpha \right)= \sum_{l=2}^{k+1} {\mbox {row}}_l \left( \textsf{r}_\alpha \right) \quad {\mbox {for all}} \quad \alpha = 3, \ldots , n.
\end{equation}
Here with $ {\mbox {row}}_l $ we denote the $l$-th row.

The matrix $M$ in \equ{sistem1} is a square matrix of dimension $[ (n+1) \times (k+1) ]^2$. The entries of $M$ are numbers of the form
$$
\int_{\R^n} L(Z_{\alpha l} ) Z_{\beta j} \, dy
$$
for $\alpha$, $\beta = 0, \ldots , n$ and $l , j = 0, 1, \ldots , k$.

A first observation is that, if $\alpha $ is any of the indeces $\{ 0 , 1 , 2 \}$, and $\beta $ is any of the index in $\{ 3, \ldots , n \}$, then by symmetry the above integrals are zero, namely
$$
\int_{\R^n} L(Z_{\alpha l} ) Z_{\beta j} \, dy = 0 \quad {\mbox {for any}} \quad l , j =0, \ldots , k
$$
 This fact implies that  the  matrix $M$ has the form
\begin{equation}\label{matM}
M= \left[ \begin{array}{rr}  M_1 & 0\\
 0& M_2
\end{array}
\right]
\end{equation}
where $M_1$ is a square matrix of dimension $(3\times (k+1))^2$ and $M_2$ is a square matrix of dimension $[(n-2) \times (k+1)]^2$.

Since
$$
\int_{\R^n} L(Z_{\alpha l} ) Z_{\beta j} \, dy = \int_{\R^n} L(Z_{\beta j} ) Z_{\alpha l} \, dy
$$
for $\alpha$, $\beta = 0, \ldots , n$ and $l , j = 0, 1, \ldots , k$,
we can write
\begin{equation}\label{matM1}
M_1 = \left[ \begin{array}{rrr}  \bar A & \bar B & \bar C  \\
{\bar  B}^T & \bar F & \bar D  \\
 {\bar C }^T  & {\bar D }^T & \bar G
\end{array}
\right]
\end{equation}
where $\bar A$,  $\bar B$,  $\bar C$,  $\bar D$,  $\bar F$ and  $\bar G$ are square matrices of dimension $(k+1)^2$, with $\bar A$, $\bar F$ and $\bar G$  symmetric.
More precisely,
\begin{equation}
\label{AF0}
\bar A=  \left( \int L(Z_{0i} ) Z_{0j} \right)_{i,j=0,1,\ldots , k} , \bar  F =  \left( \int L(Z_{1i} ) Z_{1j} \right)_{i,j=0,1,\ldots , k} ,
\end{equation}
\begin{equation}
\label{GB0} \bar G=  \left( \int L(Z_{2i} ) Z_{2j} \right)_{i,j=0,1,\ldots , k}  ,\bar  B =  \left( \int L(Z_{0i} ) Z_{1j} \right)_{i,j=0,1,\ldots , k} ,
\end{equation}
and
\begin{equation}
\label{CD0}
  \bar C=  \left( \int L(Z_{0i} ) Z_{2j} \right)_{i,j=0,1,\ldots , k},\bar  D=  \left( \int L(Z_{1i} ) Z_{2j} \right)_{i,j=0,1,\ldots , k}
\end{equation}
Furthermore, again by symmetry, since
$$
\int L(Z_{\alpha i }) Z_{\beta j} \, dx = 0, \quad {\mbox {if}} \quad \alpha \not= \beta, \quad \alpha , \beta =3, \ldots , n
$$
the matrix $M_2$ has the form
\begin{equation}
\label{matM2}
M_2 = \left[ \begin{array}{rrrrr}  \bar H_3 & 0 & 0 & 0 & 0 \\
 0& \bar H_4 & 0 & 0 & 0 \\
 .. & .. & .. & ..& .. \\
 0 & 0 & 0 &\bar  H_{n-1}& 0 \\
 0 & 0 & 0 & 0 & \bar H_n
\end{array}
\right]
\end{equation}
where $\bar H_j$ are square matrices of dimension $(k+1)^2$, and each of them is symmetric. The matrices $\bar H_\alpha$ are defined by
\begin{equation}
\label{H0}
\bar H_\alpha=   \left( \int L(Z_{\alpha i} ) Z_{\alpha j} \right)_{i,j=0,1,\ldots , k}, \quad \alpha = 3, \ldots , n.
\end{equation}

Thus, given the form of the matrix $M$ as described in \equ{matM}, \equ{matM1} and \equ{matM2}, system \equ{sistem1} is equivalent to
\begin{equation}
\label{sistem2}
M_1  \left[ \begin{array}{r}  \textsf{c}_0 \\
 \textsf{c}_1\\ \textsf{c}_2
\end{array} \right] =   \left[ \begin{array}{r}  \textsf{r}_0 \\
 \textsf{r}_1\\ \textsf{r}_2
\end{array} \right], \quad \bar H_\alpha \textsf{c}_\alpha = \textsf{r}_\alpha \quad {\mbox {for}} \quad \alpha = 3, \ldots , n,
\end{equation}
where the vectors $\textsf{r}_\alpha$ are defined in \equ{sistem2}.

\medskip
\noindent
 Observe that system \equ{sistem2} impose $(n+1 ) \times (k+1)$ linear conditions on the $(n+1)\times (k+1)$ constants $c_{\alpha j}$. We shall show that $3n$ equations in \equ{sistem2} are linearly dependent. Thus in reality system \equ{sistem2} reduce to only  $(n+1 ) \times (k+1)-3n $ linearly independent conditions on the $(n+1)\times (k+1)$ constants $c_{\alpha j}$. We shall also show that system \equ{sistem2} is solvable. Indeed we have the validity of the following
\begin{proposition}\label{danilo1} There exist $k_0 $ and $C$ such that, for all $k>k_0$ System \equ{sistem2} is solvable. Furthermore, the solution has the form
$$
 \left[ \begin{array}{r}  \textsf{c}_0 \\
 \textsf{c}_1\\ \textsf{c}_2
\end{array} \right]
=  \left[ \begin{array}{r}   \textsf{v}_0 \\
  \textsf{v}_1\\ \ \textsf{v}_2
\end{array} \right]
+ s_1 \left[ \begin{array}{r}  1 \\
- \bf{1}_k \\ 0 \\ - \bf{1}_k \\ 0 \\ 0
\end{array} \right]
+ s_2  \left[ \begin{array}{r} 0\\ 0\\ 1 \\
-{1\over \sqrt{1-\mu^2}} \bf{cos } \\ 0 \\ {1\over \sqrt{1-\mu^2}} \bf{sin}
\end{array} \right] + s_3\left[ \begin{array}{r}  0\\ 0\\ 0\\ -{1\over \sqrt{1-\mu^2}} \bf{sin} \\ 1 \\
-{1\over \sqrt{1-\mu^2}} \bf{cos}
\end{array} \right]
$$
$$
+s_4 \left[ \begin{array}{r}  0 \\
0 \\ 0 \\ \\ 0 \\ \bf{1}_k
\end{array} \right] + s_5 \left[ \begin{array}{r}  0 \\
 \bf{cos} \\ 0 \\ - \bf{cos} \\ 0 \\ 0
\end{array} \right] + s_6 \left[ \begin{array}{r}  0 \\
 \bf{sin} \\ 0 \\ - \bf{sin} \\ 0 \\ 0
\end{array} \right]
$$
and
$$
\textsf{c}_\alpha = \textsf{v}_\alpha + s_{\alpha 1} \left[ \begin{array}{r}  1 \\
- \bf{1}_k
\end{array} \right] + s_{\alpha 2} \left[ \begin{array}{r}  0 \\
\bf{cos}
\end{array} \right] + s_{\alpha 3} \left[ \begin{array}{r}  0 \\
 \bf{sin}
\end{array} \right] , \quad \alpha = 3, \ldots ,n
$$
for any $s_1 , \ldots , s_6, s_{\alpha 1}, s_{\alpha 2} , s_{\alpha 3} \in \R$, where the vectors $\textbf{v}_\alpha$ are fixed vectors with
$$
\| \textsf{v}_\alpha  \| \leq C \| \varphi^{\perp}  \|, \quad \alpha = 0, 1,  \ldots , n.
$$
\end{proposition}

\medskip
\noindent
Conditions \equ{len1}--\equ{len8} guarantees that the solution $\textsf{c}_\alpha$ to \equ{sistem2} is indeed unique. Furthermore, we shall show that
 there exists a positive constant $C$ such that
\begin{equation}
\label{conse1} \sum_{\alpha =0}^n \| \textsf{c}_\alpha \| \leq C \| \varphi^\perp \|_*.
\end{equation}
Here $\| \cdot \|$ denotes the euclidean norm in $\R^k$.

\medskip
\noindent
Estimates \equ{conse1} combined with \equ{fattouno} gives that
\begin{equation}
\label{conse2}
\textsf{c}_\alpha = 0 \quad \forall \alpha = 0 , \ldots , n.
\end{equation}
Replacing equation \equ{conse2} into \equ{fattouno} we finally get \equ{conclusion}, namely
$$  \textsf{c}_\alpha = 0 \quad {\mbox {for all}} \quad \alpha \quad {\mbox {and}} \quad \varphi^\perp \equiv 0.
$$

\medskip
\noindent
{\bf Scheme of the paper}: In Section \ref{simple}  we discuss and simplify system \equ{sistem2}. In Section \ref{lineare} we establish an invertibility theory for solving \equ{sistem2}.
Section \ref{final0} is devoted to prove Proposition \ref{danilo1}.  In Section \ref{final} we prove Theorem \ref{teo}. Section \ref{proof1} is devoted to the  proof of Proposition \ref{danilo}, while Section \ref{proof2} is devoted to the proof of \equ{fattouno}. Section \ref{appe1} is devoted to the detailed proofs of several computations.

\section{A first simplification of the system \equ{sistem2}}\label{simple}

Let us consider system \equ{sistem2} and let us fix $\alpha \in \{ 3, \ldots , n\}$. Recall that the function $z_\alpha$ defined in \equ{capitalzetaj} satisfies $L(z_\alpha )=0$. Hence, by \equ{ang1}, \equ{Zalpha0} and \equ{H0} we have that
$$
{\mbox {row}}_1 (\bar H_\alpha ) = \sum_{l=2}^{k+1} {\mbox {row}}_l (\bar H_\alpha ).
$$
This implies that $  \left[ \begin{array}{r}  1 \\
- \bf{1}_k
\end{array} \right]
\in {\mbox {kernel}} (\bar H_\alpha )$ and thus that the system  $\bar H_\alpha (\textsf{c}_\alpha ) = \textsf{r}_\alpha $ is solvable only if  $ \text{r}_\alpha \cdot  \left[ \begin{array}{r}  1 \\
- \bf{1}_k
\end{array} \right] =0$. On the other hand, this last solvability condition is satisfied as consequence of \equ{decr0000}.  Thus $\bar H_\alpha \textsf{c}_\alpha = \textsf{r}_\alpha$ is solvable.

Arguing similarly, we get that
$$
{\mbox {row}}_1 (M_1 ) = \sum_{l=2}^{k+1} {\mbox {row}}_l (M_1 ) + \sum_{l=k+3}^{2k+2} {\mbox {row}}_l (M_1 ) , \quad
$$
$$
{\mbox {row}}_{k+2} (M_1 ) ={1\over \sqrt{1-\mu^2}} \left[  \sum_{l=1}^{k} \cos \theta_l  {\mbox {row}}_{k+2+l}  (M_1 ) -  \sum_{l=1}^{k} \sin \theta_l {\mbox {row}}_{2k+3+l} (M_1 ) \right]  ,
$$
and
$$
{\mbox {row}}_{2k+3} (M_1 ) = {1\over \sqrt{1-\mu^2}} \left[  \sum_{l=1}^{k} \sin \theta_l  {\mbox {row}}_{k+2+l}  (M_1 ) +  \sum_{l=1}^{k} \cos \theta_l {\mbox {row}}_{2k+3+l} (M_1 ) \right] .
$$
This implies that the vectors  $${\bf{w}_0}=  \left[ \begin{array}{r}  1 \\
- \bf{1}_k \\ 0 \\ - \bf{1}_k \\ 0 \\ 0
\end{array} \right], \quad  {\bf{w}_1} = \left[ \begin{array}{r} 0\\ 0\\ 1 \\
-{1\over \sqrt{1-\mu^2}} \bf{cos } \\ 0 \\ {1\over \sqrt{1-\mu^2}} \bf{sin}
\end{array} \right], \quad  {\bf{w}_2} =\left[ \begin{array}{r}  0\\ 0\\ 0\\ -{1\over \sqrt{1-\mu^2}} \bf{sin} \\ 1 \\
-{1\over \sqrt{1-\mu^2}} \bf{cos}
\end{array} \right]
\in {\mbox {kernel}} (M_1 )$$ and thus that the system  $M_1  \left[ \begin{array}{r}  \textsf{c}_0 \\
 \textsf{c}_1\\ \textsf{c}_2
\end{array} \right] =   \left[ \begin{array}{r}  \textsf{r}_0 \\
 \textsf{r}_1\\ \textsf{r}_2
\end{array} \right] $ is solvable only if  $  \left[ \begin{array}{r}  \textsf{r}_0 \\
 \textsf{r}_1\\ \textsf{r}_2
\end{array} \right] \cdot  \bf{w}_j=0$, for $j=0,1,2$.
 On the other hand, this last solvability condition is satisfied as consequence of \equ{decr0}, \equ{decr00} and \equ{decr000}.

\medskip
\noindent
We thus conclude that system \equ{sistem2} is solvable and the solution has the form
\begin{equation}\label{solsist1}
\left[ \begin{array}{r}  \textsf{c}_0 \\
 \textsf{c}_1\\ \textsf{c}_2
\end{array} \right] =  \left[ \begin{array}{r}  0 \\
\bar {\textsf{c}}_0 \\
0 \\ \bar {\textsf{c}}_1 \\0 \\ \bar {\textsf{c}}_\alpha
\end{array} \right] + t  {\bf{w}_0}+ s  {\bf{w}_1}+ r {\bf{w}_2} \quad {\mbox {for all}} \quad t, s, r \in \R
\end{equation}
and, if $\alpha = 3, \ldots , n$
\begin{equation}
\label{solsist2}
\textsf{c}_\alpha = \left[ \begin{array}{r}  0 \\
\bar {\textsf{c}}_\alpha
\end{array} \right] + t \left[ \begin{array}{r}  1 \\
- \bf{1}_k
\end{array} \right] \quad {\mbox {for all}} \quad t \in \R
\end{equation}
In \equ{solsist1}-\equ{solsist2},  $\bar {\textsf{c}}_\alpha$ for $\alpha =0, \ldots , n$,  are $(n+1)$ vectors in  $ \R^k$, respectively  given by
\begin{equation}
\label{vectorsc}
\bar {\textsf{c}}_\alpha = \left[ \begin{array}{r} c_{\alpha 1} \\ c_{\alpha 2} \\ \ldots \\ c_{\alpha k}
\end{array} \right].
\end{equation}
These vectors correspond to solutions of the systems
\begin{equation}
\label{sistem3}
N  \left[ \begin{array}{r}  \bar {\textsf{c}}_0 \\
\bar { \textsf{c}}_1\\ \bar {\textsf{c}}_2
\end{array} \right] = \left[ \begin{array}{r} \bar { \textsf{r}}_0 \\
\bar{ \textsf{r}}_1\\ \bar {\textsf{r}}_2
\end{array} \right] , \quad H_\alpha \, [
 \bar {\textsf{c}}_\alpha ] = \bar {\textsf{r}}_\alpha \quad {\mbox {for}} \quad \alpha  = 3, \ldots , n.
\end{equation}
In the above formula $\bar {\textsf{r}}_\alpha$ for $\alpha =0, \ldots , n$,  are $(n+1)$ vectors in  $ \R^k$, respectively  given by
$$ \bar {\textsf{r}}_\alpha =  \left[ \begin{array}{r}
  \int_{\R^n}  L(\varphi^\perp) Z_{\alpha , 1} \\.. \\   \int_{\R^n}  L(\varphi^\perp) Z_{\alpha , k}
\end{array}
\right].
$$
In \equ{sistem3} the matrix $N$ is defined by
\begin{equation}
\label{roma4}
N:= \left[ \begin{array}{rrr}  A  & B &  C \\
  B^T  &  F &  D  \\
  C^T &  D^T &  G
\end{array}
\right]
\end{equation}
where $A$, $B$, $C$, $D$, $F$, $G$ are $k\times k $ matrices whose entrances are given respectively by
\begin{equation}
\label{AF}
A=  \left( \int L(Z_{0i} ) Z_{0j} \right)_{i,j=1,\ldots , k} ,  F =  \left( \int L(Z_{1i} ) Z_{1j} \right)_{i,j=1,\ldots , k} ,
\end{equation}
\begin{equation}
\label{GB} G=  \left( \int L(Z_{2i} ) Z_{2j} \right)_{i,j=1,\ldots , k}  , B =  \left( \int L(Z_{0i} ) Z_{1j} \right)_{i,j=1,\ldots , k} ,
\end{equation}
and
\begin{equation}
\label{CD}
  C=  \left( \int L(Z_{0i} ) Z_{2j} \right)_{i,j=1,\ldots , k}, D=  \left( \int L(Z_{1i} ) Z_{2j} \right)_{i,j=1,\ldots , k}
\end{equation}
Furthermore, in \equ{sistem3} the matrix $H_\alpha$ is defined by
\begin{equation}
\label{H}
H_\alpha=   \left( \int L(Z_{\alpha i} ) Z_{\alpha j} \right)_{i,j=1,\ldots , k}, \quad \alpha = 3, \ldots , n.
\end{equation}

\bigskip
The rest of this section is devoted to compute explicitely the entrances of the matrices
$A$, $B$, $C$, $D$, $F$, $G$, $H_\alpha$  and their eigenvalues.

\medskip
We start with the following observation: all matrices $A$, $B$, $C$, $D$, $F$, $G$ and $H_\alpha$  in \equ{sistem3}  are circulant matrices of dimension $k \times k$. For properties of circulant matrices, we refer to \cite{KS}.

\medskip

A circulant matrix $X$ of dimension $k \times k$ has the form
$$
X= \left[ \begin{array}{rrrrrr} x_0 & x_1 & \ldots & \ldots & x_{k-2}  & x_{k-1} \\
x_{k-1}& x_0 & x_1 & \ldots & \ldots & x_{k-2} \\
\ldots & x_{k-1} & x_0 & x_1 & \ldots & \ldots \\
\ldots & \dots & \ldots & \ldots & \ldots & \ldots \\
\ldots & \dots & \ldots & \ldots & \ldots & x_1 \\
x_1 & \ldots & \ldots & \ldots & x_{k-1 } & x_0
\end{array}
\right],
$$
or equivalently, if $x_{ij}$, $i, j=1, \ldots , k$ are the entrances of the matrix $X$, then
$$
x_{i,j} = x_{1, |i-j|+1}.
$$
In particular, in order to know a circulant matrix it is enough to know the entrances of its first row.

\medskip
\noindent
The eigenvalues of a circulant matrix $X$ are given by the explicit  formula
\begin{equation}
\label{eigenva}
\eta_m = \sum_{l=0}^{k-1} x_l e^{ {2\pi \, m \over k} \, i \, l}, \quad m=0, \ldots , k-1
\end{equation}
and with corresponding  normalized  eigenvectors defined by
\begin{equation}
\label{eigenve}
E_m = k^{-{1\over 2}} \left[ \begin{array}{r} 1 \\
 e^{ {2\pi \, m \over k} \, i \, }\\
 e^{ {2\pi \, m \over k} \, i \, 2}\\
\ldots \\
 e^{ {2\pi \, m \over k} \, i \, (k-1)}
\end{array}
\right] \quad m=0, \ldots , k-1.
\end{equation}
Observe that any circulant matrix $X$ can be diagonalized
$$
X= P \, D_X \, P^T
$$
where $D_X$ is the diagonal matrix
\begin{equation}
\label{diag1}
D_X = {\mbox {diag}} ( \eta_0, \eta_1 , \ldots , \eta_{k-1} )
\end{equation}
and $P$ is the $k \times k$ invertible matrix defined by
\begin{equation}
\label{diag2}
P= \left[ \begin{array}{rrrrrr} E_0 & \bigl| & E_1 & \bigl| & \ldots & E_{k-1} \\
\end{array} \right].
\end{equation}

\medskip
The matrices $A$, $B$, $C$, $D$, $F$, $G$ and $H_\alpha$ are  circulant as a consequence of the invariance under rotation of an angle ${2\pi \over k}$ in the
$(x_1 , x_2 )$-plane of the functions $Z_{\alpha j}$. This is trivial in the case of $Z_{0l}$ and $Z_{\alpha , l}$ for all $\alpha =3, \ldots , n$. On the other hand, if we denote by $R_j$ the rotation in the $(x_1 , x_2)$ plane of angle ${2\pi \over k} (j-1)$, then we get
$$
Z_{1, j} (x) = \nabla U_j (x) \cdot \xi_j = \mu^{-{n-2 \over 2}} \nabla U ({R_j (y-\xi_1) \over \mu} ) \cdot R_j \xi_1
$$
$$
= \mu^{-{n-2 \over 2}} R_j^{-1} U ({R_j (y-\xi_1) \over \mu} ) \cdot \xi_1, \quad x=R_j y.
$$
Thus, for instance
$$
(F)_{jj} = \int L(Z_{1j} ) Z_{1j} =  \int L(Z_{11} ) Z_{11} = (F)_{11} , \quad j=1, \ldots , k
$$
and, after a rotation of an angle of ${2\pi \over k} (|h-j|+1)$,
$$
(F)_{hj} = \int L (Z_{1h} ) Z_{1j} = \int L (Z_{11} ) Z_{1 (j-h+1)} = (F)_{1 (|j-h|+1)}
$$
In a similar way one can show that
$$
Z_{2, j} (x)
= \mu^{-{n-2 \over 2}} R_j^{-1} U ({R_j (y-\xi_1) \over \mu} ) \cdot \xi_1^{\perp} , \quad x=R_j y.
$$
With this in mind, it is straightforward to show that also the matrices $B$, $C$, $D$ and  $G$ are circulant.

\medskip
A second observation we want to make is that
$$
A, B, F, G, H_\alpha \quad {\mbox {are symmetric}}
$$
while
$$
C, D \quad {\mbox {are anti-symmetric}}.
$$
The fact that $A$, $F$, $G$ and $H_\alpha$ are symmetric follows directly from their definition.
On the other hand, we have
$$
Z_{1j} (x) = \mu^{-{n-2 \over 2}} R_{2j}^{-1} \nabla U ({R_{2j} (y-\xi_{k-j+1}) \over \mu} ) \cdot \xi_{k-j+1},\quad x=R_{2j} y
$$
thus
$$
B_{1 , j} = \int L(Z_{0,1} ) Z_{1, j} = \int L(Z_{0,1} ) Z_{1, k-j+2} = B_{1, k-j+2}.
$$
Furthermore,
$$
Z_{2j} (x) = \mu^{-{n-2 \over 2}} R_{2j}^{-1} \nabla U ({R_{2j} (y-\xi_{k-j+1}) \over \mu} ) \cdot  (- \xi_{k-j+1} )^{\perp} , \quad  x=R_{2j} y
$$
and thus
$$
C_{1 , j} = \int L(Z_{0,1} ) Z_{2, j} = - \int L(Z_{0,1} ) Z_{2, k-j+2} =-  C_{1, k-j+2},
$$
and
$$
D_{1 , j} = \int L(Z_{1,1} ) Z_{2, j} =- \int L(Z_{1,1} ) Z_{2, k-j+2} = - D_{1, k-j+2},
$$
for $j\geq 2$. Combining this property with the property of being circulant, we get that $B$ is symmetric,
while $C$ and $D$ are anti-symmetric.

\medskip

Let us now introduce the following positive number
\begin{equation}
\label{Xi}
\Xi =  p \, \gamma  {(n-2) \over 2} \left( - \int y_1 \, U^{p-1} \, Z_1 (y) \, dy \right).
\end{equation}

\medskip
\noindent
Next we describe the entrances of the matrices $A$, $F$, $G$, $B$, $C$ , $D$ and $H_\alpha$,
together with their eigenvalues.
We refer the reader to Section \ref{appe1} for the detailed proof of the following expansions.
With $O(1)$ we denotes a quantity which is uniformly bounded, as $k \to \infty$.

\medskip
\noindent
{\bf The matrix $ A$.}
The matrix  $A= (A_{ij} )_{i, j=1, \ldots , k}$ defined by
$$
A_{ij} = \int_{\R^n} L(Z_{0i} ) Z_{0j}
$$
is  symmetric.
We have
\begin{equation} \label{A11}
A_{11}= k^{n-2} \mu^{n-1} O(1)
\end{equation}
and
 for any integer $l>1$,
\begin{equation}
\label{A1l}
A_{1l}=  \, \Xi  \,\left[ {- {(n-2) \over 2} \over (1-\cos \theta_l )^{n-2 \over 2} } \right] \, \mu^{n-2} +  \mu^{n-1 } k^{n-2} O(1),
\end{equation}
where $O(1)$ is bounded as $k \to \infty$.

\medskip
\noindent
{\bf Eigenvalues for $A$:} A direct application of \equ{eigenva} gives that the eigenvalues of the matrix $A$ are given by
\begin{equation}
\label{am}
a_m =  \Xi   \,  \bar a_m  \,  \mu^{n-2} .
\end{equation}
for $m=0, 1, \ldots , k-1$, where
\begin{eqnarray}
\label{baram}
\bar a_m &= & - {n-2 \over 2}  \left[ \sum_{l>1}^k {\cos (m\theta_l ) \over (1-\cos \theta_l )^{n-2 \over 2} }  \right] \, \left( 1+ O({1\over k}) \right) \nonumber \\
&=&  - {n-2 \over 2} {k^{n-2} \over (\sqrt{2} \pi )^{n-2} } g''  ({2\pi \over k } m )  \left(1+ O({1\over k}) \right)
\end{eqnarray}
where $g$ is the function defined in \equ{defg}.

\medskip
\noindent
{\bf The matrix $F$.}
The matrix  $F= (F_{ij} )_{i, j=1, \ldots , k}$ defined by
$$
F_{ij} = \int_{\R^n} L(Z_{1i} ) Z_{1j}
$$
is symmetric. We have
\begin{equation}
\label{F11}
F_{11} =  \Xi   \, \left[ \sum_{l>1}^k {  \cos \theta_l \over (1-\cos \theta_l )^{n \over 2}} \right] \, \mu^{n- 2 \over 2} + O(\mu^{n \over 2} )
\end{equation}
and, for any $l >1$
\begin{equation}
\label{F1l}
F_{1l} = \, \Xi  \,  \left[ {{n-2  \over 2} \cos \theta_l - {n \over 2}   \over (1-\cos \theta_l )^{n\over 2}} \right] \mu^{n-2} + O (\mu^{n\over 2} )
\end{equation}
where $O(1) $ is bounded as $k \to 0$.

\medskip
\noindent
{\bf Eigenvalues for $F$.} For any $m=0, \ldots , k-1$, the eigenvalues of $F$ are
\begin{equation}
\label{fm}
f_m =  \Xi \, \bar f_m  \, \mu^{n-2} .
\end{equation}
where
\begin{eqnarray}\label{newfbarm}
\bar f_m &=& \left[ \sum_{l>1}^k { \cos  \theta_l \over (1-\cos \theta_l )^{n-2 \over 2} }
 \right. \nonumber \\
 &+ &  \left. \sum_{l>1}^k {{n-2 \over 2} \cos \theta_l -{n\over 2}   \over (1-\cos \theta_l )^{n \over 2}} \cos m \theta_l \right]  \left( 1+ O({1\over k}) \right).
\end{eqnarray}

\medskip
\noindent
{\bf The matrix $ G$.}
The matrix  $G= (G_{ij} )_{i, j=1, \ldots , k}$ defined by
$$
G_{ij} = \int_{\R^n} L(Z_{2i} ) Z_{2j}
$$
is symmetric.
We have
\begin{equation}
\label{G11}
G_{11} =    \Xi \,   \left[ \sum_{l>1}^k {{n-2\over 2} \cos \theta_l  +{ n \over 2}  \over (1-\cos \theta_l )^{n \over 2}} \right]  \,   \mu^{n-2 \over 2 }  \, + \, \mu^{n\over 2} O(1)
\end{equation}
and, for $l>1$,
\begin{equation}
\label{G1l}
G_{1l}=  - \, \Xi \,   \left[  {{n-2\over 2} \cos \theta_l  +{ n \over 2}  \over (1-\cos \theta_l )^{n \over 2}} \right] \, \mu^{n-2} \, + \, O(\mu^{n\over 2})
\end{equation}
Again $O(1)$ is bounded as $k \to \infty$.

\medskip
\noindent
{\bf Eigenvalues for $G$.} The eigenvalues of $G$ are given by
\begin{equation}
\label{gm}
g_m =   \, \Xi \, \bar g_m  \, \mu^{n-2}
\end{equation}
for  $m=0, \ldots , k-1$ where
\begin{eqnarray} \label{bargm}
\bar g_m &= &-
 \left[  \sum_{l>1}^k {\left( {n-2 \over 2} \cos \theta_l + {n \over 2} \right) \left(1 -  \cos m \theta_l \right)\over (1-\cos \theta_l )^{n \over 2}} \right]  \left( 1+ O({1\over k}) \right) \nonumber \\
&=&-
 {k^{n} \over (\sqrt{2} \pi )^{n}} \, (n-1) \, g({2\pi \over k} m) \left( 1+  O({1 \over k}) \right),
\end{eqnarray}
see \equ{defg} for the definition of $g$.

\medskip
\noindent
{\bf The matrix $ B$.}
The matrix  $B= (B_{ij} )_{i, j=1, \ldots , k}$ defined by
$$
B_{ij} = \int_{\R^n} L(Z_{0i} ) Z_{1j}
$$
is symmetric. We have
\begin{equation}
\label{B11}
B_{11} = \mu^{n-1} k^{n-2} O(1)
\end{equation}
and, for any $l>1$,
\begin{equation}
\label{B1l}
B_{1l}= \, \Xi \,  \left[ {{n-2 \over 2}  \over (1-\cos \theta_l )^{n-2 \over 2}} \right] \mu^{n-2}  + \mu^{n-1} k^{n-2} O(1).
\end{equation}

\medskip
\noindent
{\bf Eigenvalues for $B$.}  For any $m=0, \ldots , k-1$
\begin{equation}
\label{bm}
b_m =  \Xi \, \bar b_m \, \mu^{n-2}
\end{equation}
with
\begin{eqnarray}\label{newbbarm}
\bar b_m &=& {n-2 \over 2} \sum_{l>1}^k {\cos m \theta_l \over (1-\cos \theta_l )^{n-2 \over 2}}   \, \left( 1+ O({1\over k}) \right) \nonumber \\
&=&  {n-2 \over 2} {k^{n-2} \over (\sqrt{2} \pi )^{n-2} }\, g''  ({2\pi \over k } m ) \left( 1+ O({1\over k})\right)
\end{eqnarray}
see \equ{defg} for the definition of $g$.

\medskip
\noindent
{\bf The matrix $ C$.}
The matrix  $C= (C_{ij} )_{i, j=1, \ldots , k}$ defined by
$$
C_{ij} = \int_{\R^n} L(Z_{0i} ) Z_{2j}
$$
is anti symmetric. We have
\begin{equation}
\label{C11}
C_{11} = k^{n-2} \mu^{n-1} O(1)
\end{equation}
and, for $l>1$,
\begin{equation}
\label{C1l}
C_{1l}= \, \Xi \,  \left[ {{n-2 \over 2} \sin \theta_l  \over (1-\cos \theta_l )^{n\over 2}} \right] \mu^{n-2} + k^{n-2} \mu^{n-1} O(1).
\end{equation}

\medskip
\noindent
{\bf Eigenvalues for $C$.} For any $m=0, \ldots , k-1$
\begin{equation}
\label{cm}
c_m =  \, \, \Xi \, i \, \bar c_m  \, \mu^{n-2}
\end{equation}
where
\begin{eqnarray}
\label{barcm}
\bar c_m  &= & {n-2 \over 2}    \left[
\sum_{l>1}^k  {\sin \theta_l \, \sin m \theta_l \over
(1-\cos \theta_l )^{n \over 2} } \right] \,  \left( 1+ O({1\over k}) \right)  \nonumber
\\
&=&  {n-2 \over 2} {\sqrt{2} k^{n-1} \over (\sqrt{2} \pi )^{n-1} } g' ({2\pi \over k } m )  \,  \left( 1+ O({1\over k}) \right)
\end{eqnarray}
see \equ{defg} for the definition of $g$.

\medskip
\noindent
{\bf The matrix $ D$.}
The matrix  $D= (D_{ij} )_{i, j=1, \ldots , k}$
$$
D_{ij} = \int_{\R^n} L(Z_{1i} ) Z_{2j}
$$
is anti symmetric. We have
\begin{equation}
\label{D11}
D_{11} = k^{n-1} \mu^{n-1} O(1)
\end{equation}
and, for $l>1$,
\begin{equation}
\label{D1l}
D_{1l}=- \, \Xi \,  \left[ {{n-2 \over 2} \sin \theta_l  \over (1-\cos \theta_l )^{n\over 2}} \right] \mu^{n-3} + k^{n-1} \mu^n O(1).
\end{equation}

\medskip
\noindent
{\bf Eigenvalues for $D$.} For any $m=0, \ldots , k-1$
\begin{equation}
\label{dm}
d_m = -  i \,  \, \Xi   \, \bar d_m \, \mu^{n-2} \,
\end{equation}
with
\begin{eqnarray*}
\bar d_m & =& -{n -2 \over 2} \left[
\sum_{l>1}^k {\sin \theta_l \sin m \theta_l \over
(1-\cos \theta_l )^{n \over 2} } \right]  \, \left( 1+ O({1\over k}) \right) \\
&=&-   {n-2 \over 2} {\sqrt{2} k^{n-1} \over (\sqrt{2} \pi )^{n-1} } g' ({2\pi \over k } m )  \,  \left( 1+ O({1\over k}) \right)
\end{eqnarray*}
see \equ{defg} for the definition of $g$.

\medskip
\noindent
{\bf The matrix $ H_\alpha$, for $\alpha = 3, \ldots , n$.} Fix $\alpha = 3$. The other dimensions can be treated in the same way.
The matrix  $H_3= (H_{3, ij} )_{i, j=1, \ldots , k}$ defined by
$$
H_{3, ij} = \int_{\R^n} L(Z_{3i} ) Z_{3j}
$$
is symmetric. We have
\begin{equation}
\label{H311}
H_{3, 11} =   \Xi \,   \mu^{n-2 \over 2} \,   \left[  \sum_{l>1}^k {-\cos \theta_l  \over (1-\cos \theta_l )^{n \over 2}} \right]  \, + \, O(\mu^{n\over 2} )
\end{equation}
and, for $l>1$,
\begin{equation}
\label{H31l}
H_{3, 1l}=  \, \Xi \,   \left[ {1 \over (1-\cos \theta_l )^{n\over 2}} \right] \, \mu^{n-2} \, + \, O(\mu^{n\over 2}).
\end{equation}

\medskip
\noindent
{\bf Eigenvalues for $H_3$.} For any $m=0, \ldots , k-1$
\begin{equation}
\label{h3m}
h_{3, m} =\, \Xi \,  \bar h_{3  ,m } \, \mu^{n-2}
\end{equation}
where
\begin{eqnarray}
\bar h_{3,m} &= & \left[
\sum_{l>1}^k { - \cos \theta_l +\cos m \theta_l \over
(1-\cos \theta_l )^{n \over 2} } \right] \, \left (1+ O({1\over k} ) \right) . \nonumber
\end{eqnarray}

\medskip
\section{Solving a linear system.}\label{lineare}

This section is devoted to solve system \equ{sistem3}, namely
$$
N  \left[ \begin{array}{r}  \bar {\textsf{c}}_0 \\
\bar { \textsf{c}}_1\\ \bar {\textsf{c}}_2
\end{array} \right] = \left[ \begin{array}{r} \bar { \textsf{s}}_0 \\
\bar{ \textsf{s}}_1\\ \bar {\textsf{s}}_2
\end{array} \right] , \quad H_\alpha \, [
 \bar {\textsf{c}}_\alpha ] = \bar {\textsf{s}}_\alpha \quad {\mbox {for}} \quad \alpha  = 3, \ldots , n.
$$
for a given right hand side $ \left[ \begin{array}{r} \bar { \textsf{s}}_0 \\
\bar{ \textsf{s}}_1\\ \bar {\textsf{s}}_2
\end{array} \right] \in \R^{3k}$, and $ \bar {\textsf{s}}_\alpha \in \R^k$,  where $N$ is the matrix defined in \equ{roma4} and $H_\alpha$ are the matrices defined in \equ{H}.

\medskip
\noindent

\medskip
\noindent
Let
\begin{equation}
\label{upsilon}
\Upsilon=  {( \sqrt{2} \pi )^{n-2}\over p \gamma {n-2 \over 2} \Xi },
\end{equation}
where $\Xi$ is defined in \equ{Xi}.
We have the validity of the following

\begin{proposition}\label{invsist}

Part a. \\ There exist $k_0$ and $C>0$  such that, for all $k > k_0 $ ,  System
\begin{equation}
\label{primo}
N  \left[ \begin{array}{r}  \bar {\textsf{c}}_0 \\
\bar { \textsf{c}}_1\\ \bar {\textsf{c}}_2
\end{array} \right] = \left[ \begin{array}{r} \bar { \textsf{s}}_0 \\
\bar{ \textsf{s}}_1\\ \bar {\textsf{s}}_2
\end{array} \right]
\end{equation}
is solvable if
\begin{equation} \label{solvable1}
\bar {\textsf{s}}_2 \cdot {\bf {1}}_k = ( \bar {\textsf{s}}_0 + \bar {\textsf{s}}_1 ) \cdot {\bf{cos}} =  ( \bar {\textsf{s}}_0 + \bar {\textsf{s}}_1 ) \cdot {\bf{sin}} = 0.
\end{equation}
Furthermore,  the solutions of System \equ{primo} has the form
\begin{equation}
\label{primo1}
\left[ \begin{array}{r}  \bar {\textsf{c}}_0 \\
\bar { \textsf{c}}_1\\ \bar {\textsf{c}}_2
\end{array} \right] = \left[ \begin{array}{r}  \bar {\textsf{w}}_0 \\
\bar { \textsf{w}}_1\\ \bar {\textsf{w}}_2
\end{array} \right] + t_1 \left[ \begin{array}{r}  0 \\
0 \\ {\bf{1}}_k
\end{array} \right] + t_2 \left[ \begin{array}{r}  {\bf{cos }} \\
-  { \bf{cos}}\\ 0
\end{array} \right] + t_3 \left[ \begin{array}{r}   {\bf{sin}} \\ -
 { \bf{sin}}\\ 0
\end{array} \right]
\end{equation}
for all $t_1 , t_2 , t_3 \in \R$, and with $ \left[ \begin{array}{r}  \bar {\textsf{w}}_0 \\
\bar { \textsf{w}}_1\\ \bar {\textsf{w}}_2
\end{array} \right]$ a fixed vector such that
\begin{equation}
\label{primo2}
\| \left[ \begin{array}{r}  \bar {\textsf{w}}_0 \\
\bar { \textsf{w}}_1\\ \bar {\textsf{w}}_2
\end{array} \right] \| \leq {C \over k^n \mu^{n-2}}  \left[ \begin{array}{r} \bar { \textsf{s}}_0 \\
\bar{ \textsf{s}}_1\\ \bar {\textsf{s}}_2
\end{array} \right]\| .
\end{equation}

\bigskip
Part b. \ \ Let $\alpha =3, \ldots , n$.
There exist $k_0$ and $C$ such that, for any $k > k_0$, system
\begin{equation}
\label{secondo}
 H_\alpha \, [
 \bar {\textsf{c}}_\alpha ] = \bar {\textsf{s}}_\alpha
\end{equation}
is solvable only if
\begin{equation}
\label{solvable2}
\bar { \textsf{s}}_\alpha \cdot {\bf{cos}} = \bar { \textsf{s}}_\alpha \cdot  {\bf{sin}} = 0.
\end{equation}
Furthermore,  the solutions of System \equ{secondo} has the form
\begin{equation}
\label{primo1}
  \bar {\textsf{c}}_\alpha
=   \bar {\textsf{w}}_\alpha
+ t_1    {\bf{cos }}
+ t_2   {\bf{sin }}
\end{equation}
for all $t_1 , t_2  \in \R$, and with $ \left[ \begin{array}{r}  \bar {\textsf{w}}_\alpha
\end{array} \right]$ a fixed vector such that
\begin{equation}
\label{primo2}
\| \left[ \begin{array}{r}  \bar {\textsf{w}}_\alpha
\end{array} \right] \| \leq {C \over k^n \mu^{n-2}}  \left[ \begin{array}{r} \bar { \textsf{s}}_\alpha
\end{array} \right]\| .
\end{equation}

\end{proposition}

\medskip
\noindent
\begin{proof}

\noindent {\it Part a}. \ \

Define $${\mathcal P} = \left[ \begin{array}{rrr} P & 0 & 0 \\
0 & P & 0 \\
0 & 0 & P
\end{array} \right]
$$
where $P$ is defined in \equ{diag2}, a simple algebra gives that
$$
 N = {\mathcal P} {\mathcal D} {\mathcal P}^T
$$
where
$$
{\mathcal D} = \left[ \begin{array}{rrr}  D_{A  } &  D_{B}  & D_{ C } \\
D_{B}  & D_{F} & D_{D} \\
D_{-C} & D_{-D } & D_{G }
\end{array}
\right].
$$
Here $D_X$ denotes the diagonal matrix of dimension $k \times k$ whose entrances are given by the eigenvalues of $X$. For instance
$$
D_{A } = {\mbox {diag}} \left( a_0 , a_1 , \ldots , a_{k-1} \right)
$$
where $a_j$ are the eigenvalues of the matrix $A$, defined in \equ{am}.
Using the change of variables
\begin{equation}
\label{changeofvariables}
 \left[ \begin{array}{r}  \bar {\textsf{y}}_0 \\
\bar { \textsf{y}}_1\\ \bar {\textsf{y}}_2
\end{array} \right]  ={\mathcal  P}^T  \left[ \begin{array}{r}  \bar {\textsf{c}}_0 \\
\bar { \textsf{c}}_1\\ \bar {\textsf{c}}_2
\end{array} \right] ; \quad  \left[ \begin{array}{r}  \bar {\textsf{s}}_0 \\
\bar { \textsf{s}}_1\\ \bar {\textsf{s}}_2
\end{array} \right]  = {\mathcal P}  \left[ \begin{array}{r}  \bar {\textsf{h}}_0 \\
\bar { \textsf{h}}_1\\ \bar {\textsf{h}}_2
\end{array} \right],
\end{equation}
with
$
  \bar {\textsf{y}}_\alpha =  \left[ \begin{array}{r}  y_{\alpha , 1} \\
y_{\alpha , 2} \\ \ldots \\ y_{\alpha , k}  \end{array}\right] ,  \bar {\textsf{h}}_\alpha =  \left[ \begin{array}{r}  h_{\alpha , 1} \\
h_{\alpha , 2} \\ \ldots \\ h_{\alpha , k}  \end{array}\right]  \in \R^k,
$ $\alpha = 0 , 1 , 2$,
one sees that solving
$$
N \left[ \begin{array}{r}  \bar {\textsf{c}}_0 \\
\bar { \textsf{c}}_1\\ \bar {\textsf{c}}_2
\end{array} \right] = \left[ \begin{array}{r} \bar { \textsf{s}}_0 \\
\bar{ \textsf{s}}_1\\ \bar {\textsf{s}}_2
\end{array} \right]
$$
is equivalent to solving
\begin{equation}
\label{roma7}
 {\mathcal D} \left[ \begin{array}{r}  \bar {\textsf{y}}_0 \\
\bar { \textsf{y}}_1\\ \bar {\textsf{y}}_2
\end{array} \right] = \left[ \begin{array}{r} \bar { \textsf{h}}_0 \\
\bar{ \textsf{h}}_1\\ \bar {\textsf{h}}_2
\end{array} \right] .
\end{equation}
Furthermore, observe that
\begin{equation}
\label{lunes3} \| \bar {\textsf{y}} _\alpha \| = \| \bar {\textsf{c}}_\alpha \|, \quad {\mbox {and}} \quad
\| \bar {\textsf{h}}_\alpha  \| = \| \bar {\textsf{s}}_\alpha  \|, \quad \alpha = 0, 1,2.
\end{equation}
Let us now introduce the matrix
$$
D = \left[ \begin{array}{rrrr}  D_0 & 0  & \ldots & 0 \\
0 & D_1 & 0  & \ldots \\
\ldots & \ldots & \ldots & \ldots \\
\ldots & 0 & 0 & D_{k-1}
\end{array}
\right]
$$
where for any $m = 0 , \ldots , k-1$, $D_m$ is the $3\times 3$ matrix given by
\begin{equation}
\label{fi1}
D_m = \left[ \begin{array}{rrr}  a_m & b_m & c_m \\
b_m & f_m & d_m \\
-c_m &- d_m & g_m
\end{array}
\right] =  \Xi \, \mu^{n-2} \,  \left[ \begin{array}{rrr}   \bar a_m &  \bar b_m & i  \bar c_m \\
\bar b_m &  \bar f_m &i   \bar d_m \\
- i  \bar c_m &-  i  \bar d_m &  \bar g_m
\end{array}
\right] \quad
\end{equation}
where $a_m$, $b_m$, $c_m$, $f_m$, $g_m$, $d_m$ are the eigenvalues of the matrices $A$, $B$, $C$, $F$, $G$ and $D$ respectively.
In the above formula we have used  the computation for the eigenvalues
$a_m$, $b_m$, $c_m$, $d_m$, $f_m$ and $g_m$ that we obtained in \equ{am}, \equ{bm}, \equ{cm}, \equ{dm},
\equ{fm} and \equ{gm}.

An easy argument implies that system \equ{roma7} can be re written in the form
\begin{equation}
\label{roma8}
D_m \left[ \begin{array}{r} y_{0, m+1} \\ y_{1,  m+1} \\ y_{2,m+1} \end{array} \right]
=  \left[ \begin{array}{r} h_{0, m+1} \\ h_{1, m+1} \\ h_{2, m+1} \end{array} \right]
\quad m=0, 1, \ldots , k-1.
\end{equation}
Taking into account that $\bar a_m = - \bar b_m$ and $\bar c_m = - \bar d_m$,
a direct algebraic manipulation of the system gives that \equ{roma8} reduces to the simplified system
\begin{equation}
\label{roma10}
   \left[ \begin{array}{rrr}-  \bar b_m &  0 & i  \bar c_m \\
0 & \bar  f_m + \bar b_m & 0 \\
-i  \bar c_m &  0  &  \bar g_m
\end{array}
\right] \left[ \begin{array}{r}  y_{0, m+1} -  y_{1, m+1}  \\  y_{1, m+1} \\
 y_{2, m+1} \end{array} \right] = {1\over \Xi \mu^{n-2}  }  \,  \left[ \begin{array}{r} h_{0, m+1} \\ h_{1, m+1} +  h_{0, m+1} \\ h_{2, m+1} \end{array} \right].
\end{equation}
Let, for any $m=0, \ldots , k-1$,
\begin{equation}
\label{determinante}
\ell_m :=- \left( \bar b_m + \bar f_m \right) \, \left[ \bar g_m \bar b_m + \bar c_m ^2 \right],
\end{equation}
being $\ell_m$ the determinant of the above matrix.

\medskip
\noindent
We have the following cases

\noindent
{\bf Case 1}. \ \
If $m=0$, we have that   $\bar g_0 = \bar c_0 =0$ and so $\ell_0 = 0$. Furthermore,
$$
\bar b_0 = {n-2 \over 2} {k^{n-2} \over (\sqrt{2} \pi )^{n-2} } g'' (0) \, \left( 1+ O({1\over k} ) \right)
$$
and
$$
\bar f_0 + \bar b_0 = - {k^{n-2} \over (\sqrt{2} \pi )^{n-2} } g'' (0) \, \left( 1+ O({1\over k} ) \right).
$$
We conclude that System \equ{roma10} for $m=0$ is solvable if
$$
h_{21} = 0
$$
and there exists a positive constant $C$, independent of $k$, such that  the solution has the form
$$
  \left[ \begin{array}{r}  y_{0, 1}   \\  y_{1, 1} \\
 y_{2, 1} \end{array} \right]  =   \left[ \begin{array}{r}  \hat y_{0, 1}   \\ \hat y_{1, 1} \\
 \hat y_{2, 1} \end{array} \right] + t   \left[ \begin{array}{r}  0   \\  0 \\
 1 \end{array} \right]
$$
for any $t \in \R$ and for a fixed vector $  \left[ \begin{array}{r}  \hat y_{0, 1}   \\ \hat y_{1, 1} \\
 \hat y_{2, 1} \end{array} \right] $ with
$$
\|  \left[ \begin{array}{r}  \hat y_{0, 1}   \\ \hat y_{1, 1} \\
 \hat y_{2, 1} \end{array} \right]  \| \leq {C \over \mu^{n-2 } k^{n-2}} \|  \left[ \begin{array}{r} h_{0, 1} \\ h_{1, 1} \\ h_{2, 1} \end{array} \right] \|.
$$

\medskip
\noindent
{\bf Case 2}. \ \  If $m=1$, we have that $\bar f_1 + \bar b_1 = 0$. By symmetry, for $m=k-1$ we also  have $\bar f_{k-1} + \bar b_{k-1} = 0$. Furthermore
$$
\bar b_1 = \bar b_{k- 1} =  {n-2 \over 2} {  k^{n-2 } \over (\sqrt{2} \pi )^{n-2} } g '' (0) \left( 1+ O({1\over k})  \right),
$$
$$
\bar g_1 = \bar g_{k-1} = - (n-1)  {  k^{n-2 } \over (\sqrt{2} \pi )^{n-2} } g '' (0) \left( 1+ O({1\over k} ) \right), \quad
$$
and
$$
\bar c_1 = - \bar c_{k-1} =  (n-2 ) {  k^{n-2 } \over (\sqrt{2} \pi )^{n-2} } g '' (0) \left( 1+ O({1\over k} ) \right).
$$
We conclude that System \equ{roma10} for $m=1$ is solvable if
$$
h_{02} + h_{12} = 0
$$
and  there exists a positive constant $C$, independent of $k$, such that the solution has the form
$$
  \left[ \begin{array}{r}  y_{0, 2}   \\  y_{1, 2} \\
 y_{2, 2} \end{array} \right]  =   \left[ \begin{array}{r}  \hat y_{0, 2}   \\ \hat y_{1, 2} \\
 \hat y_{2, 2} \end{array} \right] + t   \left[ \begin{array}{r}  1   \\  -1 \\
 0 \end{array} \right]
$$
for any $t \in \R$ and for a fixed vector $  \left[ \begin{array}{r}  \hat y_{0, 2}   \\ \hat y_{1, 2} \\
 \hat y_{2, 2} \end{array} \right] $ with
$$
\|  \left[ \begin{array}{r}  \hat y_{0, 2}   \\ \hat y_{1, 2} \\
 \hat y_{2, 2} \end{array} \right]  \| \leq {C \over \mu^{n-2 } k^{n-2}} \|  \left[ \begin{array}{r} h_{0, 2} \\ h_{1, 2} \\ h_{2, 2} \end{array} \right] \|.
$$
On the other hand, when $m=k-1$  System \equ{roma10}  is solvable if
$$
h_{0 , k} + h_{1, k} = 0
$$
and  there exists a positive constant $C$, independent of $k$, such that the solution has the form
$$
  \left[ \begin{array}{r}  y_{0, k}   \\  y_{1, k} \\
 y_{2, k} \end{array} \right]  =   \left[ \begin{array}{r}  \hat y_{0, k}   \\ \hat y_{1, k} \\
 \hat y_{2, k} \end{array} \right] + t   \left[ \begin{array}{r}  1   \\  - 1 \\
 0 \end{array} \right]
$$
for any $t \in \R$ and for a fixed vector $  \left[ \begin{array}{r}  \hat y_{0, k}   \\ \hat y_{1, k} \\
 \hat y_{2, k} \end{array} \right] $ with
$$
\|  \left[ \begin{array}{r}  \hat y_{0, k}   \\ \hat y_{1, k} \\
 \hat y_{2, k} \end{array} \right]  \| \leq {C \over \mu^{n-2 } k^{n-2}} \|  \left[ \begin{array}{r} h_{0, k} \\ h_{1, k} \\ h_{2, k} \end{array} \right] \|.
$$

\medskip
\noindent
{\bf Case 3}. \ \  Let now $m$ be $\not= 0, 1, k-1$. In this case we have
$$
\bar b_m =  {n-2 \over 2} {  k^{n-2 } \over (\sqrt{2} \pi )^{n-2} } g '' ({2\pi \over k} m ) \left( 1+ O({1\over k})  \right),
$$
$$
\bar f_m + \bar b_m =   {  k^{n } \over (\sqrt{2} \pi )^{n} } g '' ({2\pi \over k} m ) \left( 1+ O({1\over k})  \right),
$$
$$
\bar g_m =-  (n-1) {  k^{n } \over (\sqrt{2} \pi )^{n} } g ({2\pi \over k} m ) \left( 1+ O({1\over k})  \right),
$$
and
$$
\bar c_m =  {n-2 \over 2} {\sqrt{2}   k^{n-1 } \over (\sqrt{2} \pi )^{n-1} } g ' ( {2\pi \over k} m ) \left( 1+ O({1\over k})  \right).
$$
In particular
\begin{eqnarray*}
\ell_m &=& -  \, {n-2 \over 2} \, {k^{3n -2} \over (\sqrt{2} \pi )^{3n-2}} \, g '' ({2\pi \over m}) \, \times \,  \\
& &  \left[- (n-1) g({2\pi \over k } m )
g'' ({2\pi \over k } m ) + (n-2) ( g'({2\pi \over k } m ) )^2 \right] \,  \left( 1+ O({1\over k})  \right)
\end{eqnarray*}
Thus under condition \equ{condig}, we have that
$$\ell_m <0 \quad \forall m=2, \ldots , k-2.
$$
Hence, for all $m\not= 0, 1, k-1$, System \equ{roma10} is uniquely solvable and  there exists a positive constant $C$, independent of $k$, such that  the solution
$  \left[ \begin{array}{r}  \hat y_{0, 1}   \\ \hat y_{1, 1} \\
 \hat y_{2, 1} \end{array} \right] $ satisfies
$$
\|  \left[ \begin{array}{r}  \hat y_{0, 1}   \\ \hat y_{1, 1} \\
 \hat y_{2, 1} \end{array} \right]  \| \leq {C \over \mu^{n-2 } k^{n}} \|  \left[ \begin{array}{r} h_{0, 1} \\ h_{1, 1} \\ h_{2, 1} \end{array} \right] \|.
$$

\medskip
\noindent
Going back to the original variables, and applying a fixed point argument for contraction mappings we get the validity of Part a of
 Proposition \ref{invsist}.

\bigskip
\noindent
{\it Part b.} \ \  Fix $\alpha = 3, \ldots , n$. We have
$$
 H_\alpha  = P D_\alpha P^T
$$
where $P$ is defined in \equ{diag2},
and
$$
D_{\alpha } = {\mbox {diag}} \left( h_{\alpha , 0}  , h_{\alpha ,1 }, \ldots , h_{\alpha , k-1} \right)
$$
where $h_{\alpha , j}$ are the eigenvalues of the matrix $H_\alpha $, defined in \equ{h3m}.
Using the change of variables $  \bar {\textsf{y}}_\alpha = P^T {\textsf{c}}_\alpha$ and $
 \bar {\textsf{s}}_\alpha = P^T {\textsf{h}}_\alpha$, we have to solve $ D_\alpha \textsf{y}_\alpha = {\textsf{h}}_\alpha$.

Recall that, for any $m=0, \ldots , k-1$
$$
h_{\alpha , m} =\, \Xi \,  \bar h_{\alpha  ,m } \, \mu^{n-2}
$$
where
$$
\bar h_{\alpha ,m} =  \left[
\sum_{l>1}^k { - \cos \theta_l +\cos m \theta_l \over
(1-\cos \theta_l )^{n \over 2} } \right] \, \left (1+ O({1\over k} ) \right) .
$$
If $m=1$ or $m=k-1$, we have that $ \sum_{l>1}^k { - \cos \theta_l +\cos m \theta_l \over
(1-\cos \theta_l )^{n \over 2} }   = 0  $, so the system is solvable only if
$h_{\alpha , 2} = h_{\alpha , k-1} = $. On the other hand
we have
$$
h_{\alpha , 0} = \Xi \mu^{n-2} {k^{n-2} \over (\sqrt{2} \pi )^{n-2}}  \left( 1+ O({1\over k} ) \right)
$$
and for $m=2, \ldots , k-2$
$$
h_{\alpha , m} = \Xi \mu^{n-2} {k^{n} \over (\sqrt{2} \pi )^{n}} g ({2\pi \over k } m ) \left( 1+ O({1\over k} ) \right)
$$
\medskip
\noindent
Going back to the original variables, we get the validity of Part b, and this concludes the proof of
 Proposition \ref{invsist}.
\end{proof}

\medskip
\medskip
\section{Proof of Proposition \ref{danilo1}} \label{final0}

A key ingredient to prove Proposition \ref{danilo1}  is the estimates on the right hand sides of sistems  \equ{sistem3}.
We have

\begin{proposition}\label{rhside}
There exists a positive constant $C$ such that, for any $\alpha =0, 1, \ldots , n$,
\begin{equation}
\label{nacho1}
\|  \bar { \textsf{r}}_\alpha \| \leq C  \, \mu^{n-2 \over 2} \, \| \varphi^\perp \|_*
\end{equation}
for any $k$ sufficiently large.
\end{proposition}

\medskip
\noindent
\begin{proof} \  \ We prove \equ{nacho1}, only for $\alpha = 0$.

Recall that
$$
\bar { \textsf{r}}_0  =  \left[ \begin{array}{r} \int_{\R^n} L(\varphi^\perp ) Z_{01}\\
\ldots \\
 \int_{\R^n} L(\varphi^\perp  ) Z_{0k}
\end{array} \right]  .
$$
Then estimate \equ{nacho1} will follows from
\begin{equation}
\label{nacho11}
\left| \int_{\R^n } L(\varphi^\perp ) Z_{0j} \right| \leq C \, \mu^{n-2 \over 2} \, \| \varphi \|_* ,
\end{equation}
 for any $j=1, \ldots , k$. To prove \equ{nacho11}, we fix $j=1$ and we write
\begin{eqnarray*}
\int_{\R^n } L(\varphi^\perp ) Z_{01} \, dx &=& \int_{\R^n } L(Z_{01} ) \varphi^\perp \\
&=&   \int_{\R^n \setminus \cup {B(\xi_j , {\eta \over k^{1+\sigma} } ) }} L(Z_{01} ) \varphi^\perp + \sum_{j=1}^k \int_{B(\xi_j , {\eta \over k^{1+\sigma} } )}  L(Z_{01} ) \varphi^\perp
\end{eqnarray*}
where $\eta$ and $\sigma$ are small positive numbers, independent of $k$.

We start to estimate $\int_{B(\xi_1 , {\eta \over k^{1+\sigma}} )} L(Z_{01} ) \varphi^\perp$. We have
$L(Z_{01}) = [ f' (u) - f'(U_1)] \, Z_{01}$.
As we have already observed very close to $\xi_1$, $U_1 (x ) = O(\mu^{-{n-2 \over 2} })$ and so in $B(\xi_1 , {\eta \over k^{1+\sigma}}) $ the function  $U_1$ dominates globally the other terms, provided $\eta$ is chosen small enough.
Thus, after the change of variable $x=\xi_1 + \mu y$,
\begin{eqnarray*}
 \left|    \int_{B(\xi_1 , {\eta \over k^{1+\sigma} } )}  L(Z_{01} ) \varphi^\perp \right| & \leq &
C
\int_{B(0, {\eta \over k^{1+\sigma } \mu } ) }  f'' (U) | \Upsilon (y) | Z_0 (y) [\mu^{{n-2 \over 2}} | \varphi^\perp (\xi_1 + \mu y  ) | ]  \, dy \\
&\leq & C   \| \varphi^\perp \|_*
\int_{B(0, {\eta \over k^{1+\sigma } \mu } ) }  f'' (U) | \Upsilon (y) | Z_0 (y)  \, dy \\
\end{eqnarray*}
where
$$
\Upsilon (y) =  \mu^{n-2 \over 2} U(\xi_1 + \mu y )
+\sum_{l\not= 1} U(y + \mu^{-1} (\xi_1 - \xi_l ) )
$$
A direct consequence of \equ{rivoli1} is then that
$$
 \left|    \int_{B(\xi_1 , {\eta \over k^{1+\sigma} } )}  L(Z_{01} ) \varphi^\perp \right|  \leq
C   \mu^{n-2 \over 2} \| \varphi^\perp \|_* .
$$
Let now $j\not= 1$ and consider $ \int_{B(\xi_j , {\eta \over k^{1+\sigma} } )}  L(Z_{01} ) \varphi^\perp$.
In this case, after the change of variables $x= \xi_j + \mu y $, we get
\begin{eqnarray*}
& & \left| \int_{B(\xi_j , {\eta \over k^{1+\sigma} } )}  L(Z_{01} ) \varphi^\perp \right| \\
&\leq &  C
\int_{B(0, {\eta \over \mu k^{1+\sigma}} ) } U^{p-1} Z_1 (y + \mu^{-1} (\xi_1 - \xi_j ) ) [\mu^{-{n-2 \over 2}} \varphi^\perp (\xi_j + \mu y )] \\
&\leq & C  \| \varphi^\perp \|_* \left( \int_{\R^n } U^{p-1} {1\over (1+ |y|)^{n-2} } \right) \, {\mu^{n-2} \over (1-\cos \theta_j )^{n-2 \over 2}}
\end{eqnarray*}
where we used \equ{man3}. Thus we estimate
$$
 \left|  \sum_{j>1}  \int_{B(\xi_j , {\eta \over k^{1+\sigma} } )}  L(Z_{01} ) \varphi^\perp \right|
\leq C  \mu^{n-2 \over 2} \| \varphi^\perp \|_* .
$$
Finally, in the exterior region $\R^n \setminus \cup B(\xi_j , {\eta \over k^{1+\sigma}}) $
we can estimate
\begin{eqnarray*}
 \left| \int_{ \R^n \setminus \cup B(\xi_j , {\eta \over k^{1+\sigma} } ) }  L(Z_{01 }) \varphi^\perp \right|
& \leq & C \| \varphi^\perp \|_*  \int_{ \R^n \setminus \cup B(\xi_j , {\eta \over k^{1+\sigma} } )}  {U^{p-1} \over (1+ |y| )^{n-2}} Z_{01} (y) \, dy \\
&\leq & C \mu^{n\over 2} \| \varphi^\perp \|_*.
\end{eqnarray*}
Thus we have proven \equ{nacho1} for $\alpha =0$. The other cases can be treated similarly.
\end{proof}

\bigskip
We have now the tools for the

\medskip
\noindent
{\bf Proof of Proposition \ref{danilo1}}. \ \ System \equ{sistem2} is solvable only if the following orthogonality conditions are satisfied:
\begin{equation}
\label{gabi1}
 \left[ \begin{array}{r}   \bar {\textsf{r}}_0 \\
 \bar {\textsf{r}}_1\\  \bar {\textsf{r}}_2
\end{array} \right] \cdot \left[ \begin{array}{r}  1 \\
- \bf{1}_k \\ 0 \\ - \bf{1}_k \\ 0 \\ 0
\end{array} \right]=  \left[ \begin{array}{r}   \bar {\textsf{r}}_0 \\
  \bar {\textsf{r}}_1\\  \bar {\textsf{r}}_2
\end{array} \right] \cdot  \left[ \begin{array}{r} 0\\ 0\\ 1 \\
-{1\over \sqrt{1-\mu^2}} \bf{cos } \\ 0 \\ {1\over \sqrt{1-\mu^2}} \bf{sin}
\end{array} \right]  = \left[ \begin{array}{r}   \bar {\textsf{r}}_0 \\
  \bar {\textsf{r}}_1\\  \bar {\textsf{r}}_2
\end{array} \right] \cdot \left[ \begin{array}{r}  0\\ 0\\ 0\\ -{1\over \sqrt{1-\mu^2}} \bf{sin} \\ 1 \\
-{1\over \sqrt{1-\mu^2}} \bf{cos}
\end{array} \right] =0,
\end{equation}
\begin{equation}\label{gabi2}
 \left[ \begin{array}{r}   \bar {\textsf{r}}_0 \\
  \bar {\textsf{r}}_1\\  \bar {\textsf{r}}_2
\end{array} \right] \cdot
\left[ \begin{array}{r}  0 \\
0 \\ 0 \\ \\ 0 \\ \bf{1}_k
\end{array} \right] =  \left[ \begin{array}{r}   \bar {\textsf{r}}_0 \\
  \bar {\textsf{r}}_1\\  \bar {\textsf{r}}_2
\end{array} \right] \cdot \left[ \begin{array}{r}  0 \\
 \bf{cos} \\ 0 \\ - \bf{cos} \\ 0 \\ 0
\end{array} \right] =  \left[ \begin{array}{r}   \bar {\textsf{r}}_0 \\
  \bar {\textsf{r}}_1\\  \bar {\textsf{r}}_2
\end{array} \right] \cdot  \left[ \begin{array}{r}  0 \\
 \bf{sin} \\ 0 \\ - \bf{sin} \\ 0 \\ 0
\end{array} \right]=0
\end{equation}
and
\begin{equation}
\label{gabi3}
 \bar {\textsf{r}}_\alpha \cdot  \left[ \begin{array}{r}  1 \\
- \bf{1}_k
\end{array} \right] =  \bar {\textsf{r}}_\alpha \cdot   \left[ \begin{array}{r}  0 \\
\bf{cos}
\end{array} \right] =  \bar {\textsf{r}}_\alpha \cdot  \left[ \begin{array}{r}  0 \\
 \bf{sin}
\end{array} \right] =0  \quad \alpha = 3, \ldots ,n
\end{equation}
We recall that $ \bar {\textsf{r}}_\alpha =  \left[ \begin{array}{r}
  \int_{\R^n}  L(\varphi^\perp) Z_{\alpha , 1} \\.. \\   \int_{\R^n}  L(\varphi^\perp) Z_{\alpha , k}
\end{array}
\right].
$
As we already mentioned at the beginning of Section \ref{simple}, the orthogonality conditions \equ{gabi1} are satisfied as consequence of \equ{decr0}, \equ{decr00} and \equ{decr000}. Similarly, the first orthogonality condition in \equ{gabi3} is satisfied as consequence of \equ{decr0000}.

\medskip
Let us recall from \equ{due} that
$$
 L(  \varphi^\perp  )  = -  \sum_{\alpha = 0}^n \sum_{l=0}^k c_{\alpha l } L(Z_{\alpha , l}  ).
$$
Thus the function $x \to  L(  \varphi^\perp  )  (x)$ is invariant under rotation  of angle ${2\pi \over k}$ in the $(x_1 , x_2)$-plane. Thus
$$
0= \sum_{l=1}^k \int L(\varphi^\perp ) Z_{2l} (x) \, dx  =  \bar {\textsf{r}}_2 \cdot \bf{1}_k
$$
and, for all $\alpha = 3, \ldots , n$,
$$
 \sum_{l=1}^k \cos \theta_l \int L(\varphi^\perp ) Z_{\alpha l} (x) \, dx  =\left(  \int L(\varphi^\perp ) Z_{\alpha 1} (x) \, dx   \right) \left(  \sum_{l=1}^k \cos \theta_l \right) =  0,
$$
thus $\bar {\textsf{r}}_\alpha \cdot \bf{cos}= 0
$, and similarly
$$
0 = \sum_{l=1}^k \sin \theta_l \int L(\varphi^\perp ) Z_{\alpha l} (x) \, dx  =  \bar {\textsf{r}}_\alpha \cdot \bf{sin}
$$
namely the first orthogonality condition in \equ{gabi2} and the remaining orthogonality conditions in \equ{gabi3} are satisfied. Let us check that also the last two orthogonality conditions in \equ{gabi2} are verified.

Observe that  $L( \varphi^\perp ) (x) = |x|^{-2-n} \, L( \varphi^\perp ) ({x \over |x|^2})  $. The remaining orthogonality conditions in \equ{gabi2} are consequence of the following

\begin{lemma}\label{mm}
Let $h$ be a function in $\R^n$ such that $h(y) =  |y|^{-n-2} h({y \over |y|^2} )$. Then
\begin{equation}
\label{mm1}
\mu \int_{\R^n} {\partial \over \partial \mu} \left( U_\mu (x-\xi_l ) \right)  h(y) \, dy =\xi_l \cdot
\int_{\R^n} \nabla U_\mu (x-\xi_l ) h (y) \, dx
\end{equation}
\end{lemma}

We postpone the proof of the above Lemma to the end of this Section.

\bigskip
Combining the result of Proposition \ref{invsist} and the a-priori estimates in Proposition \ref{rhside}, a direct application of a fixed point theorem for contraction mapping readily gives the proof of Proposition \ref{danilo1}.

\bigskip
\medskip
\noindent
We conclude this section with

\medskip

\begin{proof}[Proof of Lemma \ref{mm}] \ \

{\it Proof of \equ{mm1}.} \ \ Assume $l=1$.
Define
$$
I (t) = \int_{\R^n} \omega_\mu (y-t\xi_1 ) h(y) \, dy
\quad {\mbox  {where}} \quad
\omega_\mu (y-t \, \xi_1 ) = \mu^{-{n-2 \over 2}} U({y- t \, \xi_1 \over \mu } ).
$$
We have
\begin{equation}
\label{mm2}
{d \over dt} I(t) = - \int_{\R^n } \nabla \omega_\mu (y-t\, \xi_1 ) \cdot \xi_1 \, h(y) \, dy,
\end{equation}
and
$$
 \left( {d \over dt} I(t) \right)_{t=1} =  - \int_{\R^n } \nabla \omega_\mu (y- \xi_1 ) \cdot \xi_1 \, h(y) \, dy.
$$
On the other hand, using the change of variables $y= {x\over |x|^2}$, we have
$$
I(t) = \int_{\R^n} \omega_\mu ({x \over |x|^2} - t \, \xi_1 ) h ({x \over |x|^2} ) |x|^{-2n} \, dx =
 \int_{\R^n} \omega_\mu ({x \over |x|^2} - t \, \xi_1 ) h (x ) |x|^{2-n} \, dx =
$$
$$
= \int_{\R^n} \omega_{\bar \mu } (x - \bar p  ) h (x)  \, dx
$$
where
$$
\bar \mu (t) = {\mu \over \mu^2 + t^2 |\xi_1 |^2 } , \quad \bar p (t) =  {t \over \mu^2 + t^2 |\xi_1 |^2 }  \, \xi_1.
$$
Observe that $\bar \mu (1) = \mu$, $\bar p (1) = \xi_1$,
$$
{d \over dt} \bar \mu (t) =  {-2t \mu \over \mu^2 + t^2 |\xi_1 |^2 }
, \quad {d \over dt} \bar p (t) = \left[  {1  \over \mu^2 + t^2 |\xi_1 |^2 } - {2t^2 |\xi_1|^2  \over \mu^2 + t^2 |\xi_1 |^2 } \right]
\xi_1.
$$
Hence
$$
{d \over dt} I(t)  = {d \over dt} \bar \mu (t)  \int_{\R^n} {\partial \over \partial \bar \mu} \omega_{\bar \mu} (x-\bar p ) h(x) \, dx -
 {d \over dt} \bar p (t) \int_{\R^n} \nabla \omega_{\bar \mu } (x-\bar p ) h (x) \, dx.
$$
This gives
$$
 \left( {d \over dt} I(t) \right)_{t=1} = -2\mu |\xi_1 |^2 \int_{\R^n } {\partial \over \partial \mu } \omega_\mu (x-\xi_1 ) h (x) \, dx
$$
\begin{equation}
\label{mm3}
- (1-2|\xi_1 |^2 ) \int_{\R^n} \nabla \omega_\mu (x-\xi_1 ) \cdot \xi_1 \, h (x) \, dx.
\end{equation}
>From \equ{mm2} and \equ{mm3} we conclude with the validity of \equ{mm1}.

\medskip
If $l>1$ in \equ{mm1}, the same arguments hold true. The thus conclude with the proof of the Lemma.
\end{proof}

\medskip
\medskip
\section{Final argument.} \label{final}

\medskip

Let
$
\left[ \begin{array}{r}  \bar {\textsf{c}}_0 \\
\bar { \textsf{c}}_1\\ \bar {\textsf{c}}_2 \\
\bar {\textsf{c}}_3 \\
\ldots \\
\bar {\textsf{c}}_n
\end{array} \right] $ be the solution to \equ{sistem2} predicted by Proposition \ref{danilo1},  given by
$$
 \left[ \begin{array}{r}  \textsf{c}_0 \\
 \textsf{c}_1\\ \textsf{c}_2
\end{array} \right]
=  \left[ \begin{array}{r}   \textsf{v}_0 \\
  \textsf{v}_1\\ \ \textsf{v}_2
\end{array} \right]
+ s_1 \left[ \begin{array}{r}  1 \\
- \bf{1}_k \\ 0 \\ - \bf{1}_k \\ 0 \\ 0
\end{array} \right]
+ s_2  \left[ \begin{array}{r} 0\\ 0\\ 1 \\
-{1\over \sqrt{1-\mu^2}} \bf{cos } \\ 0 \\ {1\over \sqrt{1-\mu^2}} \bf{sin}
\end{array} \right] + s_3\left[ \begin{array}{r}  0\\ 0\\ 0\\ -{1\over \sqrt{1-\mu^2}} \bf{sin} \\ 1 \\
-{1\over \sqrt{1-\mu^2}} \bf{cos}
\end{array} \right]
$$
$$
+s_4 \left[ \begin{array}{r}  0 \\
0 \\ 0 \\ \\ 0 \\ \bf{1}_k
\end{array} \right] + s_5 \left[ \begin{array}{r}  0 \\
 \bf{cos} \\ 0 \\ - \bf{cos} \\ 0 \\ 0
\end{array} \right] + s_6 \left[ \begin{array}{r}  0 \\
 \bf{sin} \\ 0 \\ - \bf{sin} \\ 0 \\ 0
\end{array} \right]
$$
and
$$
\textsf{c}_\alpha = \textsf{v}_\alpha + s_{\alpha 1} \left[ \begin{array}{r}  1 \\
- \bf{1}_k
\end{array} \right] + s_{\alpha 2} \left[ \begin{array}{r}  0 \\
\bf{cos}
\end{array} \right] + s_{\alpha 3} \left[ \begin{array}{r}  0 \\
 \bf{sin}
\end{array} \right] , \quad \alpha = 3, \ldots ,n
$$
A direct computation shows that there exists a unique
$$( s_1^*  , \ldots , s_6^* , s_{3, 1}^* , s_{3,2}^*, s_{3,3}^* , \ldots , s_{n,1}^*, s_{n,2}^*, s_{n,3}^* ) \in \R^{2n}$$
for which the above solution satisfies all the $2n$ conditions of Proposition \ref{danilo}.
Furthermore, one can see that
$$
\| ( s_1^*  , \ldots , s_6^* , s_{3, 1}^* , s_{3,2}^*, s_{3,3}^* , \ldots , s_{n,1}^*, s_{n,2}^*, s_{n,3}^* )  \| \leq C \sqrt{\mu} \| \varphi^\perp \|_*.
$$

\medskip
Hence, there exists a unique solution $\left[ \begin{array}{r}  \bar {\textsf{c}}_0 \\
\bar { \textsf{c}}_1\\ \bar {\textsf{c}}_2 \\
\bar {\textsf{c}}_3 \\
\ldots \\
\bar {\textsf{c}}_n
\end{array} \right]$ to systems \equ{sistem3}, satisfying estimates in Proposition \ref{danilo}. Furthermore, one has
$$
\| \left[ \begin{array}{r}  \bar {\textsf{c}}_0 \\
\bar { \textsf{c}}_1\\ \bar {\textsf{c}}_2 \\
\bar {\textsf{c}}_3 \\
\ldots \\
\bar {\textsf{c}}_n
\end{array} \right] \| \leq C \| \varphi^\perp \|_*
$$
for some positive constant $C$ independent of $k$.
On the other hand, from \equ{fattouno} we conclude that
\begin{equation}
\label{mazza6}
\| \varphi^\perp \|_* \leq C \mu^{1\over 2} \, \| \left[ \begin{array}{r}  \bar {\textsf{c}}_0 \\
\bar { \textsf{c}}_1\\ \bar {\textsf{c}}_2 \\
\bar {\textsf{c}}_3 \\
\ldots \\
\bar {\textsf{c}}_n
\end{array} \right] \|
\end{equation}
where again $C$ denotes a positive constant, independent of $k$. Thus we conclude that
$$
c_{\alpha , j} = 0 , \quad {\mbox {for all}} \quad \alpha = 0 , 1 , \ldots , n, \quad j=0, \ldots , k.
$$
Plugging this information into \equ{mazza6}, we conclude that $\varphi^\perp \equiv 0$ and this proves Theorem \ref{teo}.

\section{Proof of Proposition \ref{danilo}} \label{proof1}

The key ingredient to prove Proposition \ref{danilo} are the folllowing estimates
\begin{eqnarray}\label{made}
\int |u|^{p-1} Z_{\alpha , l} Z_0 &=& \int U^{p-1} Z_0^2 \, dy + O(\mu^{n-2 \over 2} ) \quad {\mbox {if}} \quad \alpha=0, l=0 \nonumber \\
&=& O(\mu^{n-2 \over 2} )  \quad {\mbox {otherwise}}
\end{eqnarray}
\begin{eqnarray}\label{made1}
\int |u|^{p-1} Z_{\alpha , l} Z_\beta &=& \int U^{p-1} Z_1^2 \, dy + O(\mu^{n-2 \over 2} ) \quad {\mbox {if}} \quad \alpha=\beta, l=0 \nonumber \\
&=& O(\mu^{n-2 \over 2} )  \quad {\mbox {otherwise}}
\end{eqnarray}
\begin{eqnarray}\label{made2}
\int |u|^{p-1} Z_{\alpha , l} Z_{0,j} &=& \int U^{p-1} Z_0^2 \, dy + O(\mu^{n-2 \over 2} ) \quad {\mbox {if}} \quad \alpha=0, l=j \nonumber \\
&=& O(\mu^{n-2 \over 2} )  \quad {\mbox {otherwise}}
\end{eqnarray}
\begin{eqnarray}\label{made3}
\int |u|^{p-1} Z_{\alpha , l} Z_{\beta , j} &=& \int U^{p-1} Z_1^2 \, dy + O(\mu^{n-2 \over 2} ) \quad {\mbox {if}} \quad \alpha=\beta, l=j \nonumber \\
&=& O(\mu^{n-2 \over 2} )  \quad {\mbox {otherwise}}
\end{eqnarray}
We prove \equ{made2}.

Let $\eta >0$ be a small number, fixed independently from $k$. We write
\begin{eqnarray*}
\int |u|^{p-1} Z_{\alpha l } Z_{0 j} &=&   \int_{ B(\xi_l , {\eta \over k} ) }| u| ^{p-1} Z_{\alpha l} Z_{0 l}  + \int_{\R^n \setminus B(\xi_l , {\eta \over k} )}  | u|^{p-1} Z_{\alpha l}  Z_{0 , j}
\\
&=& i_1 + i_2 .
\end{eqnarray*}
We claim that the main term is $i_1$.
Performing the change of variable $x= \xi_l + \mu y$, we get
\begin{eqnarray*}
i_1 &= &  \int_{B(0, {\eta \over \mu k })}   |u|^{p-1} (\xi_l + \mu y ) Z_\alpha (y) Z_0 (y) dy  \\
&= &    \left( \int U^{p-1} Z_0^2 + O( (\mu k)^{n})  \right)\quad {\mbox {if}} \quad \alpha = 0 \\
&=& 0 \quad {\mbox {if}} \quad \alpha \not= 0.
\end{eqnarray*}
 On the other hand, to estimate $i_2$, we write
$$
i_2 = \int_{\R^n \setminus \bigcup_{j=1}^k B(\xi_j , {\eta \over k} )} |u|^{p-1}  Z_{\alpha l} Z_{0,j}  + \sum_{j\not= l} \int_{ B(\xi_j , {\eta \over k} )}  u^{p-1} Z_{\alpha l} Z_{0,j} = i_{21} + i_{22}
$$
The first integral can be estimated as follows
$$
|i_{21}| \leq C  \int_{\R^n \setminus \bigcup_{j=1}^k B(\xi_j , {\eta \over k} )} {\mu^{n+2 \over 2} \over |x-\xi_l |^{n-2} } {1\over (1+ |x|)^{n+2}} \, dx \leq C \mu^{n-2 \over 2}
$$
while the second integral can be estimated by
$$
|i_{22}| \leq C \sum_{j\not= l} \int_{B(\xi_j , {\eta \over k})} {\mu^{n-2 \over 2} \over |x-\xi_l |^{n} }  |u|^{p-1} Z_{0j} \, dx \leq C \mu^{n-2 \over 2}
$$
where again $C$ denotes an arbitrary positive constant, independent of $k$. This concludes the proof of \equ{made2}. The proofs of \equ{made}, \equ{made1} and \equ{made3} are similar, and left to the reader.

\bigskip
Now we claim that
\begin{equation}
\label{assurdo} \int U^{p-1} Z_0^2 = \int U^{p-1} Z_1^2 = 2^{n-4 \over 2} \, n \, (n-2)^2 {\Gamma ({n\over 2} )^2 \over \Gamma (n+2)}.
\end{equation}
The proof of identity \equ{assurdo} is postponed to the end of this section.

\bigskip
Let us now consider \equ{sun1} with $\beta=0$, that is
$$
 \sum_{\alpha=0}^n \sum_{l=0}^k c_{\alpha l} \int Z_{\alpha l } u^{p-1} z_0 = -\int \varphi^\perp u^{p-1} z_0.
$$
First we write $t_0 = -{1 \over  \int U^{p-1} Z_0^2 } \int \varphi^\perp u^{p-1} z_0.$ A straightforward computation gives that
$|t_0 | \leq C \| \varphi^{\perp} \|_*$, for a certain constant $C$ independent from $k$.
Second, we observe that, direct consequence of \equ{made} -- \equ{made3}, of \equ{ang1} and Proposition \ref{prop1}  is that
\begin{eqnarray*}
 \sum_{\alpha=0}^n \sum_{l=0}^k c_{\alpha l} \int Z_{\alpha l } u^{p-1} z_0 &=&
c_{00} \int U^{p-1} Z_0^2 \\
&-& \sum_{l=1}^k \left[ c_{0l}  \int U^{p-1} Z_0^2 - c_{1l}  \int U^{p-1} Z_1^2 \right] \\
&+& O(k^{-{n\over q}} ) {\mathcal L}( \left[ \begin{array}{r} \bar{\textsf{c}}_0 \\
\bar{ \textsf{c}}_1 \\
\ldots \\
\bar{\textsf{c}}_n \end{array} \right]) + O(k^{1-{n\over q}} ) \hat {{\mathcal L}} ( \left[ \begin{array}{r} c_{00} \\
c_{10} \\
\ldots \\
c_{n0} \end{array} \right])
\end{eqnarray*}
where ${\mathcal L}$ and $\hat{ {\mathcal L}}$ are linear function, whose coefficients are uniformly bounded in $k$, as $k \to \infty$. Here we have used the fact that there exists a positive constant $C$ independent of $k$ such that
$$
\left|\int |u|^{p-1} Z_{\alpha l} \pi_0 (x) \, dx \right| \leq C \| \hat \pi_0 \|_{n-2}
$$
and
$$
\left|\int |u|^{p-1} Z_{\alpha l} \pi_0 (x) \, dx \right| \leq C \| \hat \pi_{01} \|_{n-2},
$$
together with the result in Proposition \ref{prop1}.
The condition \equ{len1} follows readily.
The proof of \equ{len2} -- \equ{len8} is similar to that performed above, and we leave it to the reader.

\medskip
\noindent
We conclude this section with the proof of \equ{assurdo}.
Using the definition of $Z_0$ and $Z_1$, we have that
$$
\int U^{p-1} Z_1^2 = a_n \, {(n-2 )^2 \over n} \, \int {|x|^2 \over (1+ |x|^2)^{n+2} } \, dx
$$
and
 $$
\int U^{p-1} Z_0^2 = a_n \, {(n-2 )^2 \over 4} \, \int {(1-|x|)^2 \over (1+ |x|^2)^{n+2} } \, dx,
$$
for a certain positive number $a_n$ that depends only on $n$.
Using the formula
$$
\int_0^\infty \left( {r \over 1+r^2 } \right)^q {1\over r^{1+\alpha} } \, dr =
{\Gamma ({q+\alpha \over 2} ) \Gamma ({q-\alpha \over 2} ) \over 2 \Gamma ({q} )}
$$
we get
\begin{equation}
\label{gigi1}
\int {1 \over (1+ |x|^2 )^{n+2}} \, dx = { {n\over 2} ({n\over 2} +1) \Gamma ({n\over 2} )^2 \over 2 \Gamma (n+2)},
\end{equation}
\begin{equation}
\label{gigi2}
\int {|x|^2 \over (1+ |x|^2 )^{n+2}} \, dx = { ({n\over 2})^2  \Gamma ({n\over 2} )^2 \over 2 \Gamma (n+2)},
\end{equation}
\begin{equation}
\label{gigi3}
\int {1 \over (1+ |x|^2 )^{n+2}} \, dx = { {n\over 2} ({n\over 2} +1) \Gamma ({n\over 2} )^2 \over 2 \Gamma (n+2)}.
\end{equation}
Replacing \equ{gigi1}, \equ{gigi2} and \equ{gigi3} in $\int U^{p-1} Z_1^2 $ and $\int U^{p-1} Z_0^2 $
we obtain
\begin{eqnarray*}
\int U^{p-1} Z_1^2  - \int U^{p-1} Z_0^2 &= & (n-2)^2 \, a_n \,  { {n\over 2} \Gamma ({n\over 2} )^2 \over 2 \Gamma (n+2)} \, \times \\
& &  \left[ {1\over 2} - {1\over 4} ({n\over 2} +1) + {n\over 4} - {1\over 4} ({n\over 2} +1) \right] =0,
\end{eqnarray*}
thus \equ{assurdo} is proven.

\section{Proof of \equ{fattouno}.} \label{proof2}

We start with the following

\begin{proposition}
\label{gino}
Let
$$
L_0 (\phi ) = \Delta \phi + p \gamma U^{p-1} \phi + a(y) \phi \quad {\mbox {in}} \quad \R^n.
$$
Assume that $a \in L^{n \over 2} (\R^n )$. Assume furthermore that $h$ is a function in $\R^n$ with
$\| h \|_{L^{2n \over n+2} (\R^n )}$ bounded and such that $|y|^{-n-2} h(|y|^{-2} y ) = \pm h(y) $. Then there exists a positive constant $C$ such that any solution $\phi $
to
\begin{equation}
\label{gino0}
L_0 (\phi ) = h
\end{equation}
satisfies
$$
\| \phi \|_{n-2} \leq C \| h \|_{**}.
$$
\end{proposition}

\begin{proof} Since $a \in L^{n \over 2} (\R^n)$ and $U^{p-1} = O(1+ |y|^4 )$, the operator $L_0$ is a compact perturbation
of the Laplace operator in the space $D^{1,2} (\R^n )$. Thus standard argument gives that
$$
\| \nabla \phi \|_{L^2 (\R^n )} + \| \phi \|_{L^{2n \over n-2} (\R^n )} \leq C \| h \|_{L^{2n \over n+2} (\R^n ) }
$$
where the last inequality is a direct consequence of Holder inequality. Being $\phi$ a weak solution to \equ{gino0}, local elliptic estimates yields
$$
\| D^2 \phi \|_{L^q (B_1) } + \| D \phi \|_{L^q (B_1) } + \| \phi \|_{L^\infty (B_1) } \leq  C \| h \|_{L^{2n \over n+2} (\R^n ) } .
$$
Consider now the Kelvin's transform of $\phi $, $\hat \phi (y) = |y|^{2-n} \phi (|y|^{-2} y )$. This function satisfies
\begin{equation}
\label{gino01}
\Delta \hat \phi + p U^{p-1} \hat \phi + |y|^{-4} a (|y|^{-2} y ) \hat \phi = \hat h \quad {\mbox {in}} \quad \R^n \setminus \{ 0 \}
\end{equation}
where $\hat h (y ) = |y|^{-n-2} h(|y|^{-2} y )$.  We observe that
$$
\| \hat h \|_{L^q (|y|<2 )} = \| \,  |y|^{n+2 -{2n \over q} }\,  h \|_{L^q (|y| > {1\over 2} } \leq  C \| h \|_{L^{2n \over n+2} (\R^n ) }  ,
$$
$$
\| \,  |y|^{-4} \, a (|y|^{-2} y ) \|_{L^{n \over 2} (|y|<2 )} = \| a \|_{L^{n\over 2} ( |y|>{1\over 2} )}
$$
and
$$
\| \nabla \hat  \phi \|_{L^2 (\R^n )} + \| \hat \phi \|_{L^{2n \over n-2} (\R^n )}  \leq  C \| h \|_{L^{2n \over n+2} (\R^n ) } .
$$
Applying then elliptic estimates to \equ{gino01}, we get
$$
\| D^2\hat  \phi \|_{L^q (B_1) } + \| D \hat \phi \|_{L^q (B_1) } + \|\hat  \phi \|_{L^\infty (B_1) } \leq C  C \| h \|_{L^{2n \over n+2} (\R^n ) } .
$$
This concludes the proof of the proposition since $  \|\hat  \phi \|_{L^\infty (B_1) } = \| \phi \|_{L^\infty (\R^n \setminus B_1) }$.
\end{proof}

\medskip
\noindent
We have now the tools to give the

\noindent
{\bf Proof of \equ{fattouno}.} We start with the estimate on $\varphi_0^\perp$. We write
$$
\varphi_0^\perp = \sum_{\alpha =0}^n c_{\alpha 0} \varphi_{\alpha 0}^\perp
$$
where
$$
L(\varphi_{\alpha 0}^\perp ) = - L(Z_{\alpha 0}).
$$
We write the above equation in the following way
$$
\Delta   (\varphi_{\alpha 0}^\perp  ) + p \gamma U^{p-1}   (\varphi_{\alpha 0}^\perp  ) +\underbrace{ p (|u|^{p-1} - U^{p-1} )}_{:= a_0 (y) } \varphi_{\alpha 0}^\perp  =  - L(Z_{\alpha 0}).
$$
Observe that
$$
|y|^{-n-2} \, L(Z_{0 , 0 } ) (\, |y|^{-2} \, y ) = -   L(Z_{0 , 0 } ) (y),
$$
while
$$ |y|^{-n-2} \, L(Z_{\alpha , 0 } ) (\, |y|^{-2} \, y ) =  L(Z_{\alpha , 0 } ) (y) \quad \alpha = 1, \ldots , n.
$$
We claim that $a_0 \in L^{n\over 2} (\R^n )$,
\begin{equation}
\label{gino1}
\| a_0 \|_{L^{n \over 2} (\R^n ) } \leq C k^{2\over n}, \quad {\mbox {and}} \quad \| L(Z_{\alpha 0 } ) \|_{L^{2n \over n+2} (\R^n )} \leq C \mu^{n-1 \over n},
\end{equation}
where we take into account that $\| h \|_{**} \leq C \| h \|_{L^{2n \over n+2} (\R^n )} $.
Let $\eta >0$ be a fixed positive number, independent of $k$. We split the integral all over $\R^n$ into a first integral over $\R^n \setminus \bigcup_{j=1}^k B(\xi_j , {\eta \over k} )$ and a second integral over $ \bigcup_{j=1}^k B(\xi_j , {\eta \over k} )$. We write then
$$
\| a_0 \|_{L^{n \over 2} (\R^n ) }^{n \over 2} = \int_{\R^n \setminus \bigcup_{j=1}^k B(\xi_j , {\eta \over k} )} |a_0 (y)|^{n \over 2} \, dy +\sum_{j=1}^k  \int_{ B(\xi_j , {\eta \over k} )} |a_0 (y)|^{n\over 2} \, dy
$$
\begin{equation}
\label{viernes1}
= i_1 + i_2.
\end{equation}
In the region $\R^n \setminus \bigcup_{j=1}^k B(\xi_j , {\eta \over k} )$, we have that
$$
|a_0 (y) | = p \left| |u|^{p-1} - U^{p-1} \right| \leq C U^{p-2} \, \sum_{j=1}^k {\mu^{n-2 \over 2} \over |y-\xi_j |^{n-2} },
$$
for some positive convenient constant $C$.
Thus
$$
\int_{\R^n \setminus \bigcup_{j=1}^k B(\xi_j , {\eta \over k} )} |a_0 (y) |^{n \over 2} \, dy \leq C
\mu^{{n-2 \over 2} \, {n\over 2}} \sum_{j=1}^k \int_{\R^n \setminus \bigcup_{j=1}^k B(\xi_j , {\eta \over k} )}
U^{(p-2) {n\over 2}} {1\over |y-\xi_j |^{(n-2) {n\over 2}}} \, dy
$$
$$
 \leq C \, k \,
\mu^{{n-2 \over 2} \, {n\over 2}}  \int_{1\over k}^1 {t^{n-1} \over t^{(n-2) {n\over 2} }} \, dt \leq C \, k \, \mu^{{n-2 \over 2} \, {n\over 2}} k^{{n\over 2} (n-2) -n}.
$$
We conclude that
\begin{equation}
\label{viernes2}
\int_{\R^n \setminus \bigcup_{j=1}^k B(\xi_j , {\eta \over k} )} |a_0 (y) |^{n \over 2} \, dy \leq C \mu^{n-1 \over 2}
\end{equation}
Let us now fix $j \in \{1 , \ldots , k\}$ and consider $y \in B(\xi_j , {\eta \over k} )$. In this region we have
$$
|a_0 (y) | \leq C |U_j |^{p-1},
$$
for some proper positive constant $C$. Recalling that $U_j (y) = \mu^{-{n-2\over 2}} U({y-\xi_j \over \mu } )$, we easily get
$$
\int_{B(\xi_j , {\eta \over k} ) } |a_0 (y)|^{n\over 2} \, dy \leq C
$$
and thus
\begin{equation}
\label{viernes3}
\sum_{j=1}^k \int_{B(\xi_j , {\eta \over k} ) } |a_0 (y)|^{n\over 2} \, dy \leq C k.
\end{equation}
We conclude then that $a_0 \in L^{n \over 2} (\R^n )$, and from \equ{viernes1}, \equ{viernes2} and \equ{viernes3} we conclude the first estimate
in \equ{gino1}.

\medskip
\noindent
We prove the second estimate in \equ{gino1} for $\alpha=0$. Analogous computations give the estimate for $\alpha \not= 0$. We write
\begin{equation}
\label{viernes4}
\int_{\R^n} |L (Z_{00}) |^{2n \over n+2} \, dy =\int_{\R^n \setminus \bigcup_{j=1}^k  B(\xi_j , {\eta \over k} ) } + \sum_{j=1}^k
\int_{B(\xi_j , {\eta \over k} )} = i_1 + i_2
\end{equation}
Since $L(Z_{00} ) = p (|u|^{p-1} - U^{p-1} ) \, Z_{00} = a_0 (y) Z_{00}$, a direct application of Holder inequality gives
$$
|i_1| \leq C \left(  \int_{ \R^n \setminus \bigcup_{j=1}^k  B(\xi_j , {\eta \over k}) } |a_0 (y) |^{n\over 2} \right)^{4\over n+2}
 \left(  \int_{ \R^n \setminus \bigcup_{j=1}^k  B(\xi_j , {\eta \over k} )} |Z_{00}  (y) |^{2n\over n-2} \right)^{n-2\over n+2}
$$
Taking into account that $  \left(  \int_{ \R^n \setminus \bigcup_{j=1}^k  B(\xi_j , {\eta \over k} )} |Z_{00}  (y) |^{2n\over n-2} \right)
\leq  \left(  \int_{ \R^n } |Z_{00}  (y) |^{2n\over n-2} \right)$ and the validity of \equ{viernes2}, we get
\begin{equation}
\label{viernes5}
|i_1| \leq C \mu^{2 {n-1 \over n+2}}.
\end{equation}
Let us fix now $j \in \{ 1 , \ldots , k \}$.
Using now that
$$
\left| \int_{B(\xi_j , {\eta \over k} )} |L(Z_{00})  |^{2n \over n+2} \right| \leq C \left(  \int_{ B(\xi_j , {\eta \over k}) } |a_0 (y) |^{n\over 2} \right)^{4\over n+2}
 \left(  \int_{ B(\xi_j , {\eta \over k} )} |Z_{00}  (y) |^{2n\over n-2} \right)^{n-2\over n+2}
$$
together with the fact that
$$
 \int_{ B(\xi_j , {\eta \over k} )} |Z_{00}  (y) |^{2n\over n-2} \leq C k^{-n},
$$
we conclude that
\begin{equation}
\label{viernes6}
|i_2|\leq C \mu^{{n\over 2} {n-2 \over n+2} -{1\over 2} }
\end{equation}
>From \equ{viernes4}, \equ{viernes5} and \equ{viernes6} we conclude that
$$
\| L(Z_{00} ) \|_{L^{2n \over n+2} (\R^n )} \leq C \mu^{n-1 \over n}
$$
thus completing the proof of \equ{gino1}.

\medskip
\noindent

Let us now fix $l \in \{ 1, \ldots , k \}$. Say $l=1$. We write
$$
\varphi_1^\perp = \sum_{\alpha =0}^n c_{\alpha 1} \varphi_{\alpha 1}^\perp
$$
where
$$
L(\varphi_{\alpha1 }^\perp ) = - L(Z_{\alpha 1} ).
$$
After the change of variable $\tilde \varphi_{\alpha  1}^\perp (y) = \mu^{n-2 \over 2} \varphi_{\alpha 1}^\perp (\mu y + \xi_1 )$, the above equation gets rewritten as
$$
\Delta  (\tilde \varphi_{\alpha  1}^\perp )  + p U^{p-1}  (\tilde \varphi_{\alpha  1}^\perp ) + \underbrace{ p [ (\mu^{-{n-2 \over 2}} |u| (\mu y + \xi_1 ) )^{p-1} - U^{p-1} ] }_{:= a_1 (y) } \tilde \varphi_{\alpha  1}^\perp = h(y)
$$
where
$$
h(y) = - \mu^{n+2 \over 2} L(Z_{\alpha 1} ) (\mu y + \xi_1 ).
$$
We claim that $a_1 \in L^{n \over 2} (\R^n )$.
\begin{equation}
\label{gino2}
\| a_1 \|_{L^{n \over 2} (\R^n ) } \leq C \mu , \quad {\mbox {and}} \quad \|  h  \|_{L^{2n \over n+2} (\R^n )} \leq C \mu.
\end{equation}
We leave the details to the reader.
The proof of \equ{fattouno} follows by \equ{gino1}, \equ{gino2} and a direct application of Proposition \ref{gino}.

\section{Appendix} \label{appe1}
In this section we perform the computations of the entrances of the matrices $A$, $F$, $G$, $B$, $C$, $D$ and $H_\alpha$, $\alpha =3, \ldots , n$. We start with proving some usefull expansions and a formula.

\medskip
\noindent
{\bf Some usefull expansions.}\\
 Let $\eta >0$ and $\sigma >0$ be small and fixed numbers, independent of $k$.
Assume that  $y \in B(0,  {\eta \over \mu k^{1+\sigma}}  )$. We will provide usefull expansions of some functions in this region.

We start with the function, for  $y \in B(0,  {\eta \over \mu k^{1+\sigma}}  )$,
\begin{equation}
\label{Upsilon}
\Upsilon (y) :=  \mu^{n-2 \over 2} \,  U(\xi_1 + \mu y ) - \sum_{l>1} \, U(y + \mu^{-1} (\xi_1 - \xi_l ) ) .
\end{equation}
We have the validity of the following expansion
\begin{eqnarray}\label{rivoli1}
 \Upsilon (y)
&=&
-{n-2 \over 2} \, \mu^{n\over 2}  \left[ y_1 - \mu^{n-2 \over 2}  \sum_{l >1}^k {1  \over (1- \cos \theta_l )^{n-2 \over 2}}  (y_1 - {\sin \theta_l \over 1-\cos \theta_l } y_2 )  \right] \, \times \nonumber \\
& &  (1+ \mu^2 O(|y|) ) \nonumber \\
& & \nonumber \\
&+& {n-2 \over 4} \, \mu^{n+2 \over 2} \Biggl[    {n y_1^2 \over 2} - |y|^2  -  \mu^{n-2 \over 2} \sum_{l >1}
 {1 \over (1-\cos \theta_l )^{n \over 2}  }  \times  \nonumber \\
& &  \left(-1-|y|^2 + { n \over 2} (1-\cos \theta_l ) y_1^2 +  { n \over 2} (1+\cos \theta_l ) y_2^2 + n \sin \theta_l  y_1 y_2
 \right) \Biggl] \times \nonumber\\
& &  (1+ \mu^2 O(|y|^2 )) \nonumber\\
& & \nonumber \\
&+& \mu^{n+4 \over 2} O(1+ |y|^3 )   + O( \mu^{n+2 \over 2} )
\end{eqnarray}
for a fixed constant $A$.
Formula \equ{rivoli1} is a direct application of the fact that
$$
\mu^{n-2 \over 2} \left( 2 \over 1+ |\xi_1 |^2  \right)^{n-2 \over 2} - \mu^{n-2} \sum_{l >1}^k {1  \over (1- \cos \theta_l )^{n-2 \over 2}} = O( \mu^{n+2 \over 2})
$$
and
of  Taylor expansion applied separatedly to $ \mu^{n-2 \over 2} \,  U(\xi_1 + \mu y )$ and $ \sum_{l>1}^k \, U(y + \mu^{-1} (\xi_1 - \xi_l ) )$ in the considered region   $y \in B(0,  {\eta \over \mu k^{1+\sigma}}  )$. Indeed, we have
\begin{eqnarray} \label{man1}
U(\xi_1 + \mu y ) &=& \mu^{n-2 \over 2} \left( 2 \over (1+ |\xi_1 |^2 ) \right)^{n-2 \over 2}  \Biggl[   1 -  { (n-2) \over 2} \,  y_1 \,  \mu + {n-2 \over 4} \left( {n (y\cdot \xi_1 )^2 \over 2} - |y|^2 \right) \mu^2  \nonumber \\
& & \nonumber \\
 &+&   \mu^3 O(|y|^3 ) \Biggl] \, (1+ O(\mu^2 ) )
\end{eqnarray}
and
\begin{eqnarray} \label{man2}
U(y + \mu^{-1} (\xi_1 - \xi_l ) ) &= & {\mu^{n-2} \over (1- \cos \theta_l )^{n-2 \over 2}} \left[ 1- {(n-2) \over 2} \, {(\xi_1 - \xi_l ) \cdot y  \over (1-\cos \theta_l )} \,  \mu \right. \nonumber \\
&+& {n-2 \over 4} {\mu^2 \over (1-\cos \theta_l ) } \left(-1-|y|^2 + n\,  \left( {(\xi_1 - \xi_l ) \cdot y \over |\xi_1 -\xi_l | } \right)^2 \, \right)
\\
&+&  \left. \mu^3 {O(1+ |y|^3 ) \over |\xi_1 - \xi_l |^3 }    \right]. \nonumber
\end{eqnarray}
Recall now the definition of the functions $Z_\alpha$, $\alpha =0 , \ldots , n$  in \equ{defzetapiccole}.
In the region $y \in B(0,  {\eta \over \mu k^{1+\sigma}}  )$,
we need to describe the functions
$$
Z_\alpha (y + \mu^{-1} (\xi_l - \xi_1 ) ), \quad \alpha = 0 , 1, \ldots , n.
$$
A direct application of Taylor expansion
 gives
\begin{eqnarray} \label{man3}
Z_0 (y + \mu^{-1} (\xi_l - \xi_1 ) ) &=& - {n-2 \over 2} {\mu^{n-2} \over (1-\cos \theta_l )^{n-2 \over 2} }
\Biggl[ 1- (n-2) {(\xi_l - \xi_1 ) \cdot y \over |\xi_l - \xi_1 |^2 } \mu  \nonumber \\
&+&  {\mu^2 \over |\xi_l - \xi_1|^2} O(1+ |y|^2 ) \Biggl],
\end{eqnarray}
\begin{eqnarray} \label{man4}
Z_1 (y + \mu^{-1} (\xi_l - \xi_1 ) ) &=& - {n-2 \over 2} {\mu^n \over (1-\cos \theta_l )^{n \over 2} }
\Biggl[ \mu^{-1} (\cos \theta_l - 1)  + [ 1- {n \over 2} (1- \cos \theta_l ) ]  \, y_1  \nonumber \\
 &-&   {n \over 2} \sin \theta_l \, y_2  + \mu O(1+ |y| )  \Biggl]
\end{eqnarray}
\begin{eqnarray} \label{man4}
Z_2 (y + \mu^{-1} (\xi_l - \xi_1 ) ) &=& - {n-2 \over 2} {\mu^n \over (1-\cos \theta_l )^{n \over 2} }
\Biggl[ \mu^{-1} \sin \theta_l  + [ 1- {n \over 2} (1+ \cos \theta_l ) ]\, y_2  \nonumber \\
 &+&   {n \over 2} \sin \theta_l \, y_1  + \mu O(1+ |y| )  \Biggl]
\end{eqnarray}
and for $\alpha = 3, \ldots , n$
\begin{equation} \label{man6}
Z_\alpha (y + \mu^{-1} (\xi_l - \xi_1 ) ) =- {n-2 \over 2} \, {\mu^n \over (1-\cos \theta_l )^{n \over 2} }\,  y_\alpha \,  (1+ \mu^2 O(1+|y| )).
\end{equation}

\medskip
\medskip

\medskip
\medskip
\medskip
We have now the tools to give the proofs of \equ{A11}, \equ{A1l}, \equ{F11}, \equ{F1l}, \equ{G11} , \equ{G1l}, \equ{B11}, \equ{B1l}, \equ{C11}, \equ{C1l}, \equ{D11} and \equ{D1l}.

\medskip
\noindent
{\bf Computation of $A_{11}$.}
Let $\eta >0$ and $\sigma >0$ be small and fixed numbers. We write
\begin{eqnarray*}
A_{11} &=& \int_{\R^n }  (f' (u) - f' (U_1 ) ) Z_{01}^2 \\
&=&[ \int_{B(\xi_1 , {\eta \over k^{1+ \sigma}}) } + \int_{\R^n \setminus  B(\xi_1 , {\eta \over k^{1+ \sigma}}) }]  (f' (u) - f' (U_1) ) Z_{01}^2
\\
&=& I_1+ I_2
\end{eqnarray*}
We claim that the main part of the above expansion is $I_1$. Note that very close to $\xi_1$, $U_1 (x ) = O(\mu^{-{n-2 \over 2} })$. More in general, taking $\eta$ small if necessary, we have that $U_1$ dominates globally the other terms. We thus have

\begin{eqnarray*}
I_1 &=&   \int_{B(\xi_1 , {\eta \over k^{1+ \sigma}}) } f^{''} (U_1) [ U(x) - \sum_{l>1} U_l (x) + \tilde \phi (x) ] Z_{01}^2 (x) \, dx +  O (  k^{n-2} \mu^{n-1} )
\\
&(&x= \xi_1 + \mu y )
\\
&= &   \int_{B(0, {\eta \over \mu k^{1+ \sigma}}) } f^{''} (U)  \left[ \Upsilon(y)  \right]  Z_0^2  \, dx  \\
&+&    \int_{B(0, {\eta \over \mu k^{1+ \sigma}}) } f^{''} (U) \left[  \mu^{n-2 \over 2} \tilde \phi (\mu y + \xi_1 ) ]  \right]  Z_0^2  \, dy  +  O (  k^{n-2} \mu^{n-1} )   \\
&= &    \int_{B(0, {\eta \over \mu k^{1+ \sigma}}) } f^{''} (U)  \left[ \Upsilon (y)  \right]  Z_0^2    \\
&+&    p (p-1) \gamma  \int_{B(0, {\eta \over \mu k^{1+ \sigma}}) }  U^{p-2}   \phi_1  ( y )  Z_0^2 \, dx    +  O (  k^{n-2} \mu^{n-1} )
\end{eqnarray*}
where $ \Upsilon (y) $ is defined in \equ{Upsilon} and $\phi_1 (y) = \mu^{n-2 \over 2} \tilde \phi (\mu y + \xi_1 )$.
Using \equ{parameters}, expansion \equ{man1} and \equ{man2}, we get
$$
I_1 = O( k^{n-2} \mu^{n-1} ).
$$
On the other hand, we have that
\begin{equation}
\label{rivoli4}
I_2= O (  k^{n-2} \mu^{n-1} )
\end{equation}
Indeed, we first write
$$
I_2 = [ \sum_{j>1} \int_{B(\xi_j , {\eta \over k^{1+ \sigma}}) } + \int_{\R^n \setminus \bigcup_{j\geq 1}  B(\xi_j , {\eta \over k^{1+ \sigma}}) }]  (f' (u) - f' (U_1) ) Z_{01}^2
$$
Fix now $j >1$. In the ball $B(\xi_j  , {\eta \over k^{1+\sigma} })$, $u \sim U_j = O(\mu^{-{n-2 \over 2} })$ and $U_j$ dominates all the other terms. Taking this into consideration, we have that
\begin{eqnarray*}
\left| \int_{B(\xi_j , {\eta \over k^{1+\sigma} } )} [ f' (u) - f'(U_1 ) ] Z_{01}^2 \right| &\leq & \int_{B(\xi_j , {\eta \over k^{1+\sigma}} )} f'(U_j ) Z_{01}^2 \\
&\leq & C   \int_{B(0, {\eta \over \mu k^{1+\sigma}} )} {1\over (1+|y|^2 )^2 }  Z_{0}^2 (y + \mu^{-1} (\xi_j - \xi_1 ) ) \, dy \\
& & ({\mbox {using}} \equ{man3}) \\
&\leq & C { \mu^{2 (n-2) } \over (1- \cos \theta_j )^{n-2}}   \int_{B(0, {\eta \over \mu k^{1+\sigma}} )} {1\over (1+|y|^2 )^2 } \, dy \\
&\leq & C  { \mu^{2 (n-2) } \over (1- \cos \theta_j )^{n-2}} {1\over (\mu k^{1+\sigma} )^{n-4} }
\end{eqnarray*}
where $C$ is an appropiate positive constant independent of $k$.
Thus we conclude that
\begin{equation}
\label{rivoli2}
\left| \sum_{j>1}  \int_{B(\xi_j , {\eta \over k^{1+\sigma} } )} [ f' (u) - f'(U_1 ) ] Z_{01}^2 \right| \leq C \mu^{n-1} k^{n-2},
\end{equation}
where again  $C$ is an appropiate positive constant independent of $k$.

On the other hand
\begin{eqnarray*}
\left|  \int_{\R^n \setminus \bigcup_{j\geq 1}  B(\xi_j , {\eta \over k^{1+ \sigma}}) } (f' (u) - f' (U_1) ) Z_{01}^2 \right| &\leq & C \mu^{-n+2}
 \int_{\R^n \setminus \bigcup_{j\geq 1}  B(\xi_j , {\eta \over k^{1+ \sigma}}) }  {1\over (1+ |x|^2 )^2}  Z_{0}^2 ({x-\xi_1 \over \mu} ) \\
&\leq & C  \mu^{n-2}
 \int_{\R^n \setminus \bigcup_{j\geq 1}  B(\xi_j , {\eta \over k^{1+ \sigma}}) }  {1\over (1+ |x|^2 )^2} {1\over |x-\xi_1|^{2(n-2)}} \, dx\\
&\leq & C \mu^{n-2} k^{(n-4) (1+\sigma )}
\end{eqnarray*}
Thus we conclude that
\begin{equation}
\label{rivoli3}
\left|  \int_{\R^n \setminus \bigcup_{j\geq 1}  B(\xi_j , {\eta \over k^{1+ \sigma}}) } (f' (u) - f' (U_1) ) Z_{01}^2 \right|  \leq C \mu^{n-1} k^{n-2}
\end{equation}

Formulas \equ{rivoli2} and \equ{rivoli3} imply \equ{rivoli4}.
Thus  we get \equ{A11}.

\medskip
\noindent
{\bf Computation of $A_{1l}$.} Let $l >1$ be fixed.
Let again $\eta >0$ and $\sigma >0$ be small and fixed numbers. In this case we write
\begin{eqnarray*}
A_{1l} &=& \int_{\R^n }  (f' (u) - f' (U_1 ) ) Z_{01} Z_{0l}  \\
&=&[ \int_{B(\xi_l , {\eta \over k^{1+ \sigma}}) } + \int_{\R^n \setminus  B(\xi_l , {\eta \over k^{1+ \sigma}}) }]  (f' (u) - f' (U_1) ) Z_{01} Z_{0l}
\\
&=& I_1+ I_2
\end{eqnarray*}
 We start with the expansion of $I_1$. Using again the fact that in $B(\xi_l , {\eta \over k^{1+\sigma} })$ the leading term in $u$ is $U_l$, which is of order $\mu^{-{n-2 \over 2}}$, and dominates all the other terms in the definition of $u$, we get that
\begin{eqnarray*}
I_1 &=&   \int_{B(\xi_l , {\eta \over k})}        [f'(u) - f'(U_1)] Z_{01} Z_{0l}  \, dx  \\
&=&- p \gamma \int_{B(\xi_l, {\eta \over k})} [\mu^{-{n-2 \over 2}} U({x-\xi_l \over \mu }) ]^{p-1} \, \mu^{-n+2 } Z_0 ({x-\xi_1 \over \mu }) Z_0 ({x-\xi_l \over \mu })  +R_1   \\
& & (x=\mu y + \xi_l ) \\
&=&- p \gamma  \int_{B(0,{\eta \over \mu k})} U^{p-1}  (y) \, Z_0 (y) \, Z_0 (y+ \mu^{-1} (\xi_l - \xi_1 )) \, dy  + R_1
\end{eqnarray*}
where
$R_1 = I_1 -  p \gamma \int_{B(\xi_l, {\eta \over k})} [\mu^{-{n-2 \over 2}} U({x-\xi_l \over \mu }) ]^{p-1} \, \mu^{-n+2 } Z_0 ({x-\xi_1 \over \mu }) Z_0 ({x-\xi_l \over \mu })  $.
Now using the expansion \equ{man3}, together with formula \equ{nonso},
we get, for any integer $l>1$
\begin{equation}
\label{rivoli77}
I_1 = - p \gamma {n-2 \over 2} (-\int_{\R^n}  U^{p-1} Z_0 \, dy ) \left[ {1 \over (1-\cos \theta_l )^{n -2\over 2}} \right] \mu^{n-2} + O(\mu^{n-1 } k^{n-2}) .
\end{equation}
Observe that
\begin{equation}
\label{nonso}
\int_{\R^n} U^{p-1} Z_0 \, dy = -{n-2 \over 2} (-\int_{\R^n } y_1 U^{p-1} Z_1 (y) \, dy )
\end{equation}
 Indeed,
\begin{eqnarray}\label{sat1}
\int_{\R^n} U^{p-1} Z_0 \, dy &=& {n-2 \over 2} \int U^p + U^{p-1} \nabla U \cdot y  \nonumber \\
&=& {n-2 \over 2} \int U^p + n \int U^{p-1}  y_1  Z_1 (y) \, dy
\end{eqnarray}
On the other hand, we have
$$
p \int U^{p-1} y_1 Z_1 (y) = - \int U^p
$$
We thus conclude \equ{nonso} from \equ{sat1}. Replacing \equ{nonso}  in \equ{rivoli77} we get
\begin{equation}
\label{rivoli7}
I_1 =  p \gamma ({n-2 \over 2})^2 (-\int_{\R^n}  U^{p-1} y_1 Z_1 \, dy ) \left[ {1 \over (1-\cos \theta_l )^{n -2\over 2}} \right] \mu^{n-2} + O(\mu^{n-1 } k^{n-2}) .
\end{equation}
On the other hand, a direct computation gives that
\begin{equation}
\label{rivoli8}
R_1 =  O(\mu^{n-1 } k^{n-2}) .
\end{equation}
We now estimate the term $I_2$.  We write
$$
I_2 =[\sum_{j\not= l}   \int_{B(\xi_j , {\eta \over k^{1+ \sigma}}) } +   \int_{\R^n \setminus \bigcup_{j} B(\xi_j , {\eta \over k^{1+ \sigma}}) }]  (f' (u) - f' (U_1) ) Z_{01} Z_{0l}
$$
Fix now $j \not= l$. In the ball $B(\xi_j  , {\eta \over k^{1+\sigma} })$, $u \sim U_j = O(\mu^{-{n-2 \over 2} })$ and $U_j$ dominates all the other terms. Taking this into consideration, we have that
\begin{eqnarray*}
&  & \left|    \int_{B(\xi_j , {\eta \over k^{1+\sigma} } )} [ f' (u) - f'(U_1 ) ] Z_{01} Z_{0l}  \right| \leq  \int_{B(\xi_j , {\eta \over k^{1+\sigma}} )} f'(U_j ) Z_{01} Z_{0l} \\
&\leq & C  \int_{B(0, {\eta \over \mu k^{1+\sigma}} )} {1\over (1+|y|^2 )^2 }  Z_{0} (y + \mu^{-1} (\xi_j - \xi_1 ) ) \,  Z_{0} (y + \mu^{-1} (\xi_j - \xi_l ) ) \, dy \\
& & ({\mbox {using}} \equ{man3}) \\
&\leq & C { \mu^{2 (n-2) } \over (1- \cos \theta_j )^{n-2}}   \int_{B(0, {\eta \over \mu k^{1+\sigma}} )} {1\over (1+|y|^2 )^2 } \, dy  \quad {\mbox {if}} \quad j\not= 1\\
& & {\mbox {while}} \\
&\leq & C { \mu^{2 (n-2) } \over (1- \cos \theta_j )^{n-2 \over 2}}    \quad {\mbox {if}} \quad j= 1\\
\end{eqnarray*}
Thus we conclude that
\begin{equation}
\label{rivoli5}
\left| \sum_{j\not= l}  \int_{B(\xi_j , {\eta \over k^{1+\sigma} } )} [ f' (u) - f'(U_1 ) ] Z_{01} Z_{0l} \right|  \leq C \mu^{n-1} k^{n-2},
\end{equation}
where again  $C$ is an appropiate positive constant independent of $k$.

On the other hand
\begin{eqnarray*}
& & \left|  \int_{\R^n \setminus \bigcup_{j\geq 1}  B(\xi_j , {\eta \over k^{1+ \sigma}}) } (f' (u) - f' (U_1) ) Z_{01} Z_{0l} \right|\\
& \leq & C \mu^{-n+2}
 \int_{\R^n \setminus \bigcup_{j\geq 1}  B(\xi_j , {\eta \over k^{1+ \sigma}}) }  {1\over (1+ |x|^2 )^2}  Z_{0} ({x-\xi_1 \over \mu} )  Z_{0} ({x-\xi_l \over \mu} ) \\
&\leq & C  \mu^{n-2}
 \int_{\R^n \setminus \bigcup_{j\geq 1}  B(\xi_j , {\eta \over k^{1+ \sigma}}) }  {1\over (1+ |x|^2 )^2} {1\over |x-\xi_1|^{(n-2)}}  {1\over |x-\xi_l|^{(n-2)}} \, dx
\end{eqnarray*}
Thus we conclude that
\begin{equation}
\label{rivoli6}
\left|  \int_{\R^n \setminus \bigcup_{j\geq 1}  B(\xi_j , {\eta \over k^{1+ \sigma}}) } (f' (u) - f' (U_1) ) Z_{01}^2 \right|  \leq C \mu^{n-1} k^{n-2}
\end{equation}

\medskip
Summing up the information in \equ{rivoli7}, \equ{rivoli8}, \equ{rivoli5} and \equ{rivoli6}, we conclude that
the validity of \equ{A1l}.

\medskip
\noindent
{\bf Computation of $F_{11}$.} Let $\eta >0$ and $\sigma >0$ be small and fixed numbers. We write
$$
F_{1 1} =       \int_{\R^n } [f' ( u ) - f' ( U_1 ) ] Z_{11}^2 \, dx
$$
$$
=\left[  \int_{B(\xi_1 , {\eta \over k^{1+\sigma}})  }+  \int_{\R^n \setminus B(\xi_1 , {\eta \over k^{1+\sigma}}) }\right]
  [f' ( u ) - f' ( U_1 ) ]   Z_{11}^2 \, dx = I_1 + I_2
$$
We claim that the main part of the above expansion is $I_1$. In $B(\xi_1 , {\eta \over k^{1+ \sigma}})$, the main part in $u$ is given by $U_1$, which is of size $\mu^{-{n-2 \over 2}}$ in this region, and which dominates all the other terms of $u$. Thus  we can perform a Taylor expansion of the function
$$ f' (u) - f' (U_1) = f^{''} (U_1 + s (u-U_1) ) [u-U_1] \quad {\mbox {for some}} \quad 0<s<1,
$$ so we write
$$
I_1 = \int_{B(\xi_1 , {\eta \over k^{1+\sigma}})  } f^{''}(  U_1 )  \left[ U(x) - \sum_{l> 1}^k U_l (x) + \tilde \phi (x) \right] Z_{11}^2 \, dx  +R_1 ,
$$
Performing the change of variables
$x=\xi_1 + \mu y$, and recalling that $Z_{11} (x) = \mu^{-{n\over 2}} Z_1 ({x-\xi_1 \over \mu }) (1+ O(\mu^2))$, we get
$$
I_1 - R_1 =\mu^{-2}
  \int_{B(0 , {\eta \over \mu k^{1+\sigma}})  } f^{''}(  U_1 )  \Upsilon (y)\,  Z_1^2 (y) \,  dy
$$
$$
+ \mu^{-2} \,  \int_{B(0 , {\eta \over \mu k^{1+\sigma}})  } f^{''}(  U_1 ) \mu^{n-2 \over 2} \tilde \phi (\xi_1 + \mu y ) \, Z_1^2 (y) \,  dy +  O(\mu^{n \over 2})
$$
where we recall that
$$
\Upsilon (y) =  \left[ \mu^{n-2 \over 2} U(\xi_1 + \mu y) - \sum_{l>1}^k U(y+ \mu^{-1} (\xi_1 - \xi_l ) )   \right] .
$$
Recall now that $\tilde \phi_1 (y) = \mu^{n-2 \over 2} \tilde \phi_1 (\mu y + \xi_1 ) $ solves the equation
$$
\Delta \phi_1 + f' ( U ) \phi_1  + \chi_1 (\xi_1 + \mu y ) \mu^{n+2 \over 2} E(\xi_1 + \mu y ) + \gamma \mu^{n+2 \over 2} {\mathcal N} (\phi_1 ) (\xi_1 + \mu y ) = 0 \quad {\mbox {in}} \quad \R^n
$$
Hence we observe that
$$
p (p-1) \gamma \int_{\R^n}U^{p-2}  \phi_1 Z_1^2 = p \gamma \int_{\R^n} {\partial \over \partial y_1} (U^{p-1} ) \phi_1 Z_1
$$
$$
= -p\gamma \int_{\R^n} U^{p-1} \phi_1 ( \partial_1 Z_1 ) - p \gamma \int_{\R^n} U^{p-1} (\partial_1 \phi_1 ) Z_1
$$
$$
=\int_{\R^n} \chi_1 (\xi_1 + \mu y ) \mu^{n+2 \over 2} E (\xi_1 + \mu y ) \partial_1 Z_1 \, dy + \gamma \mu^{n+2 \over 2} \int_{\R^n} {\mathcal N} (\phi_1 ) (\xi_1 + \mu y ) \partial_1 Z_1
$$
$$ + \underbrace{
\int_{\R^n} \left[ \Delta \phi_1 ( \partial_1 Z_1 ) + \Delta Z_1 (\partial_1 \phi_1 )\right]}_{=0}
$$
$$
=  \mu^{n+2 \over 2}  \int_{ B(0, {\eta \over \mu k^{1+ \sigma}})} E (\xi_1 + \mu y ) (\partial_1 Z_1 ) \, dy + \gamma \mu^{n+2 \over 2} \int_{\R^n} {\mathcal N} (\phi_1 ) (\xi_1 + \mu y )( \partial_1 Z_1 ) + O( \mu^{n\over 2} )
$$
Taking this into account, we first observe that
$$
I_1 - R_1= p \gamma \mu^{-2} \int_{B(0, {\eta \over \mu k^{1+ \sigma}} ) }
\,  \Upsilon (y) \,  \partial_1 (U^{p-1} Z_1 )  \, dy
 + O(\mu^{n\over 2})
$$
On the other hand
recall that
$$
R_1 = \int_{B(\xi_1 , {\eta \over k^{1+\sigma}})  } [ f^{''} (U_1 + s (u-U_1 ) ) - f^{''}(  U_1 )]  \left[ U(x) - \sum_{l> 1}^k U_l (x) + \tilde \phi (x) \right] Z_{11}^2 \, dx
$$
Thus we have
$$
| R_1 | \leq C  \int_{B(\xi_1 , {\eta \over k^{1+\sigma}})  } U_1^{p-2}  | (1 + s U_1^{-1}  (u-U_1 ) )^{p-2} - 1 | \left| U(x) - \sum_{l> 1}^k U_l (x) + \tilde \phi (x) \right| Z_{11}^2 \, dx
$$
$$
\leq C \mu^{n-2 \over 2}  \int_{B(\xi_1 , {\eta \over k^{1+\sigma}})  } U_1^{p-2}   \left| U(x) - \sum_{l> 1}^k U_l (x) + \tilde \phi (x) \right|^2 Z_{11}^2 \, dx
$$
Arguing as before, we get that
$$
R_1 = \mu^{n\over 2} O(1)
$$
where $O(1)$ is bounded as $k \to 0$.
Using the definition of $\mu$
and the expansions \equ{man1}, \equ{man2}  we conclude that
\begin{eqnarray}\label{rivoli10}
I_{1} &=& p\gamma \mu^{n-2 \over 2} {n-2 \over 4} \int_{\R^n} ( {n\over 2} y_1^2 - |y|^2 ) \partial_1 (U^{p-1} Z_1 )  \nonumber \\
&-& \mu^{n - 2} {n-2 \over 4} \sum_{l>1} {1\over (1-\cos \theta_l )^{n \over 2}} \int_{\R^n} [ -1 - |y|^2 +{n \over 2} (1-\cos \theta_l ) y_1^2 + {n \over 2} (1+ \cos \theta_l ) y_2^2 ]
\partial_1 (U^{p-1} Z_1 )  \nonumber  \\
&+& O(\mu^{n\over 2})  \nonumber  \\
&=& p \, \gamma \, {n-2 \over 4} \, \mu^{n-2 \over 2} \left[ (n-2) + \mu^{n-2 \over 2} \sum_{l>1}^k { n \cos \theta_l - (n-2) \over (1-\cos \theta_l )^{n \over 2}} \right] \, (-\int_{\R^n } y_1 U^{p-1} Z_1 )  \nonumber \\
&+& O ( \mu^{n\over 2})
\end{eqnarray}
On the other hand, we have that
\begin{equation}
\label{rivoli11}
I_2= \mu^{n\over 2} O(1)
\end{equation}
where $O(1)$ is bounded as $k \to 0$.
Indeed, we first write
$$
I_2 = [ \sum_{j>1} \int_{B(\xi_j , {\eta \over k^{1+ \sigma}}) } + \int_{\R^n \setminus \bigcup_{j\geq 1}  B(\xi_j , {\eta \over k^{1+ \sigma}}) }]  (f' (u) - f' (U_1) ) Z_{11}^2
$$
Fix now $j >1$. In the ball $B(\xi_j  , {\eta \over k^{1+\sigma} })$, $u \sim U_j = O(\mu^{-{n-2 \over 2} })$ and $U_j$ dominates all the other terms. Taking this into consideration, we have that
\begin{eqnarray*}
\left| \int_{B(\xi_j , {\eta \over k^{1+\sigma} } )} [ f' (u) - f'(U_1 ) ] Z_{11}^2 \right| &\leq & \int_{B(\xi_j , {\eta \over k^{1+\sigma}} )} f'(U_j ) Z_{11}^2 \\
&\leq & C  \mu^{-2} \int_{B(0, {\eta \over \mu k^{1+\sigma}} )} {1\over (1+|y|^2 )^2 }  Z_{1}^2 (y + \mu^{-1} (\xi_j - \xi_1 ) ) \, dy \\
& & ({\mbox {using}} \equ{man3}) \\
&\leq & C { \mu^{2 n - 2} \over (1- \cos \theta_j )^{n}}   \int_{B(0, {\eta \over \mu k^{1+\sigma}} )} {1\over (1+|y|^2 )^2 } \, dy \\
&\leq & C  { \mu^{2 n - 2} \over (1- \cos \theta_j )^{n}} {1\over (\mu k^{1+\sigma} )^{n-4} }
\end{eqnarray*}
where $C$ is an appropiate positive constant independent of $k$.
Thus we conclude that
\begin{equation}
\label{rivoli12}
\left| \sum_{j>1}  \int_{B(\xi_j , {\eta \over k^{1+\sigma} } )} [ f' (u) - f'(U_1 ) ] Z_{11}^2 \right|  \leq C \mu^{n \over 2},
\end{equation}
where again  $C$ is an appropiate positive constant independent of $k$.

On the other hand
\begin{eqnarray*}
& & \left|  \int_{\R^n \setminus \bigcup_{j\geq 1}  B(\xi_j , {\eta \over k^{1+ \sigma}}) } (f' (u) - f' (U_1) ) Z_{11}^2 \right| \leq  C \mu^{-n}
 \int_{\R^n \setminus \bigcup_{j\geq 1}  B(\xi_j , {\eta \over k^{1+ \sigma}}) }  {1\over (1+ |x|^2 )^2}  Z_{1}^2 ({x-\xi_1 \over \mu} ) \\
&\leq & C  \mu^{n-2}
 \int_{\R^n \setminus \bigcup_{j\geq 1}  B(\xi_j , {\eta \over k^{1+ \sigma}}) }  {1\over (1+ |x|^2 )^2} {1\over |x-\xi_1|^{2(n-1)}} \, dx\\
&\leq & C \mu^{n-2} k^{(n-2) (1+\sigma )}
\end{eqnarray*}
Thus we conclude that
\begin{equation}
\label{rivoli13}
\left|  \int_{\R^n \setminus \bigcup_{j\geq 1}  B(\xi_j , {\eta \over k^{1+ \sigma}}) } (f' (u) - f' (U_1) ) Z_{11}^2 \right|  \leq C \mu^{n \over 2}
\end{equation}

\medskip
>From  \equ{rivoli12} and \equ{rivoli13} we get \equ{rivoli11}. From \equ{rivoli10} and \equ{rivoli11} we conclude \equ{F11}.

\noindent
{\bf Computation of $F_{1l}$.}  Let $l >1$ be fixed.
Let again $\eta >0$ and $\sigma >0$ be small and fixed numbers. In this case we write
\begin{eqnarray*}
F_{1l} &=& \int_{\R^n }  (f' (u) - f' (U_1 ) ) Z_{11} Z_{1l}  \\
&=&[ \int_{B(\xi_l , {\eta \over k^{1+ \sigma}}) } + \int_{\R^n \setminus  B(\xi_l , {\eta \over k^{1+ \sigma}}) }]  (f' (u) - f' (U_1) ) Z_{11} Z_{1l}
\\
&=& I_1+ I_2
\end{eqnarray*}
 We start with the expansion of $I_1$. Recall that
$$
Z_{1l} (x) = \left[ \cos \theta_l \mu^{-{n\over 2}} Z_1 ({x-\xi_l \over \mu}) + \sin \theta_l \mu^{-{n\over 2}} Z_2 ({x-\xi_l \over \mu}) \right] \, \left( 1+ O(\mu^2) \right).
$$
Using again the fact that in $B(\xi_l , {\eta \over k^{1+\sigma} })$ the leading term in $u$ is $U_l$, which is of order $\mu^{-{n-2 \over 2}}$, and dominates all the other terms in the definition of $u$, we get that
\begin{eqnarray*}
I_1 &=&-
p \cos \theta_l  \int_{B(\xi_l, {\eta \over k})} [\mu^{-{n-2 \over 2}} U({x-\xi_l \over \mu }) ]^{p-1} \, \mu^{-n } Z_1 ({x-\xi_1 \over \mu }) Z_1 ({x-\xi_l \over \mu })    \\
 &-&
p \sin \theta_l  \int_{B(\xi_l, {\eta \over k})} [\mu^{-{n-2 \over 2}} U({x-\xi_l \over \mu }) ]^{p-1} \, \mu^{-n } Z_1 ({x-\xi_1 \over \mu }) Z_2 ({x-\xi_l \over \mu })  +R_1   \\
& & (x=\mu y + \xi_l ) \\
&=& - p \gamma \mu^{-2}   \cos \theta_l \int_{B(0,{\eta \over \mu k})} U^{p-1} Z_1 Z_1 (y+ \mu^{-1} (\xi_l - \xi_1 )) \, dy  \\
&-&  p \gamma \mu^{-2} \sin \theta_l  \int_{B(0,{\eta \over \mu k})} U^{p-1} Z_1 Z_2 (y+ \mu^{-1} (\xi_l - \xi_1 )) \, dy + R_1
\end{eqnarray*}
Now using the expansion \equ{man4}
we get, for any $l>1$
\begin{eqnarray}
\label{F1l}
I_1 -R_1& = & p \,  \gamma \, {n-2 \over 4} \, \Xi  \,  \cos \theta_l \, \left[ {n-2 - n \cos \theta_l \over (1-\cos \theta_l )^{n\over 2}} \right] \mu^{n-2} \nonumber \\
&-& p \,  \gamma \, {n-2 \over 4} \, \Xi  \,  \sin \theta_l \, \left[ {n \sin \theta_l \over (1-\cos \theta_l )^{n\over 2}} \right] \mu^{n-2}
 + O(\mu^{n\over 2} ) \nonumber \\
&=& p \,  \gamma \, {n-2 \over 2} \, \Xi  \, \left[  {{n-2 \over 2} \cos \theta_l  - {n \over 2}  \over (1-\cos \theta_l )^{n\over 2}} \right] \mu^{n-2} + O(\mu^{n \over 2} )
\end{eqnarray}
On the other hand we directly compute
$$
R_1 =  \mu^{n \over 2}O(1)
$$
where $O(1)$ is bounded as $k \to 0$.
We now estimate the term $I_2$.  We write
$$
I_2 =[\sum_{j\not= l}   \int_{B(\xi_j , {\eta \over k^{1+ \sigma}}) } +   \int_{\R^n \setminus \bigcup_{j} B(\xi_j , {\eta \over k^{1+ \sigma}}) }]  (f' (u) - f' (U_1) ) Z_{11} Z_{1l}
$$
Fix now $j \not= l$. In the ball $B(\xi_j  , {\eta \over k^{1+\sigma} })$, $u \sim U_j = O(\mu^{-{n-2 \over 2} })$ and $U_j$ dominates all the other terms. Taking this into consideration, we have that
\begin{eqnarray*}
&  & \left|    \int_{B(\xi_j , {\eta \over k^{1+\sigma} } )} [ f' (u) - f'(U_1 ) ] Z_{01} Z_{0l}  \right| \leq  \int_{B(\xi_j , {\eta \over k^{1+\sigma}} )} f'(U_j ) Z_{11} Z_{1l} \\
&\leq & C  \mu^{-2} \int_{B(0, {\eta \over \mu k^{1+\sigma}} )} {1\over (1+|y|^2 )^2 }  Z_{1} (y + \mu^{-1} (\xi_j - \xi_1 ) ) \,  Z_{1} (y + \mu^{-1} (\xi_j - \xi_l ) ) \, dy \\
& & ({\mbox {using}} \equ{man4}) \\
&\leq & C { \mu^{2 n - 2} \over (1- \cos \theta_j )^{n}}   \int_{B(0, {\eta \over \mu k^{1+\sigma}} )} {1\over (1+|y|^2 )^2 } \, dy  \quad {\mbox {if}} \quad j\not= 1\\
& & {\mbox {while}} \\
&\leq & C { \mu^{2 n - 2} \over (1- \cos \theta_j )^{n \over 2}}    \quad {\mbox {if}} \quad j= 1\\
\end{eqnarray*}
Thus we conclude that
\begin{equation}
\label{rivoli5}
\left| \sum_{j\not= l}  \int_{B(\xi_j , {\eta \over k^{1+\sigma} } )} [ f' (u) - f'(U_1 ) ] Z_{11} Z_{1l} \right| \leq  C \mu^{n \over 2},
\end{equation}
where again  $C$ is an appropiate positive constant independent of $k$.

On the other hand
\begin{eqnarray*}
& & \left|  \int_{\R^n \setminus \bigcup_{j\geq 1}  B(\xi_j , {\eta \over k^{1+ \sigma}}) } (f' (u) - f' (U_1) ) Z_{11} Z_{1l} \right|\\
& \leq & C \mu^{-n}
 \int_{\R^n \setminus \bigcup_{j\geq 1}  B(\xi_j , {\eta \over k^{1+ \sigma}}) }  {1\over (1+ |x|^2 )^2}  Z_{1} ({x-\xi_1 \over \mu} )  Z_{1} ({x-\xi_l \over \mu} ) \\
&\leq & C  \mu^{n-4}
 \int_{\R^n \setminus \bigcup_{j\geq 1}  B(\xi_j , {\eta \over k^{1+ \sigma}}) }  {1\over (1+ |x|^2 )^2} {1\over |x-\xi_1|^{(n-1)}}  {1\over |x-\xi_l|^{(n-1)}} \, dx
\end{eqnarray*}
Thus we conclude that
\begin{equation}
\label{rivoli6}
\left|  \int_{\R^n \setminus \bigcup_{j\geq 1}  B(\xi_j , {\eta \over k^{1+ \sigma}}) } (f' (u) - f' (U_1) ) Z_{01}^2 \right|  \leq C \mu^{n \over 2}
\end{equation}

\noindent
{\bf Computation of $G_{11}$.}
Let $\eta >0$ and $\sigma >0$ be small and fixed numbers. We write
\begin{eqnarray*}
G_{11} &=& \int_{\R^n }  (f' (u) - f' (U_1) ) Z_{21}^2 \\
&=&[ \int_{B(\xi_1 , {\eta \over k^{1+ \sigma}}) } + \int_{\R^n \setminus  B(\xi_1 , {\eta \over k^{1+ \sigma}}) }]  (f' (u) - f' (U_1) ) Z_{21}^2
\\
&=& I_1+ I_2
\end{eqnarray*}
Recall that $Z_{21}  (x) = \mu^{-{n\over 2}} Z_2 ({x-\xi_1 \over \mu})$.
We claim that the main part of the above expansion is $I_1$. Arguing as in the expansion of $F_{11}$, in the set $B(\xi_1 , {\eta \over k^{1+ \sigma}})$ we  perform a Taylor expansion of the function $ (f' (u) - f' (U_1) )$ so that
\begin{eqnarray*}
I_1 &=&   \int_{B(\xi_1 , {\eta \over k^{1+ \sigma}}) } f^{''} (U_1) [ U(x) - \sum_{l>1}^k U_l (x) + \tilde \phi (x) ] Z_{21}^2 (x) \, dx +R_1
\\
&(& {\mbox {changing variables }} x= \xi_1 + \mu y )
\\
&= &  \mu^{-2} \int_{B(0, {\eta \over \mu k^{1+ \sigma}}) } f^{''} (U) \Upsilon (y) \, Z_2^2   \\
&+&  \mu^{-2}  \int_{B(0, {\eta \over \mu k^{1+ \sigma}}) } f^{''} (U)  \mu^{n-2 \over 2} \tilde \phi (\mu y + \xi_1 )  Z_{2}^2  \, dx   + R_1 \\
&= &  \mu^{-2} \int_{B(0, {\eta \over \mu k^{1+ \sigma}}) } f^{''} (U) \Upsilon (y) Z_2^2    \\
&+&    p (p-1) \gamma   \mu^{-2} \int_{B(0, {\eta \over \mu k^{1+ \sigma}}) }  U^{p-2}   \phi_1  ( y )  Z_2^2 \, dx    + R_1
\end{eqnarray*}
where $\phi_1 (y) = \mu^{n-2 \over 2} \tilde \phi (\mu y + \xi_1 )$ and $\Upsilon (y) =  \left[ \mu^{n-2 \over 2} U( \xi_1 + \mu y ) - \sum_{l>1}^k U (y + \mu^{-1} (\xi_1 - \xi_l ) )   \right] $.

Using the equation satisfied by $\phi_1$ and by $Z_2$ in $\R^n$, we get
\begin{eqnarray*}
 & & p(p-1) \gamma \int U^{p-2} \phi_1 Z_2^2 = p \gamma \int {\partial \over \partial y_2} U^{p-1} \, \phi_1 Z_2
\\
&=& - p \gamma \int U^{p-1} \partial_{y_2 } \phi_1 Z_2 - p \gamma \int U^{p-1} \phi_1 \partial_{y_2 } Z_2  \\
&=& \int \zeta_1 (\xi_1 + \mu y ) \mu^{n+2 \over 2} E(\xi_1 + \mu y )
\partial_{y_2} \, Z_2  + \gamma \mu^{n+2 \over 2} \int  {\mathcal N} (\phi_1 ) (\xi_1 + \mu y ) \partial_{y_2} Z_2 \\
&=& p\gamma \int_{B(0,{\eta \over \mu k^{1+ \sigma}})} U^{p-1} \left[ \mu^{n-2 \over 2} U(\xi_1 + \mu y ) - \sum_{l>1} U(y + \mu^{-1} (\xi_1 - \xi_l )) \right] \partial_{y_2} Z_2 \\
&+& O(\mu^{n\over 2} )
\end{eqnarray*}
Thus we conclude that
$$
I_1= p \gamma \mu^{-2}  \,  \int_{B(0, {\eta \over \mu k^{1+ \sigma}}) }  \Upsilon (y) \, \partial_{y_2} \left( U^{p-1} Z_2 \right) \, dy
+  O(\mu^{n\over 2} )
$$
Using the definition of $\mu$ in \equ{parameters},  we see that the first order term in expansions  \equ{man1} and \equ{man2} gives a lower order contribution to $I_1$. Furthermore, by symmetry, also the second order term in the expansions \equ{man1} and \equ{man2} gives a small contribution. Thus, the third order term in the above mentioned expansions is the one that counts. We get indeed
\begin{eqnarray*}
I_1 &=& p \, \gamma  \, {n-2 \over 4}\,  \mu^{n-2 \over 2} \, \int [{n\over 2} y_1^2 - |y|^2 ] \partial_{y_2 } (U^{p-1} Z_2 )  \\
&-&   p \, \gamma \, {n-2 \over 4} \,   \mu^{n-2} \, \sum_{l>1}^k  {1\over (1-\cos \theta_l )^{n \over 2} } \int \Biggl[ -1 -|y|^2 \\
& + &{n \over 2} (1-\cos \theta_l ) y_1^2 +{n \over 2} (1+ \cos \theta_l ) y_2^2 \Biggl] \partial_{y_2} (U^{p-1} Z_2 ) +  O(\mu^{n\over 2} )
\\
&=& -  p \gamma {n-2 \over 4}   \mu^{n-2 \over 2}  \left[ 2 + \mu^{n-2 \over 2} \sum_{l>1}^k {-2 + n (1+ \cos \theta_l ) \over (1-\cos \theta_l )^{n \over 2}} \right] (-\int y_2 U^{p-1} Z_2 ) +  O(\mu^{n\over 2} )
\end{eqnarray*}
On the other hand, arguing as in the proof of estimate \equ{rivoli11}, we have that
$$
I_2= \mu^{n\over 2} O(1)
$$
where $O(1)$ is bounded as $k \to \infty$.
Thus we conclude \equ{G11}.

\medskip
\noindent
{\bf Computation of $G_{1l}$.} Let $l >1$ be fixed. Arguing as in the computation of $F_{1l}$, we first observe that
$$
G_{1l}  =
 \int_{B(\xi_l , {\eta \over k})}        [f'(u) - f'(U_1)] Z_{21} Z_{2l}  \, dy +   \,  O(\mu^{n\over 2} )
$$
Recall that
$$
Z_{2l} (x) =\left[  -\sin \theta_l \, \mu^{-{n\over 2}} Z_1 ({x-\xi_l \over \mu}) + \cos \theta_l \, \mu^{-{n\over 2}} \, Z_2 ({x-\xi_l \over \mu }) \right] \left( 1+ O(\mu^2) \right).
$$
In the ball $B(\xi_l , {\eta \over k})$, we expand as before in Taylor, and we get
\begin{eqnarray*}
G_{1l} &=&-
p\, \gamma \cos \theta_l  \int_{B(\xi_l, {\eta \over k})} [\mu^{-{n-2 \over 2}} U({x-\xi_l \over \mu }) ]^{p-1} \, \mu^{-n } Z_2 ({x-\xi_1 \over \mu }) Z_2 ({x-\xi_l \over \mu })    \\
&+& p\, \gamma \sin \theta_l  \int_{B(\xi_l, {\eta \over k})} [\mu^{-{n-2 \over 2}} U({x-\xi_l \over \mu }) ]^{p-1} \, \mu^{-n } Z_2 ({x-\xi_1 \over \mu }) Z_1 ({x-\xi_l \over \mu })    \\
&+& O(\mu^{n\over 2} ) \\
& & (x=\mu y + \xi_l ) \\
&=& - p \, \gamma \, \mu^{-2} \,  \sin \theta_l \, \int_{B(0,{\eta \over \mu k})} U^{p-1} Z_2 Z_2 (y+ \mu^{-1} (\xi_l - \xi_1 )) \, dy \\
&+&  p \, \gamma \, \mu^{-2} \,  \cos \theta_l \, \int_{B(0,{\eta \over \mu k})} U^{p-1} Z_2 Z_1 (y+ \mu^{-1} (\xi_l - \xi_1 )) \, dy + O(\mu^{n \over 2}).
\end{eqnarray*}
Now using the expansion \equ{man4}
we get, for any $l>1$, the validity of \equ{G1l}.

\medskip
\noindent
{\bf Computation of $B_{11}$.} Let $\eta >0$ and $\sigma >0$ be small and fixed numbers. We write
\begin{eqnarray*}
B_{11} &=& \int_{\R^n }  (f' (u) - f' (U_1) ) Z_{01} Z_{11} \\
&=& [ \int_{B(\xi_1 , {\eta \over k^{1+ \sigma}}) } + \int_{\R^n \setminus  B(\xi_1 , {\eta \over k^{1+ \sigma}}) }]  (f' (u) - f' (U_1) ) Z_{01} Z_{11}
\\
&=& I_1+ I_2
\end{eqnarray*}
We claim that the main part of the above expansion is $I_1$. We have

\begin{eqnarray*}
I_1 &=&   \int_{B(\xi_1 , {\eta \over k^{1+ \sigma}}) } f^{''} (U_1) [ U(x) - \sum_{l>1}^k U_l (x) + \tilde \phi (x) ] Z_{01} Z_{11}  \, dx \,   +  O(\mu^{n \over 2} )
\\
&(&x= \xi_1 + \mu y )
\\
&= &    \mu^{-2}  \int_{B(0, {\eta \over \mu k^{1+ \sigma}}) } f^{''} (U) \Upsilon (y) \, Z_0 Z_1  \, dy   \\
&+&    \mu^{-2} \int_{B(0, {\eta \over \mu k^{1+ \sigma}}) } f^{''} (U) \left[  \mu^{n-2 \over 2} \tilde \phi (\mu y + \xi_1 ) ] Z_{0} Z_1  \, dx  \right] \,  +  O(\mu^{n \over 2} )  \\
&= &  \mu^{-2} \int_{B(0, {\eta \over \mu k^{1+ \sigma}}) } f^{''} (U)  [ \mu^{n-2 \over 2} U( \xi_1 + \mu y ) - \sum_{l>1}^k U (y + \mu^{-1} (\xi_1 - \xi_l ) )  ] Z_0 Z_1 \, dy \\
&+&    p (p-1) \gamma   \mu^{-2} \left[  \int_{B(0, {\eta \over \mu k^{1+ \sigma}}) }  U^{p-2}   \phi_1  ( y )  Z_0 Z_1 \, dy \right] \,    +  O(\mu^{n \over 2} )
\end{eqnarray*}
where $\phi_1 (y) = \mu^{n-2 \over 2} \tilde \phi (\mu y + \xi_1 )$.

Using the equation satisfied by $\phi_1$ and by $Z_0$ , $Z_1$  in $\R^n$, we have that
\begin{eqnarray*}
p(p-1) \gamma & &  \int U^{p-2} \phi_1 Z_0 Z_1= p  \gamma \int {\partial \over \partial y_1} U^{p-1} \, \phi_1 Z_0
\\
&=& - p  \gamma\int U^{p-1} \partial_{y_1 } \phi_1 Z_0 - p  \gamma \int U^{p-1} \phi_1 \partial_{y_1 } Z_0  \\
&=& p  \gamma \int_{B(0,{\eta \over \mu k^{1+ \sigma}})} U^{p-1} \Upsilon (y) \partial_{y_1} (U^{p-1} Z_0 )
\end{eqnarray*}
where $\Upsilon (y) =  \left[ \mu^{n-2 \over 2} U(\xi_1 + \mu y ) - \sum_{l>1}^k U(y + \mu^{-1} (\xi_1 - \xi_l )) \right]$.
Using expansions \equ{man1} and \equ{man2}, and taking into account that $\partial_{y_1} \left( U^{p-1} Z_0 \right) = (p-1) U^{p-1} Z_0 Z_1 + U^{p-1} \partial_{y_1} Z_0$,
\begin{eqnarray*}
I_1 &=& p \gamma \mu^{-2}   \int_{B(0,{\eta \over \mu k^{1+ \sigma}})} \Upsilon (y)  \partial_{y_1} \left( U^{p-1}  Z_0 \right) \\
&=&   p \gamma {n-2 \over 2}  \Biggl[ -\mu^{n-4 \over 2}  \int y_1 \partial_{y_1} (U^{p-1} Z_0 ) \, dy \\
&+&  \mu^{n-3} \sum_{l>1}^k {1\over (1-\cos \theta_l )^{n -2 \over 2} } \int y_1 \partial_{y_1}  (U^{p-1} Z_0 ) \Biggl]  +  O(\mu^{n \over 2} )
\\
&=&   O(\mu^{n \over 2} ).
\end{eqnarray*}
On the other hand, arguing as in the expansion of $A_{11}$, one can easily prove that
$$
I_2 =    O(\mu^{n \over 2} ).
$$
Taking into account \equ{nonso}, we conclude \equ{B11}.

\noindent
{\bf Computation of $B_{1l}$.} Let $l >1$ be fixed. We have
\begin{eqnarray*}
B_{1l} &=&
 \int_{B(\xi_l , {\eta \over k})}        [f'(u) - f'(U_1)] Z_{01} Z_{1l}  \, dx \,  +  O(\mu^{n \over 2} )\\
&=&-
p \gamma \, \cos \theta_l  \int_{B(\xi_l, {\eta \over k})} [\mu^{-{n-2 \over 2}} U({x-\xi_l \over \mu }) ]^{p-1} \, \mu^{-n } Z_0 ({x-\xi_1 \over \mu }) Z_1 ({x-\xi_l \over \mu })    \\
&-&
p \gamma \, \sin \theta_l  \int_{B(\xi_l, {\eta \over k})} [\mu^{-{n-2 \over 2}} U({x-\xi_l \over \mu }) ]^{p-1} \, \mu^{-n } Z_0 ({x-\xi_1 \over \mu }) Z_2 ({x-\xi_l \over \mu })    \\
&+&    O(\mu^{n \over 2} ) \\
& & (x=\mu y + \xi_l ) \\
&=&-  p \gamma \mu^{-1} \cos \theta_l  \int_{B(0,{\eta \over \mu k})} U^{p-1} Z_1 Z_0 (y+ \mu^{-1} (\xi_l - \xi_1 )) \, dy
\\
&-&  p \gamma \mu^{-1} \sin \theta_l  \int_{B(0,{\eta \over \mu k})} U^{p-1} Z_2 Z_0 (y+ \mu^{-1} (\xi_l - \xi_1 )) \, dy
 +  O(\mu^{n \over 2} ).
\end{eqnarray*}
Now using the expansion \equ{man3}
we get, for any $l>1$, \equ{B1l}.

\medskip
\noindent
{\bf Computation of $C_{11}$.} Arguing as in the computation of $G_{11}$, we are led to
\begin{eqnarray*}
C_{11} &=& p\gamma \mu^{-1} \left[  \int_{B(0,{\eta \over \mu k})} \Upsilon (y) \, \partial_{y_2} (U^{p-1} Z_0 ) \, dy \right] \, (1+ O(\mu ) ) \nonumber \\
&+&  k^{n-2} \mu^{n-1} O(1) \nonumber \\
&=& - p \gamma \mu^{n} \left(  \sum_{l>1}^k {\sin \theta_l \over (1-\cos \theta_l )^{n \over 2}} \right) \int U^{p-1} Z_0  \, + k^{n-2} \mu^{n-1} O(1),
\end{eqnarray*}
where
$$
\Upsilon (y) =  \left[ \mu^{n-2 \over 2} U(\xi_1 + \mu y ) - \sum_{l>1}^k U(y + \mu^{-1} (\xi_1 - \xi_l )) \right]
$$
so that we conclude, by cancellation, the validity of \equ{C11}.

\medskip
\noindent
{\bf Computation of $C_{1l}$.} Let $l >1$ be fixed. We have
\begin{eqnarray*}
C_{1l} &=& \
 \int_{B(\xi_l , {\eta \over k})}        [f'(u) - f'(U_1)] Z_{01} Z_{2l} \, dx +  k^{n-2} \mu^{n-1} O(1)\\
&=&-
p\gamma \, \cos \theta_l \int_{B(\xi_l, {\eta \over k})} [\mu^{-{n-2 \over 2}} U({x-\xi_l \over \mu }) ]^{p-1} \, \mu^{-n + } Z_0 ({x-\xi_1 \over \mu }) Z_2 ({x-\xi_l \over \mu })   \, dx   \\
&+&
p\gamma \, \sin \theta_l \int_{B(\xi_l, {\eta \over k})} [\mu^{-{n-2 \over 2}} U({x-\xi_l \over \mu }) ]^{p-1} \, \mu^{-n +1} Z_0 ({x-\xi_1 \over \mu }) Z_1 ({x-\xi_l \over \mu })   \, dx   \\
&+&  k^{n-2} \mu^{n-1} O(1)\\
& & (x=\mu y + \xi_l ) \\
&=& - p \gamma \mu^{-1} \, \cos \theta_l \,  \int_{B(0,{\eta \over \mu k})} U^{p-1} Z_2 Z_0 (y+ \mu^{-1} (\xi_l - \xi_1 )) \, dy \\
&+&  p \gamma \mu^{-1}\, \sin \theta_l \,  \int_{B(0,{\eta \over \mu k})} U^{p-1} Z_2 Z_0 (y+ \mu^{-1} (\xi_l - \xi_1 )) \, dy +  k^{n-2} \mu^{n-1} O(1).
\end{eqnarray*}
Now using the expansion \equ{man3}
we get, for any $l>1$, \equ{C1l}.

\medskip
\noindent
{\bf Computation of $D_{11}$.}  Arguing as in the computation of $G_{11}$, we are led to
\begin{eqnarray*}
D_{11} &=& p\gamma \mu^{-2}  \int \Upsilon (y) \, \partial_{y_1} (U^{p-1} Z_2 ) \, dy  \nonumber \\
&+&  k^{n-1} \mu^{n} O(1) \nonumber \\
&=& p \gamma  {n-2 \over 4} n \mu^{n-2} \left(  \sum_{l>1}^k {\sin \theta_l \over (1-\cos \theta_l )^{n \over 2}} \right) \int y_2 U^{p-1} Z_2 +  k^{n-1} \mu^{n} O(1)
\end{eqnarray*}
so that we conclude \equ{D11}.

\medskip
\noindent
{\bf Computation of $D_{1l}$.}
Let $l >1$ be fixed. We have
\begin{eqnarray*}
D_{1l} &=&
 \int_{B(\xi_l , {\eta \over k})}        [f'(u) - f'(U_1)] Z_{11} Z_{2l} \, dx \, +  k^{n-1} \mu^{n} O(1) \\
&=&-
p \gamma  \, \cos \theta_l  \int_{B(\xi_l, {\eta \over k})} [\mu^{-{n-2 \over 2}} U({x-\xi_l \over \mu }) ]^{p-1} \, \mu^{-n } Z_1 ({x-\xi_1 \over \mu }) Z_2 ({x-\xi_l \over \mu })  \, dx \\
&+&
p \gamma  \, \sin \theta_l  \int_{B(\xi_l, {\eta \over k})} [\mu^{-{n-2 \over 2}} U({x-\xi_l \over \mu }) ]^{p-1} \, \mu^{-n } Z_1 ({x-\xi_1 \over \mu }) Z_1 ({x-\xi_l \over \mu })  \, dx \\
&+&  k^{n-1} \mu^{n} O(1)  \\
& & (x=\mu y + \xi_l ) \\
&=&- p \gamma \mu^{-2} \, \cos \theta_l  \int_{B(0,{\eta \over \mu k})} U^{p-1} Z_2 Z_1 (y+ \mu^{-1} (\xi_l - \xi_1 )) \, dy \\
&+& p \gamma \mu^{-2} \, \sin \theta_l  \int_{B(0,{\eta \over \mu k})} U^{p-1} Z_2 Z_2 (y+ \mu^{-1} (\xi_l - \xi_1 )) \, dy
+  k^{n-1} \mu^{n} O(1).
\end{eqnarray*}
Now using the expansion \equ{man4}
we get \equ{D1l}.

\noindent
{\bf Computation of $H_{3, 11}$.}
Let $\eta >0$ and $\sigma >0$ be small and fixed numbers. We write
\begin{eqnarray*}
H_{3, 11} &=& \int_{\R^n }  (f' (u) - f' (U_1) ) Z_{31}^2 \\
&=&[ \int_{B(\xi_1 , {\eta \over k^{1+ \sigma}}) } + \int_{\R^n \setminus  B(\xi_1 , {\eta \over k^{1+ \sigma}}) }]  (f' (u) - f' (U_1) ) Z_{31}^2
\\
&=& I_1+ I_2
\end{eqnarray*}
Arguing as before one can show that
$$
I_2 = O(\mu^{n\over 2} ).
$$
In $B(\xi_1 , {\eta \over k^{1+ \sigma}})$ we can perform a Taylor expansion of the function $ (f' (u) - f' (U_1) )$ so that
\begin{eqnarray*}
I_1 &=&   \int_{B(\xi_1 , {\eta \over k^{1+ \sigma}}) } f^{''} (U_1) [ U(x) - \sum_{l>1}^k U_l (x) + \tilde \phi (x) ] Z_{31}^2 (x) \, dx + O(\mu^{n\over 2} )
\\
&(&x= \xi_1 + \mu y )
\\
&= &  \mu^{-2} \int_{B(0, {\eta \over \mu k^{1+ \sigma}}) } f^{''} (U)  \Upsilon (y)  \, Z_3^2   \\
&+&  \mu^{-2}  \int_{B(0, {\eta \over \mu k^{1+ \sigma}}) } f^{''} (U)  \mu^{n-2 \over 2} \tilde \phi (\mu y + \xi_1 )  Z_{3}^2  \, dx   + O(\mu^{n\over 2} )  \\
&= &  \mu^{-2} \int_{B(0, {\eta \over \mu k^{1+ \sigma}}) } f^{''} (U)  \Upsilon (y)   Z_3^2    \\
&+&    p (p-1) \gamma   \mu^{-2} \int_{B(0, {\eta \over \mu k^{1+ \sigma}}) }  U^{p-2}   \phi_1  ( y )  Z_3^2 \, dx + O(\mu^{n\over 2} )
\end{eqnarray*}
where $\phi_1 (y) = \mu^{n-2 \over 2} \tilde \phi (\mu y + \xi_1 )$.
Using the equation satisfied by $\phi_1$ and by $z_2$ in $\R^n$, and arguing as in the previous steps, we get
$$
p(p-1) \gamma \int U^{p-2} \phi_1 Z_3^2
 = p\gamma \int_{B(0,{\eta \over \mu k^{1+ \sigma}})} U^{p-1}\Upsilon (y) \partial_{y_3} Z_3
$$
where we recall that
$$
\Upsilon (y) =  \left[ \mu^{n-2 \over 2} U(\xi_1 + \mu y ) - \sum_{l>1}^k U(y + \mu^{-1} (\xi_1 - \xi_l )) \right] .
$$
Thus we conclude that
$$
I_1= p \gamma \mu^{-2}  \,  \int_{B(0, {\eta \over \mu k^{1+ \sigma}}) } \Upsilon (y)   \, \partial_{y_3} \left( U^{p-1} Z_3 \right) \, dy +O(\mu^{n\over 2} ) .
$$
Using the definition of $\mu$ in \equ{parameters},  we see that the first order term in expansions  \equ{man1} and \equ{man2} gives a lower order contribution to $I_1$. Furthermore, by symmetry, also the second order term in the expansions \equ{man1} and \equ{man2} gives a small contribution. Thus, the third order term in the above mentioned expansions is the one that counts. We get indeed
\begin{eqnarray*}
I_1 &=& p \, \gamma  \, {n-2 \over 4}\,  \mu^{n-2 \over 2} \, \int [{n\over 2} y_1^2 - |y|^2 ] \partial_{y_3 } (U^{p-1} z_3 )    \\
&-&   p \, \gamma \, {n-2 \over 4} \,   \mu^{n-2} \, \sum_{l>1}^k {1\over (1-\cos \theta_l )^{n \over 2} } \int \Biggl[  -1 -|y|^2 +{n \over 2} (1-\cos \theta_l ) y_1^2 \\
&+&{n \over 2} (1+ \cos \theta_l ) y_2^2 \Biggl] \partial_{y_3} (U^{p-1} Z_3 ) + O(\mu^{n\over 2} )
\\
&=&  p \gamma {n-2 \over 2}   \mu^{n-2 \over 2}  \left[ 1 -\mu^{n-2 \over 2} \sum_{l>1}^k {1 \over (1-\cos \theta_l )^{n \over 2}} \right] (-\int y_3 U^{p-1} Z_3 ) + O(\mu^{n\over 2} )
\end{eqnarray*}
Thus we conclude \equ{H311}.

\medskip
\noindent
{\bf Computation of $H_{3, 1l}$.} Let $l >1$ be fixed. Arguing as before, we get
\begin{eqnarray*}
H_{3, 1l} &=&
 \int_{B(\xi_l , {\eta \over k})}        [f'(u) - f'(U_1)] Z_{31} Z_{3l} \, + O(\mu^{n\over 2} ) \\
&=&-
p\, \gamma \int_{B(\xi_l, {\eta \over k})} [\mu^{-{n-2 \over 2}} U({x-\xi_l \over \mu }) ]^{p-1} \, \mu^{-n } Z_3 ({x-\xi_1 \over \mu }) Z_3 ({x-\xi_l \over \mu })    \\
&+& O(\mu^{n\over 2} ) \\
& & (x=\mu y + \xi_l ) \\
&=& - p \, \gamma \, \mu^{-2} \,  \int_{B(0,{\eta \over \mu k})} U^{p-1} Z_3 (y) Z_3 (y+ \mu^{-1} (\xi_l - \xi_1 )) \, dy + O(\mu^{n\over 2} )
\end{eqnarray*}
Now using the expansion \equ{man6}
we get \equ{H31l}.


\begin{thebibliography}{AAA}

\bibitem{as} M. Abramowitz, I. A. Stegun, eds. Handbook of Mathematical Functions with Formulas, Graphs, and Mathematical Tables, (1972) Dover, New York.

            \bibitem{aubin}
T. Aubin,  Probl\`emes isop\'erimetriques et espaces de Sobolev, {\em J. Differ. Geometry} 11 (1976), 573--598

\bibitem{bahri}
 A. Bahri, S. Chanillo,  The difference of topology at infinity in changing-sign Yamabe problems on $S^3$ (the case of two masses). {\em Comm. Pure Appl. Math.} 54 (2001), no. 4, 450--478.

\bibitem{BX} A. Bahri, Y. Xu,  Recent progress in conformal geometry. ICP Advanced Texts in Mathematics, 1. Imperial College Press, London, 2007.







\bibitem{CGS} L.A. Caffarelli, B. Gidas, J. Spruck, Asymptotic symmetry and local behavior
of semilinear elliptic equations with critical Sobolev growth, {\em Comm. Pure
Appl. Math.} 42 (1989), 271-297.

\bibitem{D} W. Ding, On a conformally invariant elliptic equation on $R^n$, {\em Communications on Mathematical Physics }  107(1986), 331-335.



\bibitem{dmpp1}  M. del Pino, M. Musso, F. Pacard, A. Pistoia.  Large Energy Entire Solutions for the Yamabe Equation. {\em Journal of Differential Equations} 251,  (2011), no. 9,  2568--2597.

\bibitem{dmpp2}  M. del Pino, M. Musso, F. Pacard. A. Pistoia.  Torus action on $S^n$ and sign changing solutions for conformally invariant equations. {\em Annali della Scuola Normale Superiore di Pisa,  Cl. Sci.} (5) 12 (2013), no. 1, 209--237.

\bibitem{DKM} T. Duyckaerts, C. Kenig and F. Merle, Solutions of the focusing nonradial critical wave equation with the compactness property,  arxiv:1402.0365v1


\bibitem{Heb} E. Hebey, {\em Introduction \`a l'analyse non lin\'eaire sur les vari\'et\'ees}, Diderot \'editeur (1997).

\bibitem{HV} E. Hebey,  M. Vaugon,  Existence and multiplicity of nodal solutions for nonlinear elliptic equations with critical Sobolev growth.{\em  J. Funct. Anal.} 119 (1994), no. 2, 298--318.


    \bibitem{KM1} C. Kenig,  F.Merle, Global well-posedness, scattering and blow-up for the energy-critical focusing non-linear wave equation in the radial case, {\em Invent. Math.} 166(2006), 645-675.

        \bibitem{KM2} C. Kenig, F. Merle, Global well-posedness, scattering and blow-up for the energy-critical focusing nonlinear wave equation, {\em Acta Math.} 201(2008), 147-212.


\bibitem{KMPS} N. Korevaar, R. Mazzeo  F. Pacard and R. Schoen,  Refined asymptotics for
constant scalar curvature metrics with isolated singularities, {\em Inventiones
Math.} 135 (1999), pp. 233-272.



\bibitem{KS} I. Kra, S. R. Simanca,  On Circulant Matrices, {\em Notices. Amer. Math. Soc.} 59(2012), no.3,
368-377.

\bibitem{liwei} Y.-Y. Li, J. Wei, H. Xu,
Multibump solutions for $-\Delta u = K(x)u^{\frac{n+2}{n-2}}$ on lattices in $\R^n$,  preprint 2010.

\bibitem{MP}R. Mazzeo and F. Pacard, Constant  scalar curvature metrics with isolated
singularities, {\em Duke Mathematical Journal} 99 No. 3 (1999), pp. 353-418.


\bibitem{obata}
M. Obata, {\em Conformal changes of Riemannian metrics on a Euclidean sphere}. Differential geometry (in honor of Kentaro Yano), pp. 347--353. Kinokuniya, Tokyo, (1972).


\bibitem{pohozaev}
S. Pohozaev,  Eigenfunctions of the equation $\Delta u + \lambda f(u) = 0$, {\em Soviet. Math. Dokl.} 6, (1965), 1408--1411.

\bibitem{rey}
O. Rey,  The role of the Green's function in a nonlinear elliptic equation involving the critical Sobolev exponent. {\em J. Funct. Anal.} 89 (1990), no. 1, 1-–52.

\bibitem{RV1} F. Robert, J. Vetois,  Examples of non-isolated blow-up for perturbations of the scalar curvature equation on non locally conformally flat manifolds,  to appear in {\em J. of Differential Geometry}.


\bibitem{RV2} F. Robert, J. Vetois,   Sign-changing solutions to elliptic second order equations: glueing a peak to a degenerate critical manifold,  {\em arXiv:1401.6204}.



\bibitem{Sch-Yau} R. Schoen, S.T. Yau, {\em Lectures on differential geometry}. Conference Proceedings and Lecture Notes in Geometry and Topology, I. International Press, Cambridge, MA, 1994.

\bibitem{talenti}
G. Talenti,  Best constants in Sobolev inequality, {\em Annali di Matematica } 10 (1976), 353--372.

\bibitem{vaira}
G. Vaira, {A new kind of blowing-up solutions for the
Brezis-Nirenberg problem}, to appear in Calculus of Variations and PDEs.


\bibitem{wei} J. Wei, S. Yan,  Infinitly many solutions for the prescribed scalar curvature problem on $S^N$, {\em J. Funct. Anal. } 258 (2010), no. 9, 3048-3081.







\end{thebibliography}
\end{document}